\documentclass[12pt,a4paper,twoside,reqno]{amsart}
\title[Chern Correspondence for Higher Principal Bundles]
{Chern Correspondence for Higher Principal Bundles}

\date{\today}

\author[R. Tellez-Dominguez]{Roberto Tellez-Dominguez}
\address{Dep. Matem\'aticas \\ Universidad Aut\'onoma de Madrid\\ and Instituto de Ciencias Matem\'aticas (CSIC-UAM-UC3M-UCM)\\ Nicol\'as Cabrera 13--15, Cantoblanco\\ 28049 Madrid, Spain}
  \email{roberto.tellez@icmat.es}

\thanks{
Grants CEX2019-000904-S and SEV-2015-0554-19-2 funded by MCIN/AEI/10.13039/501100011033 and FSE invierte en tu futuro}

\usepackage{hyperref,url}
\usepackage{amssymb,latexsym}
\usepackage{amsmath,amsthm,mathtools,stmaryrd}
\usepackage[latin1]{inputenc}
\usepackage{enumerate}
\usepackage[all]{xy} \CompileMatrices
\SelectTips{cm}{12} 
\usepackage{tikz-cd}
\usepackage{adjustbox}
\tikzset{shorten <>/.style={shorten >=#1,shorten <=#1}}
\tikzset{2cell/.style={anchor=center,fill=white}}

\usepackage{verbatim} 
\usepackage{wrapfig}
\usepackage{graphicx}
\usepackage{slashed}
\usepackage{multicol}

\theoremstyle{plain}
\newtheorem{theorem}{Theorem}[section]
\newtheorem{lemma}[theorem]{Lemma}
\newtheorem{corollary}[theorem]{Corollary}
\newtheorem{proposition}[theorem]{Proposition}

\newtheorem*{theorem*}{Theorem}
\theoremstyle{definition}
\newtheorem{definition}[theorem]{Definition}
\newtheorem{definition-theorem}[theorem]{Definition-Theorem}
\newtheorem{example}[theorem]{Example}
\newtheorem*{acknowledgements}{Acknowledgements}
\theoremstyle{remark}
\newtheorem{remark}[theorem]{Remark}

\numberwithin{equation}{section} 

\renewcommand{\Im}{\operatorname{Im}}

\begin{document}

\begin{abstract}
The classical Chern correspondence states that a choice of Hermitian metric on a holomorphic vector bundle determines uniquely a unitary `Chern connection'. This basic principle in Hermitian geometry, later generalized to the theory of holomorphic principal bundles, provides one of the most fundamental ingredients in  modern gauge theory, via its applications to the Donaldson-Uhlenbeck-Yau Theorem. In this work we study a generalization of the Chern correspondence in the context of higher gauge theory, where the structure group of the bundle is categorified. For this, we define connective structures on a multiplicative gerbe and propose a natural notion of complexification for an important class of 2-groups. Using this, we put forward a new notion of higher connection which is well-suited for describing holomorphic principal 2-bundles for these 2-groups, and establish a Chern correspondence in this way. As an upshot of our construction, we unify two previous notions of higher connections in the literature, namely those of adjusted connections and of trivializations of Chern-Simons 2-gerbes with connection.
\end{abstract}

\maketitle

\setlength{\parskip}{5pt}
\setlength{\parindent}{0pt}


\section{Introduction}
\label{sec:intro}

The study of categorical generalizations of principal bundles has attracted attention from mathematical physicists over the last decades because of its potential to describe the geometry underlying some aspects of string theory and higher dimensional $\sigma$-models \cite{AschCanJurco,BaezYM,CarMultGer,Freed,JurcoSaeWolf,LanWenZhu,SatiSchSta2}. According to this line of thought, part of the background data defining such a theory is a higher principal bundle, an analog of a principal bundle in which the structure group is replaced by a suitable categorification, and the field content includes a connection on it. 

Higher Lie groups and higher principal bundles can also be defined in the holomorphic category. An interesting problem in this context is the construction of a moduli space of holomorphic higher principal bundles of a given topological type over a complex manifold via gauge theoretic methods. The interaction between complex geometry and gauge theory dates back to the work of Atiyah and Bott \cite{AtBott}, later expanded by the Narasimhan-Seshadri \cite{NarSesh} and Donaldson-Uhlenbeck-Yau \cite{Don,UhlYau} theorems. These show that the moduli of stable holomorphic vector bundles and the moduli of Hermite-Yang-Mills connections over a compact K\"ahler manifold are diffeomorphic. The first step on proving this is to define a map between both spaces, and this is done through the Chern correspondence, which states that holomorphic structures on the complexification of a Hermitian vector bundle $E_h$ are in bijection with unitary connections on $E_h$ whose curvature $F_h$ satisfies $F_h^{0,2} = 0$. 

In this article we initiate the programme of generalizing these classical results to the setting of higher principal bundles by proving an analog of the Chern correspondence for principal 2-bundles (Theorem \ref{th:chern}). While the notions of Lie 2-groups and principal 2-bundles are obtained in a relatively straightforward way by internalizing the axioms of groups and principal bundles in the 2-category of differentiable stacks, a consistent definition of connections on such principal 2-bundles is less obvious. For this reason we restrict ourselves to the study of principal 2-bundles whose structure 2-group $\mathfrak{G}$ arises as a central extension of the form
\begin{equation}
    1 \rightarrow BT \rightarrow \mathfrak{G} \rightarrow G \rightarrow 1,
\end{equation}
where $G$, $T$ are Lie groups with $T$ abelian. The most studied 2-group, String$(n)$, is of this form. These 2-groups admit a description in terms of multiplicative $T$-gerbes $\mathcal{G}$ over $G$ \cite{ScPr} which we recall in Section \ref{sec:backgr} and in Section \ref{sec:constrcomp} we use this language to define the new notion of `connective structure' (Definition \ref{def:constr}) on them, closely related to the connections from \cite{WaldMultCon}. A multiplicative $T$-gerbe over $G$ equipped with a connective structure $\mathcal{G}_{\nabla}$ gives in particular an $Ad$-invariant bilinear form $\langle\cdot,\cdot \rangle:\mathfrak{g} \otimes \mathfrak{g} \rightarrow \mathfrak{t}$ (where $\mathfrak{g}$ and $\mathfrak{t}$ are the Lie algebras of $G$ and $T$, respectively) and a curving $\Theta^L$ on $\mathcal{G}_{\nabla}$ (cf. Theorem \ref{th:MaurerCartan}). We use these to prove our first main theorem (see Theorem \ref{th:complexification}), which we regard as a complexification result for 2-groups that generalizes a construction in \cite{Upm}.

\begin{theorem}\label{th:complexificationintro}
    Let $K$, $T_{\mathbb R}$ be compact connected Lie groups with $T_{\mathbb R}$ abelian and let $j_T:T_{\mathbb R} \rightarrow T$, $j_K:K \rightarrow G$ be their complexifications. For $\mathcal{K}$ any $T_{\mathbb R}$-multiplicative gerbe over $K$ there is a unique holomorphic multiplicative $T$-gerbe with holomorphic connective structure $\mathcal{G}_{\nabla}$ over $G$ such that $j_K^{\ast}\mathcal{G} = \mathcal{K}^{\mathbb C}$ as smooth multiplicative $T$-gerbes over $K$. We call $\mathcal{G}_{\nabla}$ the \emph{complexification} of $\mathcal{K}$.
\end{theorem}

Theorem \ref{th:complexificationintro} allows us to complexify smooth $\mathcal{K}$-bundles to smooth $\mathcal{G}$-bundles (Proposition \ref{prop:complexification2pb}), and since $\mathcal{G}$ has a holomorphic connective structure we can define holomorphic structures with holomorphic connective structures on $\mathcal{G}$-bundles. After introducing a notion of `enhanced connection' on a $\mathcal{K}$-bundle (Definition \ref{def:2pbcon}), we obtain our second main result (see Theorem \ref{th:chern} and Corollary \ref{cor:chern} for a precise formulation). It is an analog of the classical Chern correspondence which generalizes previous results for gerbes \cite{Gualt} and which seems to be related to the Chern correspondence for Courant algebroids from \cite{GarRuTi} by a similar procedure to the one described in \cite{ShengXuZhu}. 

\begin{theorem}\label{th:chernintro}
    Let $\mathcal{P}_h$ be a smooth $\mathcal{K}$-bundle over a complex manifold $X$. Then holomorphic structures with holomorphic connective structure on its complexification are in bijection with enhanced connections $(A_h,B_h,g)$ on $\mathcal{P}_h$ such that $g \in \Gamma(S^2T^{\ast}X \otimes \mathfrak{t}_{\mathbb R})$ satisfies $g^{0,2} = 0$ and whose curvature satisfies $F_{A_h}^{0,2} = 0$, $H_h = d^c(g(J\cdot,\cdot))$.
\end{theorem}

Our definition of connection for $\mathcal{G}$-bundles is obtained by generalizing the notion of trivialization of the Chern-Simons 2-gerbe with connection from \cite{Bunke,CarMultGer,WaldString}, which we can do by using the data $\langle \cdot,\cdot \rangle$, $\Theta^L$ associated to a connective structure on $\mathcal{G}$. In particular, a $\mathcal{G}$-bundle with connection $(P,\mathcal{P},A,B)$ over $M$ has an underlying $G$-bundle with connection $(P,A)$ and the curvature of $(A,B)$ is a pair $(F_A,H)$ where $F_A \in \Omega^2(M,ad\,P)$ is the curvature of $A$ and $H \in \Omega^3(M,\mathfrak{t})$ satisfies the Green-Schwartz anomally cancellation equation \cite{Kil,Witten} which is expected in string theory from the field strength of a string with $\mathcal{G}$-symmetry
\begin{equation}\label{eq:anomalyintro}
    dH - \langle F_A \wedge F_A \rangle = 0.
\end{equation}
Moreover, we expand this notion of connection so as to include a symmetric tensor into the picture and call the resulting object an `enhanced connection'. This is a natural structure appearing in Theorem \ref{th:MaurerCartan} and Theorem \ref{th:chern}, that establish the existence of canonical connections with accompanying symmetric tensors in different contexts. A precursor of this idea in the context of Courant algebroids can be found in \cite{GarRuTi}.

As another proof of concept for our notion of connection, we relate it to the theory of adjusted connections for crossed modules from \cite{RistSaeWolf}. Namely, if $(\tilde{G},H,f,\triangleright)$ is a crossed module and $\mathcal{G}$ is a multiplicative gerbe modelling the same 2-group, then we prove that there is a correspondence between connective structures on $\mathcal{G}$ and a certain type of adjustments on $(\tilde{G},H,f,\triangleright)$ that we call `strong'. Furthermore, we show that under this correspondence the connections defined as above are equivalent to the \emph{adjusted connections} from loc. cit. 

The relation with adjusted connections is useful in the study of moduli spaces of connections for String$(n)$-principal bundles. Since the models for String$(n)$ as a multiplicative gerbe are not explicit, it is easier to construct and describe String$(n)$-bundles by using any of the infinite-dimensional models for String$(n)$ as an adjusted crossed module \cite{BaezCransSchSte,LudewigWald,RistSaeWolf}. Our Equations \eqref{eq:cocycleadj} allow to define connections on a String$(n)$-bundle which is described in those terms as certain local differential forms taking values in finite-dimensional vector spaces, as opposed to the Fr\'echet vector spaces in which the adjusted connections from \cite{RistSaeWolf} take values. Then the groupoid of connections on a fixed bundle can be analytically completed to a groupoid described by Banach manifolds, well suited for using techniques from functional analysis to study gauge-theoretic equations. 

One conclusion that we can draw from our work is that the object that describes the symmetries of higher gauge theory is not simply a Lie 2-group, but rather a Lie 2-group equipped with some additional data, equivalent to a connective structure if the Lie 2-group is described by a multiplicative gerbe. This has already been argued in the case of adjustments for crossed modules \cite{KimSaeTransport,RistSaeWolf,SaeSch} and the evidence for it goes beyond the fact that there is no consistent definition of connections for a general Lie 2-group in the literature that does not impose fake curvature conditions. Namely, from the point of view of complex geometry we can now note the following two facts. Firstly, that the complexification of a multiplicative gerbe is uniquely determined only if we require the existence of a \emph{holomorphic connective structure} on it. Secondly, that the Chern correspondence, in its cleanest form, relates enhanced connections on a principal 2-bundle and holomorphic structures \emph{with holomorphic connective structure} on its complexification.

Finally, we note that Theorem \ref{th:chernintro}, along with the anomally cancelation \eqref{eq:anomalyintro}, imply that a holomorphic structure with holomorphic connective structure on the complexification of a $\mathcal{K}$-bundle over a complex manifold $X$ determines a connection $A_h$ on the underlying $K$-bundle and an $\omega \in \Omega^{1,1}(X,\mathfrak{t}_{\mathbb R})$ satisfying
\begin{equation}\label{eq:HullStro}
    F_{A_h}^{0,2} = 0, \quad \quad dd^c\omega - \langle F_{A_h} \wedge F_{A_h} \rangle = 0.
\end{equation}
These two equations, for $\omega$ the Hermitian form of a Hermitian metric on $X$, appear naturally in heterotic string theory upon imposing supersymmetry in the Hull-Strominger system \cite{Hull,Strom}. This system has been proposed by S.-T. Yau as a generalization of K\"ahler geometry which could serve in the geometrization of Reid's fantasy of constructing a connected moduli space of Calabi-Yau manifolds with varying topological type \cite{Reid,Yau}.

This observation suggests two directions of further research. Firstly, there might be a natural algebro-geometric stability condition for holomorphic principal 2-bundles with holomorphic connective structure having the equations from the Hull-Strominger system as gauge-theoretic counterpart, a possibility that has recently been studied from the point of view of Courant algebroids in \cite{AshStrickTennWald,GarGon1,GarGon2,GarRuTi}. Secondly, it seems that holomorphic $\mathcal{G}$-bundles with holomorphic connective structure are the geometric objects that naturally prequantize Hermitian metrics satisfying \eqref{eq:HullStro}, similarly to how holomorphic line bundles prequantize K\"ahler metrics. Note that this is already interesting in the case $K = \{1\}$, $T_{\mathbb R} = U(1)$; then, this discussion implies that holomorphic $\mathbb C^{\ast}$-gerbes with holomorphic connective structures prequantize pluriclosed metrics, an idea that is present in the literature on generalized complex geometry \cite{Gualt}. It would be interesting to pursue a complete geometric quantization and to contrast it with other approaches in the quantization programme for the Hull-Strominger system, such as those based on vertex algebras \cite{AlvArrGar1,AlvArrGar2}.

This paper is organized as follows. Section \ref{sec:backgr} recalls sheaf cohomology on semi-simplicial manifolds, which is our main tool for proving classification results, and provides some background on Lie 2-groups and their relation with multiplicative gerbes. In Section \ref{sec:constrcomp} we introduce the notion of connective structure on a multiplicative gerbe, we explain their relation with adjustments in crossed modules, and we prove our two main results about multiplicative gerbes: the characterization of connective structures in terms of an enhanced curving (Theorem \ref{th:MaurerCartan}) and the complexification theorem for multiplicative gerbes (Theorem \ref{th:complexification}). In Section \ref{sec:2pbsandcon} we recall principal 2-bundles and define (enhanced) connections on them, giving cocycle data and studying their basic properties, as well as proving the equivalence with adjusted connections. In Section \ref{sec:2pbshol} we define semiconnections on principal 2-bundles, characterizing holomorphic structures on them, we prove the Chern correspondence (Theorem \ref{th:chern}) and we discuss some examples of holomorphic principal 2-bundles over the Hopf surface. We also include Appendix \ref{ap:gerbes} to recall some basic definitions about gerbes.

\begin{acknowledgements}
I would like to thank Domenico Fiorenza for helping me understand connections on principal 2-bundles. I am very grateful to Miquel Cueca and Chenchang Zhu for their kind hospitality and for the helpful discussions we had during my visit to G\"ottingen University. I have also benefited from some interesting conversations with Severin Bunk, Christian S\"amann, Carlos Shahbazi and Konrad Waldorf. Finally, I would like to thank my supervisors Luis \'Alvarez-C\'onsul and Mario Garcia-Fernandez for their guidance and careful reading throughout this project.
\end{acknowledgements} 

\section{Background on Lie 2-groups and multiplicative gerbes}
\label{sec:backgr}

\subsection{Sheaf cohomology on semi-simplicial manifolds}\label{ap:sheaf}

In this section we give a brief overview of sheaf cohomology on semi-simplicial manifolds following \cite{Del,Gomi}, and we recall Bott's Theorem from \cite{Bott} on the cohomology of the sheaves of differential forms on the classifying space of a Lie group. The tools provided here are used for stating and proving classification results for multiplicative gerbes and structures on them in the forthcoming sections. 

\begin{definition}
A \emph{semi-simplicial manifold} $M_{\bullet}$ is a collection of manifolds $M_n$, $n \geq 0$ with smooth maps $d^n_j: M_n \rightarrow M_{n-1}$, $j = 0,\,...,n$, called \emph{face maps}, satisfying $d^n_j d^{n-1}_i =  d^n_i d^{n-1}_{j-1}$ for $i < j$. A \emph{sheaf} $\mathcal{S}^{\bullet}$ on $M_{\bullet}$ is a collection of sheaves $\mathcal{S}^n$ on $M_n$ with maps $\partial^n_j:(d^n_j)^{\ast}\mathcal{S}^{n-1} \rightarrow \mathcal{S}^n$ satisfying the condition 
\begin{equation}
    \partial^n_j \circ (d^n_j)^{\ast}\partial^{n-1}_i =  \partial^n_i \circ (d^n_i)^{\ast}\partial^{n-1}_{j-1}
\end{equation}
for $i < j$. This ensures that the maps $\delta_{n-1}: \mathcal{S}^{n-1}(M_{n-1}) \rightarrow \mathcal{S}^n(M_n)$ defined by $\delta_{n-1} := \sum_{j=0}^{n-1}(-1)^{j}\partial^{n-1}_j(d_j^{n-1})^{\ast}$ satisfy $\delta_n \circ \delta_{n-1} = 0$. A \emph{morphism of sheaves} $\mathcal{S}^{\bullet}_0 \rightarrow \mathcal{S}^{\bullet}_1$ is a collection of morphisms $\mathcal{S}^n_0 \rightarrow \mathcal{S}^n_1$ intertwining the maps $(\partial^n_j)_0$, $(\partial^n_j)_1$. The \emph{global sections functor} $\Gamma$ is the functor from the category of sheaves of abelian groups on $M_{\bullet}$ to the category of abelian groups sending $\mathcal{S}^{\bullet} \mapsto \ker (\delta_0: \mathcal{S}^0(M_0) \rightarrow \mathcal{S}^1(M_1))$.
\end{definition}

The category of sheaves of abelian groups on $M_{\bullet}$ is abelian; moreover, it can be interpreted as the category of sheaves in a certain site and so it has enough injectives \cite[2.1.2]{Tamme}. For $p \in \mathbb N$, the \emph{sheaf cohomology group} $H^{p}(M_{\bullet},\mathcal{S}^{\bullet})$ is then the image of $\mathcal{S}^{\bullet}$ by the $p$-th derived functor of $\Gamma$. To compute these, it is sufficient to take a resolution $\mathcal{S}^{\bullet}  \stackrel{\epsilon}{\rightarrow} I^{\bullet,0} \stackrel{d}{\rightarrow} I^{\bullet,1} \stackrel{d}{\rightarrow} ...$ of the sheaf $\mathcal{S}^{\bullet}$ such that the standard sheaf cohomology groups $H^r(M_n,I^{n,m})$ vanish for $r > 0$ and $n, \,m \geq 0$ (for example, by taking \emph{functorial} acyclic resolutions of each sheaf $\mathcal{S}^n$ over $M_n$). 
Then $H^{\ast}(M_{\bullet},\mathcal{S}^{\bullet})$ is obtained as the total cohomology of the double complex $(I^{n,m}(M_n),\delta,d)$.

\emph{Sheaf hypercohomology} $\mathbb H(M_{\bullet},\mathcal{S}_0^{\bullet} \rightarrow ... \rightarrow \mathcal{S}_l^{\bullet})$ of a complex of sheaves is similarly defined as the derived functors of the functor $$\mathcal{S}_0^{\bullet} \stackrel{t}{\rightarrow} ... \stackrel{t}{\rightarrow} \mathcal{S}_l^{\bullet} \mapsto \ker (\delta_0\oplus t: \mathcal{S}^0_0(M_0) \rightarrow \mathcal{S}^1_0(M_1) \oplus \mathcal{S}_1^0(M_0));$$
it can be computed by taking acyclic resolutions $I_l^{\bullet,\bullet}$ of each $\mathcal{S}_l$ with maps $I_0^{\bullet,\bullet} \stackrel{t}{\rightarrow} I_1^{\bullet,\bullet} \stackrel{t}{\rightarrow} ...$ such that everything commutes and computing the total cohomology of the triple complex $(I^{n,m}_{k}(M_n),\delta,d,t)$. Cohomology and hypercohomology of sheaves satisfy the expected properties with respect to exact sequences and triangles, as they are defined in terms of derived functors.

The most natural examples of sheaves on semi-simplicial manifolds arise from considering a sheaf $\mathcal{S}$ that is constructed functorially over any manifold and letting $\mathcal{S}^n$ be the sheaf $\mathcal{S}$ on $M_n$, with $\partial^n_j$ the maps obtained by functoriality. All the examples that we consider in this article are of this form. For example, for a fixed finite-dimensional vector space $V$ and for fixed $q \in \mathbb N$, we define the sheaf $\Omega^q_{V}$ of $V$-valued $q$-forms on a semi-simplicial manifold $M_{\bullet}$ to be the sheaf on $M_{\bullet}$ which at each level $M_n$ is the sheaf of $V$-valued $q$-forms on $M_n$, with the maps $\partial^n_j = (d^n_j)^{\ast}$ between them. 

\begin{remark}\label{rk:simpacyclic}
    If $\mathcal{S}_{\bullet}$ is a sheaf over a semi-simplicial manifold $M_{\bullet}$ such that each $\mathcal{S}_n$ is an acyclic sheaf over $M_n$, then $H^{p}(M_{\bullet},\mathcal{S}_{\bullet}) = \ker (\delta: \mathcal{S}^p(M_p) \rightarrow \mathcal{S}^{p+1}(M_{p+1})) / \Im (\delta: \mathcal{S}^{p-1}(M_{p-1}) \rightarrow \mathcal{S}^p(M_p))$. In particular, for any semi-simplicial manifold $M_{\bullet}$ and any $p, \, q \in \mathbb N$ we have
    \begin{equation}
        H^{p}(M_{\bullet},\Omega^q_V) = \frac{\ker (\delta: \Omega^q(M_p,V) \rightarrow \Omega^q(M_{p+1},V)) }{\Im (\delta: \Omega^q(M_{p-1},V) \rightarrow \Omega^q(M_p,V))}.
    \end{equation}
\end{remark}

The following examples play an important role in this article.

\begin{example}\label{ex:simpderham}
    Consider over any semi-simplicial manifold $M_{\bullet}$ the sheaf $\underline{V}$ of locally constant $V$-valued functions on $M_{\bullet}$, for $V$ a fixed vector space. We write $H^{\ast}(M_{\bullet},V)$ for its cohomology. To compute it we may take de Rham resolutions $\Omega^{\bullet}_V$ on each $M_n$ 
    $$\underline{V} \rightarrow C^{\infty}_V \stackrel{d}{\rightarrow} \Omega^1_V \stackrel{d}{\rightarrow} ...,$$
    with the maps $\delta:\Omega^{m}_V(M_{n-1}) \rightarrow \Omega^{m}_V(M_n)$ being simply $\delta := \sum_{j= 0}^{n-1}(-1)^j (d^{n-1}_j)^{\ast}$. Then $H^{\ast}(M_{\bullet},V)$ is the total cohomology of $(\Omega^m_V(M_n),\delta,d)$. It is shown in \cite{BottShulSta} that this agrees with the singular $V$-valued cohomology of the \emph{geometric realization} of $M_{\bullet}$. 
\end{example}


\begin{example}\label{ex:bott}
    For $G$ a Lie group, the semi-simplicial manifold $BG_{\bullet}$ is defined by $BG_n := G^n$ and $d^n_{j}: G^n \rightarrow G^{n-1}$, $j = 0, ..., n$, $n \geq 2$,
\begin{align*}
    d^n_0(g_1,...,g_n) &= (g_2,g_3,g_4...,g_n), \\
    d^n_1(g_1,...,g_n) &= (g_1g_2,g_3,g_4,...,g_n), \\
    d^n_2(g_1,...,g_n) &= (g_1,g_2g_3,g_4,...,g_n), \\
    &... \\
    d^n_{n-1}(g_1,...,g_n) &= (g_1,g_2,...,g_{n-2},g_{n-1}g_n),\\
    d^n_n(g_1,...,g_n) &= (g_1,g_2,...,g_{n-2},g_{n-1}).
\end{align*}
The geometric realization of $BG_{\bullet}$ is the classifying space of $G$, and so for a fixed vector space $V$ the cohomology of this space with coefficients in $V$ can be computed as in Example \ref{ex:simpderham}.
\end{example}

Let $H^{\ast}_{gr,cont}(G,N)$ be group cohomology with coefficients in a $G$-module $N$, computed with continuous cochains. We are interested in taking $N = S^q \mathfrak{g}^{\ast} \otimes V$, for $q \in \mathbb N$, $\mathfrak{g}^{\ast}$ the dual of the Lie algebra of $G$ regarded as a $G$-module with the coadjoint action and $V$ a finite-dimensional vector space with trivial $G$-action (so, in particular, $H^{0}_{gr,cont}(G,S^{q}\mathfrak{g}^{\ast} \otimes V)$ is the space of $Ad$-invariant degree $q$ polynomials in $\mathfrak{g}$ with values in $V$). The following theorem from \cite{Bott} relates sheaf cohomology of the sheaf $\Omega^q_V$ on $BG_{\bullet}$ with group cohomology with coefficients in $S^q \mathfrak{g}^{\ast} \otimes V$. 

\begin{theorem}[\cite{Bott}]\label{th:bott}
    Let $G$ be a Lie group and let $V$ be a finite-dimensional vector space. Then
    \begin{equation}\label{eq:bott}
    H^{p}(BG_{\bullet},\Omega^q_V) = H^{p-q}_{gr,cont}(G,S^{q}\mathfrak{g}^{\ast} \otimes V).
    \end{equation}
    Moreover, if $G$ is compact, then $H^{r}_{gr,cont}(G,S^{q}\mathfrak{g}^{\ast} \otimes V) = 0$ for $r > 0$.
\end{theorem}
We also need the following explicit formulas for the isomorphism \eqref{eq:bott} in low degrees. In the sequel we will need the following notation: for manifolds $X$, $Y$ and points $x \in X$, $y \in Y$ we identify $T_{(x,y)}(X \times Y) = T_xX \oplus T_yY$ and write $v_x + v_y \in T_{(x,y)}(X \times Y)$ for vectors tangent to the product manifold.

\begin{lemma}\label{lem:bott}
Let $G$ be a Lie group and let $V$ be a vector space. Then 
    \begin{enumerate}
    \item\label{it:bott1} The isomorphism $H^2(BG_{\bullet},\Omega^1_V) = H^1_{gr,cont}(G,\mathfrak{g}^{\ast} \otimes V)$ is induced by the map
    \begin{align*}
        \phi:\frac{\{\tau \in \Omega^1(G^2,V)\, | \: \delta \tau = 0 \}}{\{\delta \sigma \, | \: \sigma \in \Omega^1(G,V)\}} &\rightarrow \frac{\{\kappa: G \times \mathfrak{g} \rightarrow V \,| \: \kappa(g_1g_2,v) = \kappa(g_1,v) + \kappa(g_2,g_1^{-1}vg_1) \}}{\{\kappa(g,v) = \chi(g^{-1}vg)-\chi(v) \, | \: \chi: \mathfrak{g} \rightarrow V\}}
    \end{align*}
    with 
    \begin{align}
        \phi(\tau)(g,v) &= \tau_{(g^{-1},1)}(0+v) + \tau_{(g^{-1},g)}(g^{-1}v+0),\\
        \phi^{-1}(\kappa)_{(g_1,g_2)}(v_{g_1}+v_{g_2}) &= \kappa(g_2,g_1^{-1}v_{g_1}).
    \end{align}
    \item\label{it:bott2} The isomorphism $H^2(BG_{\bullet},\Omega^2_V) = H^0_{gr,cont}(G,S^2\mathfrak{g}^{\ast} \otimes V)$ is induced by the map:
    \begin{align*}
        \psi:\frac{\{\nu \in \Omega^2(G^2,V)\, | \: \delta \nu = 0 \}}{\{\delta \sigma \, | \: \sigma \in \Omega^2(G,V)\}} &\rightarrow \{\langle \cdot,\cdot \rangle: S^2\mathfrak{g} \rightarrow V\, | \langle Ad(g)u,Ad(g)v \rangle = \langle u,v \rangle\}
    \end{align*}
    with 
    \begin{align}
        \psi(\nu)(u,v) &= \frac{1}{2}\nu_{(1,1)}(0+u,v+0) + \frac{1}{2}\nu_{(1,1)}(0+v,u+0), \label{eq:bottpairing}\\
        \psi^{-1}(\langle \cdot,\cdot \rangle) &= -\langle g_1^{\ast}\theta^L \wedge g_2^{\ast}\theta^R \rangle,
    \end{align}
    where $\theta^L, \,\theta^R \in \Omega^1(G,\mathfrak{g})$ are the left- and right-invariant Maurer-Cartan 1-forms on $G$.
\end{enumerate}
\end{lemma}
\begin{proof}
    It follows from straightforward computations.
\end{proof}

\subsection{Lie 2-groups}

In this section we recall the definition of Lie 2-groups in terms of Lie groupoids and anafunctors following \cite{ScPr}.

\begin{definition}
    A \emph{Lie groupoid} is a small category $\mathfrak{X}$ such that the space of objects $\mathfrak{X}_0$ and the space of arrows $\mathfrak{X}_1$ are manifolds, with the source and target maps $s,t: \mathfrak{X}_1 \rightarrow \mathfrak{X}_0$ being surjective submersions and the identity $id:\mathfrak{X}_0 \rightarrow \mathfrak{X}_1$ and composition $\circ: \mathfrak{X}_1 \prescript{}{s}{\times}^{}_t \mathfrak{X}_1 \rightarrow \mathfrak{X}_1$ being smooth maps. A \emph{smooth functor} $F: \mathfrak{X} \rightarrow \mathfrak{Y}$ between Lie groupoids is a functor whose corresponding maps $F_0: \mathfrak{X}_0 \rightarrow \mathfrak{Y}_0$, $F_1: \mathfrak{X}_1 \rightarrow \mathfrak{Y}_1$ are smooth and a \emph{smooth natural transformation} $\alpha: F \Rightarrow G: \mathfrak{X} \rightarrow \mathfrak{Y}$ is a natural transformation whose corresponding map $\alpha: \mathfrak{X}_0 \rightarrow \mathfrak{Y}_1$ is smooth.
\end{definition}

The following three basic examples of Lie groupoids play an important role in this article. Firstly, if $M$ is a manifold then we regard it as a Lie groupoid with only identity arrows. Secondly, if $G$ is a Lie group then there is a Lie groupoid $BG$ with only one object and whose manifold of arrows is $G$ (using the group product to define composition). Finally, given a smooth (right) action of $G$ on $M$ then the \emph{quotient groupoid} $M/\!/G$ is the Lie groupoid with $M$ as manifold of objects and $M \times G$ as manifold of arrows, where $(x,g)$ is seen as an arrow from $x$ to $x \cdot g$ and composition is $(x \cdot g_1, g_2) \circ (x,g_1) = (x, g_1g_2)$.


We regard Lie groupoids as placeholders for differentiable stacks. As such, the natural maps $\mathfrak{X} \rightarrow \mathfrak{Y}$ between two Lie groupoids are called \emph{anafunctors}, \emph{bibundles} or \emph{Morita morphisms}.

\begin{definition}\label{def:anafunctor}
    An \emph{anafunctor} $(X,\pi,F): \mathfrak{X} \rightarrow \mathfrak{Y}$ between Lie groupoids $\mathfrak{X}$, $\mathfrak{Y}$ is a manifold $X$ with a smooth map $\pi:X \rightarrow \mathfrak{X}_0$ such that $\mathfrak{X}_1 \prescript{}{s}{\times}^{}_{\pi} X \stackrel{t}{\rightarrow} \mathfrak{X}_0$ is a surjective submersion and a smooth functor $F:\pi^{\ast}\mathfrak{X} \rightarrow \mathfrak{Y}$, where $\pi^{\ast}\mathfrak{X}$ is the Lie groupoid with $\pi^{\ast}\mathfrak{X}_0 = X$, $\pi^{\ast}\mathfrak{X}_1 = (X \times X) \prescript{}{\pi \times \pi}{\times}^{}_{s \times t} \mathfrak{X}_1$ and composition induced by the one in $\mathfrak{X}_1$. Given two anafunctors $(X,\pi,F), (X',\pi',F'): \mathfrak{X} \rightarrow \mathfrak{Y}$, a \emph{transformation} $\alpha$ between them is a smooth map $\alpha: X \rightarrow \mathfrak{X}_1 \prescript{}{s}{\times}^{}_{\pi'} X'\prescript{}{F'_0}{\times}^{}_{t} \mathfrak{Y}_1 /\sim$ denoted by $x \mapsto [\alpha^{L}(x),\alpha_0(x),\alpha^{R}(x)]$, where the equivalence relation is
    \begin{align}\label{eq:eqrelbibundle}
        (\gamma,x',\eta) \sim (\gamma \circ \gamma_{x'y'}^{-1},y',F_1'(x',y',\gamma_{x'y'}) \circ \eta) \quad \text{ for any } \gamma_{x'y'}: \pi'(x') \rightarrow \pi'(y') \text{ in } \mathfrak{X}_1 ,
    \end{align}
    such that $\alpha^L(x): \pi'(\alpha_0(x)) \rightarrow \pi(x)$, $\alpha^R(x): F_0(x) \rightarrow F_0'(\alpha_0(x))$ and for every $x, \,y \in X$ and every $\gamma_{xy}: \pi(x) \rightarrow \pi(y)$ in $\mathfrak{X}_1$ 
    we have a commutative diagram
    \begin{equation}
        \begin{tikzcd}[column sep = 13em]
             F_0(x) \ar[r,"{F_1(x,y,\gamma_{xy})}",{name=U}] \ar[d,"{\alpha^R(x)}",swap] 
             & F_0(y) \ar[d,"{\alpha^R(y)}"] \\
             F_0'(\alpha_0(x)) \ar[r," {F_1'(\alpha_0(x),\alpha_0(y),\alpha^L(y)^{-1} \circ \gamma_{xy} \circ \alpha^L(x))}",{name=D},swap] 
             & F_0'(\alpha_0(y)).
        \end{tikzcd}
        \end{equation}
\end{definition}

\begin{remark}\label{rk:easyalpha}
    Given two anafunctors $(X,\pi,F), (X',\pi',F'): \mathfrak{X} \rightarrow \mathfrak{Y}$, an intuitive way of constructing a transformation $\alpha$ between them is to give a manifold $Z$ with maps $l: Z \rightarrow X$, $l':Z \rightarrow X'$ such that $\pi \circ l = \pi' \circ l' =: \pi_Z$ and such that $\mathfrak{X}_1 \prescript{}{s}{\times}^{}_{\pi_Z} Z \stackrel{t}{\rightarrow} \mathfrak{X}_0$ is a surjective submersion together with a smooth natural transformation $\alpha^Z:l^{\ast}F \Rightarrow (l')^{\ast}F': \pi_Z^{\ast}\mathfrak{X} \rightarrow \mathfrak{Y}$. Then the transformation $\alpha:X \rightarrow \mathfrak{X}_1 \prescript{}{s}{\times}^{}_{\pi'} X'\prescript{}{F'_0}{\times}^{}_{t} \mathfrak{Y}_1 /\sim$ is defined as $x \mapsto [\gamma,l'(z),\alpha_Z(z) \circ F_1(x,l(z),\gamma^{-1})]$, where $(\gamma,z) \in \mathfrak{X}_1 \prescript{}{s}{\times}^{}_{\pi_Z} Z$ is any pair with $t(\gamma) = \pi(x)$.
\end{remark}

\begin{remark}
Given an anafunctor $(X,\pi,F):\mathfrak{X} \rightarrow \mathfrak{Y}$ as in Definition \ref{def:anafunctor} one can recover an anafunctor as defined in \cite{NikWal2pbs} as follows. The total space of the bibundle is $ \mathfrak{X}_1 \prescript{}{s}{\times}^{}_{\pi} X\prescript{}{F_0}{\times}^{}_{t} \mathfrak{Y}_1 / \sim$, where the equivalence relation is as in \eqref{eq:eqrelbibundle} and the actions of $\mathfrak{X}$ and $\mathfrak{Y}$ are given by composition on each component. This can be enhanced to an equivalence of categories between the category of anafunctors $\mathfrak{X} \rightarrow \mathfrak{Y}$ as defined here and the category of anafunctors $\mathfrak{X} \rightarrow \mathfrak{Y}$ as defined in \cite{NikWal2pbs}.
\end{remark}

In particular, smooth functors $\mathfrak{X} \rightarrow \mathfrak{Y}$ provide examples of anafunctors, but anafunctors are more general. Lie groupoids, anafunctors and their transformations form a bicategory in which composition is defined by taking fibered products of the involved surjective submersions in the only possible way.

\begin{definition}\label{def:2grp}
    A \emph{Lie 2-group} is a Lie groupoid $\mathfrak{G}$ with an anafunctor $m: \mathfrak{G} \times \mathfrak{G} \rightarrow \mathfrak{G}$ (the \emph{product}) and a transformation $\alpha: m \circ (m \times id) \Rightarrow m \circ (id \times m): \mathfrak{G} \times \mathfrak{G} \times \mathfrak{G} \rightarrow \mathfrak{G}$ (the \emph{associator}) satisfying
    \begin{enumerate}
        \item The \emph{pentagon identity}: 
        \begin{equation}
        \begin{tikzcd}
            \mathfrak{G} \times \mathfrak{G} \times \mathfrak{G} \times \mathfrak{G} \ar[rr,bend left=70,"{((g_1g_2)g_3)g_4}"{name=F}] 
              \ar[rr,bend left=20,"{(g_1(g_2g_3))g_4}"{name=G,2cell}]
              \ar[rr,bend right=20,"{g_1((g_2g_3)g_4)}"{name=H,2cell}]
              \ar[rr,bend right=70,"{g_1(g_2(g_3g_4))}"{name=E,swap}] & & 
                \mathfrak{G} \\
             & \ar[Rightarrow,from=F,to=G,"{\alpha}"{swap,pos=0.3},shorten >=1.5pt] \ar[Rightarrow,from=G,to=H,"{\alpha}"{swap}] \ar[Rightarrow,from=H,to=E,"{\alpha}"{swap}] &
        \end{tikzcd} = 
        \begin{tikzcd}
            \mathfrak{G} \times \mathfrak{G} \times \mathfrak{G} \times \mathfrak{G} \ar[rr,bend left=50,"{((g_1g_2)g_3)g_4}"{name=F}] 
              \ar[rr,"{(g_1g_2)(g_3g_4)}"{name=G,2cell}]
              \ar[rr,bend right=50,"{g_1(g_2(g_3g_4))}"{name=H,swap}] & & 
                \mathfrak{G} \\
             & \ar[Rightarrow,from=F,to=G,"{ \alpha}"{swap,pos=0.3},shorten >=1.5pt] \ar[Rightarrow,from=G,to=H,"{\alpha}"{swap}] & 
        \end{tikzcd}.
        \end{equation}

        \item \emph{Existence of inverses}: The anafunctor $p_2 \times m: \mathfrak{G} \times \mathfrak{G} \rightarrow \mathfrak{G} \times \mathfrak{G}$ has a quasi-inverse, where $p_2$ is projection onto the first factor.
    \end{enumerate}
A \emph{homomorphism of Lie 2-groups} $(F,\alpha^F):\mathfrak{G} \rightarrow \mathfrak{H}$ is an anafunctor $F:\mathfrak{G} \rightarrow \mathfrak{H}$ with a transformation $\alpha^F: F \circ m_{\mathfrak{G}} \Rightarrow m_{\mathfrak{H}} \circ F: \mathfrak{G} \times \mathfrak{G} \rightarrow \mathfrak{H}$ such that
    \begin{equation}
        \begin{tikzcd}[column sep = large]
            \mathfrak{G} \times \mathfrak{G} \times \mathfrak{G} \ar[rr,bend left=70,"{F((g_1g_2)g_3)}"{name=F}] 
              \ar[rr,bend left=20,"{F(g_1g_2)F(g_3)}"{name=G,2cell}]
              \ar[rr,bend right=20,"{(F(g_1)F(g_2))F(g_3)}"{name=H,2cell}]
              \ar[rr,bend right=70,"{F(g_1)(F(g_2)F(g_3))}"{name=E,swap}] & & 
                \mathfrak{H} \\
             & \ar[Rightarrow,from=F,to=G,"{\alpha^{F}}"{swap,pos=0.3},shorten >=1.5pt] \ar[Rightarrow,from=G,to=H,"{\alpha^{F}}"{swap,pos=0.8}] \ar[Rightarrow,from=H,to=E,"{\alpha^{\mathfrak{H}}}"{swap}] &
        \end{tikzcd} = 
        \begin{tikzcd}
            \mathfrak{G} \times \mathfrak{G} \times \mathfrak{G} \ar[rr,bend left=70,"{F((g_1g_2)g_3)}"{name=F}] 
              \ar[rr,bend left = 20,"{F(g_1(g_2g_3))}"{name=G,2cell}]
              \ar[rr,bend right=20,"{F(g_1)F(g_2g_3)}"{name=H,2cell}]
              \ar[rr,bend right=70,"{F(g_1)(F(g_2)F(g_3))}"{name=E,swap}]& & 
                \mathfrak{H}  \\
             & \ar[Rightarrow,from=F,to=G,"{\alpha^{\mathfrak{G}}}"{swap,pos=0.3},shorten >=1.5pt] \ar[Rightarrow,from=G,to=H,"{\alpha^{F}}"{swap,pos=0.9}]  \ar[Rightarrow,from=H,to=E,"{\alpha^{F}}"{swap,pos=0.8}]& 
        \end{tikzcd}.
        \end{equation}
Given homomorphisms $(F_1,\alpha^{F_1}),\, (F_2,\alpha^{F_2}): \mathfrak{G} \rightarrow \mathfrak{H}$, then a \emph{transformation} between them is a transformation of anafunctors $\psi: F_1 \Rightarrow F_2: \mathfrak{G} \rightarrow \mathfrak{H}$ such that
\begin{equation}
        \begin{tikzcd}[column sep = large]
            \mathfrak{G} \times \mathfrak{G} \ar[rr,bend left=70,"{F_1(g_1g_2)}"{name=F}] 
              \ar[rr,bend left=10,"{F_1(g_1)F_1(g_2)}"{name=G,2cell}]
              \ar[rr,bend right=40,"{F_2(g_1)F_2(g_2)}"{name=H,2cell}]
             & & \mathfrak{H} \\ \ar[Rightarrow,from=F,to=G,"{\alpha^{F_1}}"{swap,pos=0.3},shorten >=1.5pt] \ar[Rightarrow,from=G,to=H,"{\psi}"{swap,pos=0.9}] &
        \end{tikzcd} = 
        \begin{tikzcd}
            \mathfrak{G} \times \mathfrak{G} \ar[rr,bend left=70,"{F_1(g_1g_2)}"{name=F}] 
              \ar[rr,bend left = 10,"{F_2(g_1g_2)}"{name=G,2cell}]
              \ar[rr,bend right=40,"{F_2(g_1)F_2(g_2)}"{name=H,swap}] & & 
                \mathfrak{H}  \\
             & \ar[Rightarrow,from=F,to=G,"{\psi}"{swap,pos=0.3},shorten >=1.5pt] \ar[Rightarrow,from=G,to=H,"{\alpha^{F_2}}"{swap}] & 
        \end{tikzcd}.
        \end{equation}  
        Lie 2-groups, their homomorphisms and transformations form the \emph{bicategory of Lie 2-groups}.
\end{definition}

\begin{remark}
    The diagrams in Definition \ref{def:2grp} are equalities between transformations of anafunctors. For example, the pentagon identity is an equality between two transformations $m \circ (m \times id) \circ (m \times id \times id) \Rightarrow m \circ (id \times m) \circ (id \times id \times m): \mathfrak{G}^4 \rightarrow \mathfrak{G}$ that are defined by composing $\alpha$ in the different ways that the diagram depicts: each black arrow represents an anafunctor (for example, we write $((g_1g_2)g_3)g_4$ for the anafunctor $m \circ (m \times id) \circ (m \times id \times id)$) and each 2-cell represents a smooth transformation that is defined in terms of $\alpha$ in the only possible way. In this article we make frequent use of these sort of diagrams.
\end{remark}

There are two basic examples of Lie 2-groups. On the one hand, a Lie group $G$ regarded as a Lie groupoid with only identity arrows is a Lie 2-group with product given by its group product and trivial associator. On the other hand, for $T$ an abelian Lie group, the Lie groupoid $BT$ is a Lie 2-group with $m: BT \times BT \rightarrow BT$ the functor acting on arrows as $(t_1,t_2) \mapsto t_1t_2$. For this to be a functor it is necessary that $T$ is abelian. The Lie 2-groups that we study in this article are those $\mathfrak{G}$ for which there exist Lie groups $G$, $T$ with $T$ abelian such that there is a \emph{central extension of Lie 2-groups}
\begin{equation}
    1 \rightarrow BT \rightarrow \mathfrak{G} \rightarrow G \rightarrow 1.
\end{equation} 
Such central extensions are defined in the bicategory of Lie 2-groups in an analogous way to how they are defined in the category of Lie groups \cite{ScPr}. There is a way of characterizing which Lie 2-groups are of this form in terms of the following invariants that can be defined for any Lie 2-group $\mathfrak{G}$:
\begin{enumerate}
    \item The topological group $\pi_0(\mathfrak{G})$ of isomorphism classes of objects with group product given by $m$.
    \item The Lie group $\pi_1(\mathfrak{G})$ of automorphisms in $\mathfrak{G}_1$ of any $e \in \mathfrak{G}_0$ projecting to $1 \in \pi_0(\mathfrak{G})$.
    \item A continuous action $\triangleright$ of $\pi_0(\mathfrak{G})$ on $\pi_1(\mathfrak{G})$, defined by $[g] \triangleright f := id_g \cdot f \cdot id_{g^-{1}}$\footnote{More precisely: if $m$ is a smooth functor, for $[g] \in \pi_0(\mathfrak{G})$ and $f \in \pi_1(\mathfrak{G}) $ choose $g,\, g^{-1} \in \mathfrak{G}_0$ representing $[g],\, [g]^{-1}$ and an isomorphism $l:m(m(g,e),g^{-1}) \rightarrow e$; then define $[g] \triangleright f := l \circ m(m(id_g,f),id_{g^{-1}}) \circ l^{-1}$. If $m$ is an anafunctor defined over the map $\pi:X \rightarrow \mathfrak{G}_0^2$, the same idea works except that one must take representatives of $(g,e)$ and $(m(g,e),g^{-1})$ on $X$.} for $[g] \in \pi_0(\mathfrak{G})$ and $f \in \pi_1(\mathfrak{G})$.
\end{enumerate}
One can easily show that $\pi_1(\mathfrak{G})$ is abelian and that $\mathfrak{G}$ fits into an exact sequence of topological 2-groups
\begin{equation}\label{eq:2grpext}
        1 \rightarrow B\pi_1(\mathfrak{G}) \rightarrow \mathfrak{G} \rightarrow \pi_0(\mathfrak{G}) \rightarrow 1.
\end{equation}
In fact, if $\mathfrak{G}$ fits into a sequence $1 \rightarrow BT \rightarrow \mathfrak{G} \rightarrow G \rightarrow 1$ then necessarily $G = \pi_0(\mathfrak{G})$ and $T = \pi_1(\mathfrak{G})$. Moreover, central extensions are classified by the following result.

\begin{proposition}[\cite{ScPr}]\label{prop:centralextensions}
    The sequence \eqref{eq:2grpext} is a central extension of topological 2-groups if and only if the action $\triangleright$ of $\pi_0(\mathfrak{G})$ on $\pi_1(\mathfrak{G})$ is trivial.
\end{proposition}

Sequence \eqref{eq:2grpext} is a sequence of Lie (not just topological) 2-groups as long as $\pi_0(\mathfrak{G})$ is a Lie group. Such central extensions are our main object of study. 

\begin{example}\label{ex:crossed}
A Lie 2-group $(\mathfrak{G},m,\alpha)$ is \emph{strict} if $m$ is a smooth functor, the quasi-inverse of $p_1 \times m$ is also a smooth functor and $\alpha = id$. These are equivalent to \emph{Lie crossed modules}, that is, quadruples $(\tilde{G},H,f,\triangleright)$, where $\tilde{G}$, $H$ are Lie groups, $f: H \rightarrow \tilde{G}$ is a smooth homomorphism and $g \triangleright h$ denotes a smooth left action of $\tilde{G}$ on $H$ by automorphisms that satisfies 
\begin{align}
    f(g \triangleright h) &= gf(h)g^{-1},\\
    f(h_1) \triangleright h_2 &= h_1h_2h_1^{-1}
\end{align}
for $g \in \tilde{G}$, $h, \,h_1, \,h_2 \in H$. The corresponding Lie 2-group is the groupoid $\mathfrak{G} := H \backslash\!\backslash \tilde{G}$, with $H$ acting on $\tilde{G}$ on the left as $h \cdot g = f(h)g$ and $m: \mathfrak{G} \times \mathfrak{G} \rightarrow \mathfrak{G}$ defined on arrows as the semi-direct product $m((g_1,h_1),(g_2,h_2)) = (g_1g_2,h_1 \cdot g_1 \triangleright h_2)$ (see \cite{BaezLaudaV} for details on the equivalence). The invariants associated to such $\mathfrak{G}$ are $\pi_0(\mathfrak{G}) = \tilde{G}/Im(f)$ and $\pi_1(\mathfrak{G}) = Ker(f)$, with the action of $\pi_0(\mathfrak{G})$ on $\pi_1(\mathfrak{G})$ being induced by that of $\tilde{G}$ on $H$. We define a \emph{central Lie crossed module} to be a Lie crossed module $(\tilde{G},H,f,\triangleright)$ such that $\tilde{G}/Im(f)$ is a Lie group (i.e., such that $Im(f) \subset \tilde{G}$ is closed) acting trivially on $Ker(f)$.
\end{example}

\subsection{Multiplicative gerbes}\label{sec:multger}
Multiplicative gerbes appear originally in \cite{BryGauge,CarMultGer}. In \cite{ScPr} Schommer-Pries proved that, for $G$, $T$ Lie groups with $T$ abelian, central extensions of $G$ by $BT$ as Lie 2-groups are equivalent to multiplicative $T$-gerbes over $G$ and gave in this way the first finite-dimensional model for the 2-group String$(n)$. In this article we use the language of multiplicative gerbes for working with these Lie 2-groups, so we recall here the main definitions and results.

Recall the semi-simplicial manifold $BG_{\bullet}$ from Example \ref{ex:bott} with its maps $d^n_j: G^n \rightarrow G^{n-1}$. In order to simplify notation, we denote in what follows any possible composition of these maps by its image; for example, $g_1g_2g_3: G^3 \rightarrow G$ is the map $d^2_1 \circ d^3_1$, while $(g_1g_2,g_3g_4): G^4 \rightarrow G^2$ is the map $d^3_2 \circ d^4_1$. For the necessary background on gerbes, see Appendix \ref{ap:gerbes}.

\begin{definition}\label{def:multger}
	For $T$ an abelian Lie group, a \emph{multiplicative $T$-gerbe} over a Lie group $G$ is the following data:
 
	\begin{enumerate}
            \item A $T$-gerbe $\mathcal{G} \rightarrow G$,
            
            \item An isomorphism $m$ of $T$-gerbes over $G \times G$ (the \emph{product}) $m:g_1^{\ast}\mathcal{G} \otimes g_2^{\ast}\mathcal{G} \rightarrow (g_1g_2)^{\ast}\mathcal{G}$,
            
            \item A 2-isomorphism $\alpha$ of $T$-gerbes over $G \times G \times G$ (the \emph{associator})
        \begin{equation}
        \begin{tikzcd}
             g_1^{\ast}\mathcal{G} \otimes g_2^{\ast}\mathcal{G} \otimes g_3^{\ast}\mathcal{G} \ar[r,"{(g_1,g_2)^*m}",{name=U}] \ar[d,"{(g_2,g_3)^{\ast}m}",swap] 
             & (g_1g_2)^{\ast}\mathcal{G}\otimes g_3^{\ast}\mathcal{G} \ar[d,"{(g_1g_2,g_3)^{\ast}m}"] \ar[Rightarrow, dl, "\alpha"]\\
             g_1^{\ast}\mathcal{G} \otimes (g_2g_3)^{\ast}\mathcal{G} \ar[r,"{(g_1,g_2g_3)^{\ast}m}",{name=D},swap] 
             & (g_1g_2g_3)^{\ast}\mathcal{G}
        \end{tikzcd}
        \end{equation}
        
    such that, over $G \times G \times G \times G$, 
        \begin{equation}
        \adjustbox{scale=0.85,center}{
        \begin{tikzcd}
            g_1^{\ast}\mathcal{G} \otimes g_2^{\ast}\mathcal{G} \otimes g_3^{\ast}\mathcal{G} \otimes g_4^{\ast}\mathcal{G} \ar[rr,bend left=70,"{m(((g_1g_2)g_3)g_4)}"{name=F}] 
              \ar[rr,bend left=20,"{m((g_1(g_2g_3))g_4)}"{name=G,2cell}]
              \ar[rr,bend right=20,"{m(g_1((g_2g_3)g_4))}"{name=H,2cell}]
              \ar[rr,bend right=70,"{m(g_1(g_2(g_3g_4)))}"{name=E,swap}] & & 
                (g_1g_2g_3g_4)^{\ast}\mathcal{G} \\
             & \ar[Rightarrow,from=F,to=G,"{\alpha }"{swap,pos=0.3},shorten >=1.5pt] \ar[Rightarrow,from=G,to=H,"{\alpha}"{swap}] \ar[Rightarrow,from=H,to=E,"{\alpha}"{swap}] &
        \end{tikzcd} = 
        \begin{tikzcd}
            g_1^{\ast}\mathcal{G} \otimes g_2^{\ast}\mathcal{G} \otimes g_3^{\ast}\mathcal{G} \otimes g_4^{\ast}\mathcal{G} \ar[rr,bend left=50,"{m(((g_1g_2)g_3)g_4)}"{name=F}] 
              \ar[rr,bend left = 20,"{m((g_1g_2)(g_3g_4))}"{name=G,2cell}]
              \ar[rr,bend right=20,"{m(g_1(g_2(g_3g_4)))}"{name=H,swap}] & & 
                (g_1g_2g_3g_4)^{\ast}\mathcal{G} \\
             & \ar[Rightarrow,from=F,to=G,"{\alpha }"{swap,pos=0.3},shorten >=1.5pt] \ar[Rightarrow,from=G,to=H,"{\alpha}"{swap}] & 
        \end{tikzcd}.
        }
        \end{equation}
\end{enumerate}
\emph{Isomorphisms} and \emph{2-isomorphisms} of multiplicative gerbes are defined similarly as in Definition \ref{def:2grp}, replacing anafunctors by isomorphisms of gerbes and transformations by 2-isomorphisms of gerbes. This yields the \emph{bicategory of multiplicative $T$-gerbes over $G$}.
\end{definition}
\begin{remark}
    The last diagram of Definition \ref{def:multger} is an equality between 2-isomorphisms of gerbes: each black arrow represents an isomorphism (for example, we are writing $m(((g_1g_2)g_3)g_4) := (g_1g_2g_3,g_4)^{\ast}m \circ (g_1g_2,g_3)^{\ast}m \otimes id \circ (g_1,g_2)^{\ast}m \otimes id \otimes id$) and each 2-cell is a 2-isomorphism constructed from $\alpha$.
\end{remark}

In order to describe multiplicative gerbes in terms of cocycle data one must take a \emph{good semi-simplicial cover} of $BG_{\bullet}$. This is a collection $\{\mathcal{U}_n\}_{n \geq 1}$, where each $\mathcal{U}_n = \{U^n_{i_n}\}_{i_n \in I_n}$ is a good cover of $G^n$ indexed by a set $I_n$, together with maps $\tilde{d}^n_j:I_n \rightarrow I_{n-1}$ such that $d^n_j(U^n_{i_n}) \subset U^{n-1}_{\tilde{d}^n_j(i_n)}$ and that $\{I_n,\tilde{d}^n_j\}_{n,j}$ is a semi-simplicial set. In what follows we abuse notation by writing simply $\tilde{d}^n_j = d^n_j$; furthermore, we denote $U^n_{i_n^1i_n^2...i_n^k} := \bigcap_{s=1}^k U^n_{i_n^s}$. There are constructions of good semi-simplicial covers of $BG_{\bullet}$ in \cite{BryGauge,Mein}. Given such a cover, it follows directly from the definitions that a multiplicative $T$-gerbe over $G$ is given by  
\begin{equation}\label{eq:cocycle1}
\lambda_{i_1j_1k_1}: U^1_{i_1j_1k_1} \rightarrow T, \quad m_{i_2j_2}: U^2_{i_2j_2} \rightarrow T, \quad \alpha_{i_3}: U^3_{i_3} \rightarrow T
\end{equation}
satisfying
\begin{align}
\begin{split}
    &\lambda_{i_1j_1k_1}(g)\lambda_{i_1j_1l_1}^{-1}(g)\lambda_{i_1k_1l_1}(g)\lambda_{j_1k_1l_1}^{-1}(g) = 1, \\
    &m_{i_2j_2}(g_1,g_2)m_{i_2k_2}^{-1}(g_1,g_2)m_{j_2k_2}(g_1,g_2) \\
    &\quad \quad \quad \quad  = \lambda_{d_0(i_2)d_0(j_2)d_0(k_2)}(g_2)\lambda_{d_1(i_2)d_1(j_2)d_1(k_2)}^{-1}(g_1g_2)\lambda_{d_2(i_2)d_2(j_2)d_2(k_2)}(g_1),\\
    &\alpha_{i_3}(g_1,g_2,g_3)\alpha_{j_3}^{-1}(g_1,g_2,g_3) \\
    &\quad \quad \quad \quad  = m_{d_3(i_3)d_3(j_3)}^{-1}(g_1,g_2) m_{d_2(i_3)d_2(j_3)}(g_1,g_2g_3) m_{d_1(i_3)d_1(j_3)}(g_1g_2,g_3)^{-1}m_{d_0(i_3)d_0(j_3)}(g_2,g_3) ,\\
    &\alpha_{d_4(i_4)}(g_1,g_2,g_3)\alpha_{d_3(i_4)}^{-1}(g_1,g_2,g_3g_4)\alpha_{d_2(i_4)}(g_1,g_2g_3,g_4)\alpha_{d_1(i_4)}^{-1}(g_1g_2,g_3,g_4)\alpha_{d_0(i_4)}(g_2,g_3,g_4) = 1.
\end{split}
\end{align}

The following is the main result in \cite{ScPr}. Here we give a brief sketch of the proof.

\begin{theorem}[\cite{ScPr}]\label{th:ScPr}
    Let $G$, $T$ be Lie groups with $T$ abelian. There is an equivalence of bicategories between the bicategory of central extensions of $G$ by $BT$ and the bicategory of multiplicative $T$-gerbes over $G$.
\end{theorem}
\begin{proof}
    Given a multiplicative $T$-gerbe over $G$, we describe it with cocycle data $\lambda_{i_1j_1k_1}$, $m_{i_2j_2}$, $\alpha_{i_3}$ as above. Then we construct the Lie groupoid $\mathfrak{G}$ by $\mathfrak{G}_0 := \sqcup_{i_1 \in I_1} U_{i_1}^1$, $\mathfrak{G}_1 := \sqcup_{i_1j_1 \in I_1} U_{i_1j_1}^1 \times T$, where $(i_1,j_1,g,t) \in \mathfrak{G}_1$ is seen as an arrow $(i_1,g) \rightarrow (j_1,g)$ and composition is defined as $(j_1,k_1,g,t_2) \circ (i_1,j_1,g,t_1) := (i_1,k_1,g,t_1t_2\lambda_{i_1j_1k_1}(g))$; the cocycle condition for $\lambda_{i_1j_1k_1}$ ensures that this is associative. Then the anafunctor $m: \mathfrak{G} \times \mathfrak{G} \rightarrow \mathfrak{G}$ is defined on the cover
    \begin{align*}
        \pi: \sqcup_{i_2 \in I_2} U_{i_2}^2 &\rightarrow \sqcup_{i_1 \in I_1} U_{i_1}^1 \times \sqcup_{i_1 \in I_1} U_{i_1}^1\\
        (i_2,g_1,g_2) &\mapsto ((d_2(i_2),g_1), (d_0(i_2),g_2))
    \end{align*}
    as the smooth functor $\pi^{\ast}(\mathfrak{G} \times \mathfrak{G}) \rightarrow \mathfrak{G}$ defined on arrows by $(i_2,j_2,g_1,g_2,t_1,t_2) \mapsto (d_1(i_2),d_1(j_2),g_1g_2,t_1t_2m_{i_2j_2}(g_1,g_2))$. The cocycle condition for $m_{i_2j_2}$ ensures that this respects composition. The transformation $\alpha$ is defined as in Remark \ref{rk:easyalpha}: the anafunctors $F := m \circ (m \times id)$ and $F':=m \circ (id \times m)$ are defined over the manifolds
    \begin{align*}
    X &= \{(i_{2}^{1,2},i_2^{12,3},g_1,g_2,g_3) \in I_2^2 \times G^3\,| \: (g_1,g_2) \in U^2_{i_{2}^{1,2}}, \, (g_1g_2,g_3) \in U^2_{i_{2}^{12,3}}\},\\
    X' &= \{(i_{2}^{2,3},i_2^{1,23},g_1,g_2,g_3) \in I_2^2 \times G^3\,| \: (g_2,g_3) \in U^2_{i_{2}^{2,3}}, \, (g_1,g_2g_3) \in U^2_{i_{2}^{1,23}}\}
    \end{align*}
    with maps to $\mathfrak{G}_0^3$ given by 
    \begin{align*}
        \pi(i_{2}^{1,2},i_2^{12,3},g_1,g_2,g_3) &= (d_2(i_2^{1,2}),d_0(i_2^{1,2}),d_2(i_2^{12,3}),g_1,g_2,g_3) \\
        \pi'(i_{2}^{2,3},i_2^{1,23},g_1,g_2,g_3) &= (d_2(i_2^{1,23}),d_2(i_2^{2,3}),d_0(i_2^{2,3}),g_1,g_2,g_3).
    \end{align*}
    Then let $Z = \sqcup_{i_3} U^3_{i_3}$, with $l:Z \rightarrow X$ and $l':Z \rightarrow X'$ given by $l(i_3,g_1,g_2,g_3) = (d_3(i_3),d_1(i_3),g_1,g_2,g_3)$ and $l'(i_3,g_1,g_2,g_3) = (d_0(i_3),d_2(i_3),g_1,g_2,g_3)$. The cocycle equations for $\alpha_{i_3}$ imply that the map $Z \rightarrow \mathfrak{X}_1$, $(i_3,g_1,g_2,g_3) \mapsto (d_1(d_1(i_3)),d_1(d_2(i_3)),\alpha_{i_3}^{-1}(g_1,g_2,g_3))$ is a smooth natural transformation $l^{\ast}F \Rightarrow (l')^{\ast}F'$ whose corresponding transformation $\alpha: m \circ (m \times id) \Rightarrow m \circ (id \times m)$ satisfies the pentagon identity, so we have constructed a Lie 2-group. This construction can be enhanced to an equivalence of bicategories.
\end{proof}

From now on we use the language of multiplicative gerbes, while bearing in mind that these are to be regarded as Lie 2-groups to guide our intuition. For $G$, $T$ any Lie groups with $T$ abelian, we let $Ext(G,BT)$ be the space of multiplicative $T$-gerbes over $G$ up to isomorphism. We also write for the rest of the article $\mathfrak{g}$ and $\mathfrak{t}$ for the Lie algebras of $G$ and $T$, respectively. The following is a classification result that is well-known in the literature at least when $G$ is compact (see for example \cite{ScPr}).

\begin{proposition}\label{prop:class2grps}
    Let $Z := \ker exp_T \subset \mathfrak{t}$. If $T$ is connected, then there is an exact sequence
    \begin{equation}
        H^3_{gr,cont}(G,\mathfrak{t}) \rightarrow Ext(G,BT) \rightarrow H^4(BG,Z) \rightarrow H^4_{gr,cont}(G,\mathfrak{t}),
    \end{equation}
    where $H^{\ast}(BG,Z)$ denotes singular cohomology of the classifying space of $G$. In particular, $Ext(G,BT) = H^4(BG,Z)$ when $G$ is compact. 
\end{proposition}
\begin{proof}
    Consider over $BG_{\bullet}$ the sheaf $C^{\infty}_{T}$ of smooth $T$-valued functions. A good semi-simplicial cover of $BG_{\bullet}$ gives an injective resolution of this sheaf by taking the \v{C}ech resolutions $(\check{C}^{\bullet}(C^{\infty}_{T,G^n},\mathcal{U}_n),\check{\delta})$ and using the maps $\tilde{d}^n_j:I_n \rightarrow I_{n-1}$ to define the sheaf morphisms $\partial^n_j:(d^n_j)^{\ast}\check{C}^{p}(C^{\infty}_{T,G^{n-1}},\mathcal{U}_{n-1}) \rightarrow \check{C}^{p}(C^{\infty}_{T,G^n},\mathcal{U}_n)$. The total cohomology of the double complex $(\check{C}^{\bullet}(C^{\infty}_{T,G^{\bullet}},\mathcal{U}_n),\check{\delta},\delta)$ computes $H^{\ast}(BG_{\bullet},C^{\infty}_T)$ and the cocycle data for multiplicative $T$-gerbes over $G$ gives precisely an element in $H^{3}(BG_{\bullet},C^{\infty}_T)$ which classifies them completely; that is, $H^{3}(BG_{\bullet},C^{\infty}_T) = Ext(G,BT)$. This is valid for any Lie groups $G$, $T$ but when $T$ is connected then $1 \rightarrow Z \rightarrow \mathfrak{t} \stackrel{exp}{\rightarrow} T \rightarrow 1$ is exact and so there is an exact sequence
    \begin{equation}
        H^3(BG_{\bullet},C^{\infty}_{\mathfrak{t}}) \rightarrow H^3(BG_{\bullet},C^{\infty}_{T}) \rightarrow H^4(BG_{\bullet},\underline{Z}) \rightarrow H^4(BG_{\bullet},C^{\infty}_{\mathfrak{t}}),
    \end{equation}
    where $C^{\infty}_{\mathfrak{t}}$ is the sheaf of smooth $\mathfrak{t}$-valued functions on $BG_{\bullet}$. The result follows then from Example \ref{ex:bott} and Theorem \ref{th:bott}.
\end{proof}

Note that the proof of Proposition \ref{prop:class2grps} shows in particular that $Ext(G,BT) = H^{3}(BG_{\bullet},C^{\infty}_T)$.

\begin{definition}\label{def:flatmultger}
    Let $G$, $T$ be Lie groups with $T$ abelian and let $(\mathcal{G},m,\alpha)$ be a multiplicative $T$-gerbe over $G$. We say $(\mathcal{G},m,\alpha)$ is \emph{flat} if its class in $Ext(G,BT) = H^3(BG_{\bullet},C^{\infty}_{T})$ lies in the image of $H^3(BG_{\bullet},\underline{T}) \rightarrow H^3(BG_{\bullet},C^{\infty}_{T})$, for $\underline{T}$ the sheaf of locally constant $T$-valued functions.
\end{definition}

It follows from Proposition \ref{prop:class2grps} that, for connected $T$, a multiplicative $T$-gerbe $\mathcal{G}$ over $G$ has a class $c(\mathcal{G}) \in H^4(BG,Z)$. This has an image 
\begin{equation}\label{eq:drclass}
c_{\mathfrak{t}}(\mathcal{G}) \in H^4(BG,\mathfrak{t}),
\end{equation}
which we call the \emph{de Rham class} of the multiplicative gerbe.

\begin{lemma}\label{lem:drclass}
The group $H^4(BG,\mathfrak{t})$ is isomorphic to the following quotient.
$$\frac{\{(\tau_3,\tau_2,\tau_1,\tau_0) \:|\: \tau_i \in \Omega^i(G^{4-i},\mathfrak{t}), \:d\tau_3 = 0, \,d\tau_2 = - \delta \tau_{3}, \, d\tau_1 =  \delta \tau_{2}, \, d\tau_0 = - \delta \tau_{1},\, \: 0 = \delta \tau_0\}}{\{(d\beta_2,\delta \beta_2 + d\beta_1,-\delta\beta_1 + d\beta_0,\delta \beta_0) \:|\: \beta_i \in \Omega^i(G^{3-i},\mathfrak{t})\}}.$$
Moreover, if $T$ is connected, then
\begin{enumerate}
    \item\label{it:drclass1} The de Rham class $c_{\mathfrak{t}}(\mathcal{G}) \in H^4(BG,\mathfrak{t})$ of a multiplicative $T$-gerbe $\mathcal{G}$ over $G$ admits a representative $[(\tau_3,\tau_2,\tau_1,\tau_0)]$ as above with $\tau_0 = 0$.
    \item\label{it:drclass2} The de Rham class $c_{\mathfrak{t}}(\mathcal{G}) \in H^4(BG,\mathfrak{t})$ of a multiplicative $T$-gerbe $\mathcal{G}$ over $G$ is $0$ if and only if $\mathcal{G}$ is flat.
\end{enumerate}
\end{lemma}
\begin{proof}
    This description of $H^4(BG,\mathfrak{t})$ follows from the variant of the de Rham theorem for simplicial manifolds presented in Example \ref{ex:simpderham}. Now note that the exact sequence of sheaves $0 \rightarrow Z \rightarrow C^{\infty}_{\mathfrak{t}} \rightarrow C^{\infty}_T \rightarrow 0$ induces the exact sequence
    \begin{equation}
        H^3(BG_{\bullet},C^{\infty}_T) \rightarrow H^4(BG_{\bullet},Z) \rightarrow H^4(BG_{\bullet}, C^{\infty}_{\mathfrak{t}}),
    \end{equation}
    which yields \ref{it:drclass1}. Similarly, the exact sequence of sheaves $Z \rightarrow \underline{\mathfrak{t}} \rightarrow \underline{T}$ gives the exact sequence
    \begin{equation}
        H^3(BG,\underline{T}) \rightarrow H^4(BG,Z) \rightarrow H^4(BG,\mathfrak{t}),
    \end{equation}
    which implies \ref{it:drclass2}.
\end{proof}

Given a multiplicative $T$-gerbe over $G$, one representative $c_{\mathfrak{t}}(\mathcal{G}) = [(\tau_3,\tau_2,\tau_1,0)]$ as in part \ref{it:drclass1} of Lemma \ref{lem:drclass} can be obtained by choosing any connective structure and curving on $\mathcal{G}$ and any connection on $m$. This yields the curvature $3$-form $\tau_3$ of $\mathcal{G}$ on $G$, the curvature $2$-form $-\tau_2$ of $m$ on $G^2$ and the $1$-form $\tau_1$ on $G^3$ that measures non-flatness of $\alpha$; these satisfy the above equations.

The following is the main example of a multiplicative gerbe.

\begin{example}\label{ex:string}
For $G$ a compact simple simply connected Lie group, Proposition \ref{prop:class2grps} and the fact that $H^4(BG,\mathbb Z) = H^3(G,\mathbb Z) = \mathbb Z$ in this case imply that multiplicative $U(1)$-gerbes over $G$ are classified by $\mathbb Z$. The multiplicative gerbe corresponding to a choice of generator of $H^4(BG,\mathbb Z)$ is called String$(G)$. The image of such generator in $H^4(BG_{\bullet},\mathbb R)$ can be described in de Rham cohomology by the forms $\tau_3 := \frac{1}{6}\langle \theta^L, [\theta^L \wedge \theta^L] \rangle \in \Omega^3(G,\mathbb R)$, $\tau_2 := \frac{1}{2}\langle g_1^{\ast}\theta^L \wedge g_2^{\ast}\theta^R \rangle \in \Omega^2(G \times G, \mathbb R)$, $\tau_1 = 0$, $\tau_0 = 0$ (which satisfy $d\tau_3 = 0$, $d\tau_2 = \delta \tau_3$, $\delta \tau_2 = 0$), for $\theta^L, \,\theta^R \in \Omega^1(G,\mathfrak{g})$ the left- and right-invariant Maurer-Cartan forms on $G$, respectively, and $\langle \cdot,\cdot \rangle: \mathfrak{g} \otimes \mathfrak{g} \rightarrow \mathbb R$ the Killing form, normalized so that $[\tau_3] \in H^3(G,\mathbb Z) = \mathbb Z$ is a generator. 
\end{example}

It follows from Example \ref{ex:string} that a finite-dimensional cocycle description of String$(G)$ is obtained by choosing potentials for the forms $\tau_3$, $\tau_2$ in a semi-simplicial cover of $G$ but, as this is not a canonical procedure, this does not yield an explicit description of the multiplicative gerbe. The gerbe String$(G) \rightarrow G$ is described explicitly in a cover of $G$ in \cite{Mein}, where an \emph{equivariant structure} on it is also given, but there is no known explicit cocycle data for the product $m$ and the associator $\alpha$. There are, however, explicit models for String$(G)$ as an infinite-dimensional crossed module \cite{BaezCransSchSte,LudewigWald} and one can pass from these to a model in terms of infinite-dimensional bundle gerbes by the following general procedure.

\begin{example}\label{ex:crossedmultger}
    Let $(\tilde{G},H,f,\triangleright)$ be a central Lie crossed module as in Example \ref{ex:crossed} and write $G := \tilde{G}/Im(f)$, $T := \ker(f)$. It follows from Proposition \ref{prop:centralextensions} and Theorem \ref{th:ScPr} that this determines a multiplicative $T$-gerbe over $G$. We recall here an explicit description of such multiplicative gerbe in the language of bundle gerbes (see Remark \ref{rk:bundlegerbes}) that can be found in \cite{NikWalLift}.

    In order to give a $T$-gerbe over $G$ we use the surjective submersion $\tilde{G} \rightarrow G$ and we define a $T$-bundle $L \rightarrow \tilde{G} \times_G \tilde{G}$ by $L := \tilde{G} \times H$, with projection $(g,h) \mapsto (g,f(h)g)$ and $T$ acting on $H$ through the group multiplication; then there is a canonical isomorphism $p_{12}^{\ast}L \otimes p_{23}^{\ast}L \rightarrow p_{13}^{\ast}L$ over $\tilde{G}\times_G\tilde{G}\times_G\tilde{G}$ because
    \begin{align*}
        &p_{13}^{\ast}L = \{(g,g',g'',h^{0,''}) \in \tilde{G}^3 \times H\:| \: g'' = f(h^{0,''})g\}\\
        &p_{12}^{\ast}L \otimes p_{23}^{\ast}L = \{(g,g',g'',h^{0,'},h^{',''}) \in \tilde{G}^3 \times H^2\:| \: g' = f(h^{',''})g, \, g''= f(h^{0,'})g\}/\sim
    \end{align*}
    with $(h^{0,'},h^{',''}) \sim (th^{0,'},t^{-1}h^{',''})$ for $t \in T$ and so we may define $[h^{0,'},h^{',''}] \mapsto h^{',''}h^{0,'}$. This completes the construction of the bundle gerbe $\mathcal{G} \rightarrow G$. Now, in order to give an isomorphism $g_1^{\ast}\mathcal{G} \otimes g_2^{\ast}\mathcal{G} \rightarrow (g_1g_2)^{\ast}\mathcal{G}$, we cover $G \times G$ by $\tilde{G} \times \tilde{G}$ and give an isomorphism between the $T$-bundles over $(\tilde{G} \times \tilde{G}) \times_{G \times G} (\tilde{G} \times \tilde{G})$ that describe the gerbes $g_1^{\ast}\mathcal{G} \otimes g_2^{\ast}\mathcal{G}$ and $(g_1g_2)^{\ast}\mathcal{G}$; these are
    \begin{align*}
        \{(g_1,g_2,g_1',g_2',h_1,h_2) \in \tilde{G}^4 \times H^2 \:| \:g_1' = f(h_1)g_1, \,g_2' = f(h_2)g_2\}/\sim, \\
        \{(g_1,g_2,g_1',g_2',h_{12}) \in \tilde{G}^4 \times H \:| \, g_1'g_2' = f(h_{12})g_1g_2\},
    \end{align*}
    respectively, with the equivalence relation $(h_1,h_2) \sim (th_1,t^{-1}h_2)$ for $t \in T$. The isomorphism is then $(g_1,g_2,g_1',g_2',[h_1,h_2]) \mapsto (g_1,g_2,g_1',g_2',h_1 \cdot g_1 \triangleright h_2)$ (here we use that $g \triangleright t = t$) and for this to define an isomorphism of gerbes it is necessary that this isomorphism is compatible over $(\tilde{G} \times \tilde{G}) \times_{G \times G} (\tilde{G} \times \tilde{G}) \times_{G \times G} (\tilde{G} \times \tilde{G})$ with the gerbe product; this reduces to checking that
    \begin{equation}
        h_1^{',''}h_1^{0,'}g_1 \triangleright (h_2^{',''},h_2^{0,'}) = h_1^{',''} g_1' \triangleright h_2^{',''} \cdot h_1^{0,'} \cdot g_1 \triangleright h_2^{0,'}
    \end{equation}
    when $g_1' = f(h_1^{0,'})g_1$, which follows from the axioms. To conclude the construction of a multiplicative $T$-gerbe over $G$, it only remains to give the $2$-isomorphism $\alpha$ but in this case we can simply take the identity, which follows essentially from
    \begin{equation}
        h_1 \cdot g_1 \triangleright h_2 \cdot (g_1g_2) \triangleright h_3 = h_1 \cdot g_1 \triangleright (h_2 \cdot g_2 \triangleright h_3).
    \end{equation}
\end{example}


\section{Connective structures and complexification of multiplicative gerbes}
\label{sec:constrcomp}

\subsection{Connective structures on multiplicative gerbes}\label{sec:constr}

In this section we define the new notion of \emph{connective structures} on multiplicative gerbes, closely related to the \emph{connections} from \cite{WaldMultCon}. We refer to Appendix \ref{ap:gerbes} for the notion of a \emph{connection on an isomorphism of gerbes}. 

\begin{definition}\label{def:constr}
    Let $(\mathcal{G},m,\alpha)$ be a multiplicative $T$-gerbe over $G$. A \emph{connective structure} on it is the following data:
    \begin{enumerate}
        \item A connective structure $\nabla$ on the gerbe $\mathcal{G} \rightarrow G$,
        \item A connection $\nabla_m$ on the isomorphism of gerbes $m$ such that $\alpha$ is a flat 2-isomorphism.
    \end{enumerate}
    We often write $(\mathcal{G}_{\nabla},m_{\nabla},\alpha)$ for a multiplicative gerbe equipped with a connective structure. An \emph{isomorphism of connective structures on a multiplicative $T$-gerbe} $(\nabla_1,\nabla_{m_1}) \rightarrow (\nabla_2,\nabla_{m_2})$ is an isomorphism of connective structures on gerbes $\phi:\nabla_1 \rightarrow \nabla_2$ such that the following is a commutative diagram of isomorphisms of gerbes with connective structures
    \begin{equation}
        \begin{tikzcd}
             g_1^{\ast}\mathcal{G}_{\nabla_1} \otimes g_2^{\ast}\mathcal{G}_{\nabla_1} \ar[r,"{(m,\nabla_{m_1})}",{name=U}] \ar[d,"{(id,g_1^{\ast}\phi \otimes g_2^{\ast}\phi)}",swap] 
             & (g_1g_2)^{\ast}\mathcal{G}_{\nabla_1}\ar[d,"{(id,(g_1g_2)^{\ast}\phi)}"] \\
             g_1^{\ast}\mathcal{G}_{\nabla_2} \otimes g_2^{\ast}\mathcal{G}_{\nabla_2} \ar[r," {(m,\nabla_{m_2})}",{name=D},swap] 
             & (g_1g_2)^{\ast}\mathcal{G}_{\nabla_2}.
        \end{tikzcd}
        \end{equation}
\end{definition}

Given a cocycle description $\lambda_{i_1j_1k_1}$, $m_{i_2j_2}$, $\alpha_{i_3}$ of the multiplicative gerbe in a good semi-simplicial cover of $BG_{\bullet}$ as in \eqref{eq:cocycle1}, a connective structure on it is then described by 
\begin{equation}\label{eq:cocycle2}
A_{i_1j_1} \in \Omega^1(U^1_{i_1j_1},\mathfrak{t}), \quad M_{i_2} \in \Omega^1(U^2_{i_2},\mathfrak{t})
\end{equation}
satisfying
\begin{align}
\begin{split}\label{eq:cocycleeq2}
    A_{i_1j_1} - A_{i_1k_1} + A_{j_1k_1} &= \lambda_{i_1j_1k_1}^{\ast}\theta^T,\\
    M_{i_2} + d_1^{\ast}A_{d_1(i_2)d_1(j_2)} + m_{i_2j_2}^{\ast}\theta^T &= 
 d_0^{\ast}A_{d_0(i_2)d_0(j_2)} + d_2^{\ast}A_{d_2(i_2)d_2(j_2)} + M_{j_2},\\
    \alpha_{i_3}^{\ast}\theta^T + d_0^{\ast}M_{d_0(i_3)} + d_2^{\ast}M_{d_2(i_3)} &= d_1^{\ast}M_{d_1(i_3)} + d_3^{\ast}M_{d_3(i_3)},
\end{split}
\end{align}
where $\theta^T \in \Omega^1(T,\mathfrak{t})$ is the Maurer-Cartan form on $T$.

A connective structure on a multiplicative gerbe need not exist in general, and different choices might be non-isomorphic. However, in many cases of interest, such as when $G$ is compact, there is one and only one choice up to isomorphism. The whole picture is described by Proposition \ref{prop:classconstr} below (see \cite{WaldMultCon} for similar results, with the difference that an additional curving on the multiplicative gerbe is considered as part of the structure to classify).

For fixed $T$, $G$, let $Ext(G,BT)$ be the space of multiplicative $T$-gerbes over $G$ up to isomorphism and $Ext(G,BT_{\nabla})$ the space of multiplicative $T$-gerbes with connective structure over $G$ up to isomorphism. Then
\begin{proposition}\label{prop:classconstr}
A multiplicative gerbe $\mathcal{G}$ admits a connective structure if and only if its de Rham class \eqref{eq:drclass} admits a representative $(\tau_3,\tau_2,\tau_1,\tau_0)$ with $\tau_1 = 0$, $\tau_0 = 0$. Moreover, there is an exact sequence
\begin{equation}
    H^1_{gr,cont}(G,\mathfrak{g}^{\ast}\otimes \mathfrak{t}) \rightarrow Ext(G,BT_{\nabla}) \rightarrow Ext(G,BT) \rightarrow H^2_{gr,cont}(G,\mathfrak{g}^{\ast}\otimes \mathfrak{t}).
\end{equation}
In particular, $Ext(G,BT_{\nabla}) = Ext(G,BT)$ for compact $G$.
\end{proposition}
\begin{proof}
Representatives $(\tau_3,\tau_2,\tau_1,0)$ of the de Rham class of $\mathcal{G}$ are obtained from taking a connective structure and curving on $\mathcal{G}$ and a connection on $m$; since $\tau_1$ measures the failure of the associator $\alpha$ to preserve the connective structure it is clear that the multiplicative gerbe admits a connective structure if and only if the choices can be made so that $\tau_1 = 0$. Now it follows from the cocycle data above, as in the proof of Proposition \ref{prop:class2grps}, that $Ext(G,BT_{\nabla}) = \mathbb H^3(BG_{\bullet},C^{\infty}_T \rightarrow \Omega^1_{\mathfrak{t}})$, for $C^{\infty}_T \rightarrow \Omega^1_{\mathfrak{t}}$ the complex of sheaves on $BG_{\bullet}$ of smooth $T$-valued functions and smooth $\mathfrak{t}$-valued 1-forms, respectively, with the map $f \mapsto f^{\ast}\theta^T$ between them. Then the above sequence follows from the exact sequence of complexes $0 \rightarrow (0 \rightarrow \Omega^1_{\mathfrak{t}}) \rightarrow (C^{\infty}_T \rightarrow \Omega^1_{\mathfrak{t}}) \rightarrow (C^{\infty}_T \rightarrow 0) \rightarrow 0$ and Theorem \ref{th:bott}. 
\end{proof}

We prove in Theorem \ref{th:MaurerCartan} below that a multiplicative gerbe with connective structure is naturally equipped with more data. In order to state the result we introduce the following original definition.

\begin{definition}\label{def:enhcurv}
    Let $\mathcal{L}_{\nabla} \rightarrow M$ be a $T$-gerbe with connective structure given over the cover $\{U_i\}_{i}$ of $M$ by $\lambda_{ijk}$, $A_{ij}$. An \emph{enhanced curving} is a collection $\{B^{en}_i\}_i$ of $B^{en}_i \in \Gamma(T^{\ast}U_i \otimes T^{\ast}U_i \otimes \mathfrak{t})$ such that $B^{en}_i - B^{en}_j = dA_{ij}$; equivalently, it is a pair $(\{B_i\}_i,h)$ of a curving $\{B_i\}_i$ and a $h \in \Gamma(S^2T^{\ast}M \otimes \mathfrak{t})$. An \emph{enhanced connection} $(\{A_{ij}\}_{i,j},\{B^{en}_i\}_i)$ on a gerbe is a pair of a connective structure and an enhanced curving. The \emph{curvature of an enhanced connection} is the curvature of the underlying connection. If  $(\{s_{ij}\}_{i,j},\{A_i\}_i): (\{\lambda_{ijk}^1,A_{ij}^1\}) \rightarrow (\{\lambda_{ijk}^2,A_{ij}^2\})$ is an isomorphism of gerbes with connective structures and the gerbes have enhanced curvings $B^{en,1}_i$, $B^{en,2}_i$, then the \emph{curvature of the isomorphism} with respect to these enhanced curvings is $F \in \Gamma(T^{\ast}M\otimes T^{\ast}M \otimes \mathfrak{t})$ defined by $F = dA_i - B^{en,1}_i + B^{en,2}_i$.
\end{definition}

The first part of the following theorem, which is proved in \cite{WaldMultCon}, is a generalization of Example \ref{ex:string} in which one considers arbitrary groups $G$, $T$ and in which the Killing form is replaced by any $Ad$-invariant symmetric bilinear form $\langle \cdot,\cdot \rangle: \mathfrak{g} \otimes \mathfrak{g} \rightarrow \mathfrak{t}$. The second part is a converse to this result, which seems to be new. From now on we write $\theta^L, \, \theta^R \in \Omega^1(G,\mathfrak{g})$ for the left- and right-invariant Maurer-Cartan forms on $G$, respectively. Note also that, for connected $T$, the exact sequence of sheaves $0 \rightarrow Z \rightarrow \underline{\mathfrak{t}} \stackrel{exp}{\rightarrow} \underline{T} \rightarrow 0$ induces an exact sequence \begin{equation}
    H^4(BG_{\bullet},Z) \longrightarrow H^4(BG_{\bullet},\underline{\mathfrak{t}}) \stackrel{exp}{\longrightarrow} H^4(BG_{\bullet},\underline{T}).
\end{equation}

\begin{theorem}\label{th:MaurerCartan}
Let $G$, $T$ be Lie groups with $T$ abelian. Given an $Ad$-invariant symmetric bilinear form $\langle \cdot,\cdot \rangle: \mathfrak{g} \otimes \mathfrak{g} \rightarrow \mathfrak{t}$, the differential forms 
\begin{equation}\label{eq:murho}
    \mu := \frac{1}{6}\langle \theta^L \wedge [\theta^L \wedge \theta^L] \rangle \in \Omega^3(G,\mathfrak{t}), \quad \nu := -\langle g_1^{\ast}\theta^L \wedge g_2^{\ast}\theta^R \rangle \in \Omega^2(G^2,\mathfrak{t})
\end{equation}
define a class $[\mu,-\nu,0,0] \in H^4(BG,\mathfrak{t})$. If $T$ is connected and $exp([\mu,-\nu,0,0]) \in H^4(BG_{\bullet},\underline{T})$ vanishes then there is a multiplicative $T$-gerbe with connective structure $(\mathcal{G}_{\nabla},m_{\nabla},\alpha)$ over $G$ whose de Rham class \eqref{eq:drclass} is $c_{\mathfrak{t}}(\mathcal{G}) = [\mu,-\nu,0,0]$. Furthermore, such $(\mathcal{G}_{\nabla},m_{\nabla},\alpha)$ is unique up to tensor product with flat multiplicative gerbes

Conversely, a multiplicative $T$-gerbe  with connective structure $(\mathcal{G}_{\nabla},m_{\nabla},\alpha)$ over $G$ has a unique enhanced curving $\Theta^{L,en} = (\Theta^L,h)$ such that the curvature of $\nabla_m$ with respect to $\Theta^{L,en}$ is $F \in \Gamma(T^{\ast}G^2 \otimes T^{\ast}G^2 \otimes \mathfrak{t})$ determined by $h \in \Gamma(S^2T^{\ast}G \otimes \mathfrak{t})$ as $F = 2h_1(g_1^{\ast}\theta^L \otimes g_2^{\ast}\theta^R)$. Its curvature is $-\frac{1}{6}h_1(\theta^L \wedge [\theta^L \wedge \theta^L] \rangle$. In other words, $(\mathcal{G}_{\nabla},m_{\nabla},\alpha)$ determines the following data:
\begin{enumerate}
    \item\label{it:MCpairing} An $Ad$-invariant symmetric bilinear form $\langle \cdot,\cdot \rangle: \mathfrak{g} \otimes \mathfrak{g} \rightarrow \mathfrak{t}$ such that $exp([\mu,-\nu,0,0]) \in H^4(BG_{\bullet},\underline{T})$ vanishes, where $\mu, \,\nu$ are as in \eqref{eq:murho}.
    \item\label{it:MCcurving} A curving $\Theta^L$ on $\mathcal{G}_{\nabla}$ with curvature $\mu$ such that $\nabla_m$ has curvature $\nu$ with respect to it.
\end{enumerate}
$\Theta^{L,en}$ is called the \emph{Maurer-Cartan enhanced curving} on $\mathcal{G}_{\nabla}$, while $\Theta^L$ is the \emph{Maurer-Cartan curving} on $\mathcal{G}_{\nabla}$. If $T$ is connected, then both constructions are inverse to each other up to tensor product with flat multiplicative gerbes.
\end{theorem}
\begin{proof}
    For simplicity, we prove everything in the case that $T$ is connected. Consider the following exact triangle in the derived category of sheaves of abelian groups over $BG_{\bullet}$:
    \begin{equation*}
    \begin{tikzcd}
        \underline{T} \rar \dar &C^{\infty}_T \dar \rar &0 \dar \ar[r,"{[1]}"] & \, \\
        0  \rar & \Omega^1_{\mathfrak{t}} \ar[r,"d"] &\Omega^2_{\mathfrak{t},d-cl} \ar[r,"{[1]}"] & \, ,
    \end{tikzcd}
    \end{equation*}
    where $\Omega^2_{\mathfrak{t},d-cl}$ is the sheaf of closed $\mathfrak{t}$-valued two-forms. This yields the sequence
    \begin{equation*}
        H^1(BG_{\bullet},\Omega^2_{\mathfrak{t},d-cl}) \rightarrow H^3(BG, \underline{T}) \rightarrow Ext(G,BT_{\nabla}) \rightarrow H^2(BG_{\bullet},\Omega^2_{\mathfrak{t},d-cl}) \rightarrow H^4(BG,\underline{T}). 
    \end{equation*} 
    Consider now the exact sequence $0 \rightarrow \Omega^2_{\mathfrak{t},d-cl} \rightarrow \Omega^2_{\mathfrak{t}} \stackrel{d}{\rightarrow} \Omega^3_{\mathfrak{t}} \rightarrow 0$ and apply Theorem \ref{th:bott} to obtain $H^1(BG_{\bullet},\Omega^2_{\mathfrak{t},d-cl}) = 0$ and $H^2(BG_{\bullet},\Omega^2_{\mathfrak{t},d-cl}) = H^2(BG_{\bullet},\Omega^2_{\mathfrak{t}}) = H^0_{gr,cont}(G,S^2\mathfrak{g}^{\ast} \otimes \mathfrak{t})$; hence,
    \begin{equation}
        0 \rightarrow H^3(BG, \underline{T}) \rightarrow Ext(G,BT_{\nabla}) \rightarrow H^0_{gr,cont}(G,S^2\mathfrak{g}^{\ast} \otimes \mathfrak{t}) \rightarrow H^4(BG,\underline{T}).
    \end{equation}
    The theorem follows from chasing how the maps in this sequence are defined. From the triangle above it is clear that the map $Ext(G,BT_{\nabla}) \rightarrow H^2(BG_{\bullet},\Omega^2_{\mathfrak{t}})$ sends a multiplicative gerbe with connective structure $(\mathcal{G}_{\nabla},m_{\nabla},\alpha)$ to the class of the curvature 2-form  $-\tau_2 \in \Omega^2(G^2,\mathfrak{t})$ of $\nabla_m$ with respect to any choice of curving on $\mathcal{G}_{\nabla}$ (it satisfies $\delta \tau_2 = 0$ as $\alpha$ preserves the connection on $m$). By part \eqref{it:bott2} of Lemma \ref{lem:bott} this determines an $Ad$-invariant symmetric bilinear form $\langle \cdot,\cdot \rangle: \mathfrak{g} \otimes \mathfrak{g} \rightarrow \mathfrak{t}$, characterized by the condition that the curving on $\mathcal{G}_{\nabla}$ can be chosen so that $\tau_2 = \nu$. Let $\Theta^L$ be one such curving. 
    
    The curvature of $\Theta^L$ is some $H \in \Omega^3(G,\mathfrak{t})$ with $\delta H = d\nu$; since $\mu$ satisfies this, we obtain $H = \mu + h$ with $\delta h = 0$, but we see from Theorem \ref{th:bott} and Remark \ref{rk:simpacyclic} that $H^1(BG_{\bullet},\Omega^3_{\mathfrak{t}}) =  \ker(\delta: \Omega^3(G,\mathfrak{t}) \rightarrow \Omega^3(G^2,\mathfrak{t})) = 0$ and so $H = \mu$. Then take $\Theta^{L,en}$ to be given by $\Theta^L$ and the symmetric tensor $h = - \frac{1}{2}\langle \theta^L \odot \theta^L \rangle$; this enhanced curving satisfies the condition above. It is moreover unique with such property, as any other enhanced curving differs from this one by $b \in \Omega^2(G,\mathfrak{t})$ and $h' \in \Gamma(S^2T^{\ast}G \otimes \mathfrak{t})$ but then the curvature condition imposes $h' = \frac{1}{2}\langle \theta^L \odot \theta^L \rangle'$ for some $Ad$-invariant $\langle \cdot,\cdot \rangle'$ and $\delta b = \langle g_1^{\ast}\theta^L \wedge g_2^{\ast}\theta^R \rangle'$; hence, $\langle \cdot,\cdot \rangle' = 0$ and $b = 0$ by part \eqref{it:bott2} of Lemma \ref{lem:bott} and by Theorem \ref{th:bott}, which yields $H^1(BG_{\bullet},\Omega^2_{\mathfrak{t}}) =  \ker(\delta: \Omega^2(G,\mathfrak{t}) \rightarrow \Omega^2(G^2,\mathfrak{t})) = 0$. This concludes the proof of the second part of the theorem. For the first part, note simply that the preceding discussion implies that the map $H^0_{gr,cont}(G,S^2\mathfrak{g}^{\ast}\otimes \mathfrak{t}) \rightarrow H^4(BG,\underline{T})$ factorizes as $H^0_{gr,cont}(G,S^2\mathfrak{g}^{\ast}\otimes \mathfrak{t}) \rightarrow H^4(BG,\mathfrak{t}) \stackrel{exp}{\rightarrow} H^4(BG,\underline{T})$, where the first map is $\langle \cdot,\cdot \rangle \mapsto [\mu,-\nu,0,0]$.
\end{proof}

\begin{remark}\label{rk:pairingformula}
    Given a multiplicative $T$-gerbe with connective structure over $G$ described by cocycle data \eqref{eq:cocycle1}, \eqref{eq:cocycle2}, one can use formula \eqref{eq:bottpairing} from Lemma \ref{lem:bott} to compute the pairing from Theorem \ref{th:MaurerCartan} \eqref{it:MCpairing} as
    \begin{equation}\label{eq:pairing}
        \langle u, v \rangle = \frac{1}{2}dM_{i_2|(g_1,g_2)}(0+ug_2,g_1v+0) + \frac{1}{2}dM_{i_2|(g_1,g_2)}(0+vg_2,g_1u+0),
    \end{equation}
    for any choice of $(g_1,g_2) \in G^2$ and any choice of $i_2 \in I_2$ with $(g_1,g_2) \in U^2_{i_2}$. Alternatively, one can use the cocycle equations \eqref{eq:cocycleeq2} to check directly that this formula gives a well-defined $Ad$-invariant, symmetric pairing. Furthermore, one can prove that the Maurer-Cartan curving is given by the two-forms $\Theta_{i_1}^L \in \Omega^2(U_{i_1}^1,\mathfrak{t})$ defined by
    \begin{equation}
        \Theta^{L}_{i_1|g}(u_g,v_g) = dA_{i_1d_0(i_2)|g}(u_g,v_g) + \frac{1}{2}dM_{i_2|(g^{-1},g)}(0+u_g,v_g^{-1}+v_g) +  \frac{1}{2}dM_{i_2|(g^{-1},g)}(u_g^{-1}+u_g,0+v_g),
    \end{equation}
    as it follows from a tedious but straightforward computation that they satisfy the required properties. Here $i_2 \in I_2$ is any choice of index such that $(g^{-1},g) \in U^2_{i_2}$ and by $u_g^{-1}$ we mean $dinv_g(u_g)$ for $inv: G \rightarrow G$, $g \mapsto g^{-1}$. In particular, these explicit computations yield another proof of the existence of $\langle \cdot,\cdot \rangle$ and $\Theta^L$ in Theorem \ref{th:MaurerCartan} which is also valid when $T$ is not connected. Note also that we can add the two formulas to obtain a formula for $\Theta^{L,en}$:
    \begin{equation}
        \Theta^{L,en}_{i_1|g}(u_g,v_g) = dA_{i_1d_0(i_2)|g}(u_g,v_g) + dM_{i_2|(g^{-1},g)}(u_g^{-1}+u_g,0+v_g),
    \end{equation}
\end{remark}

\begin{corollary}\label{cor:compact}
    Let $G$, $T$ be Lie groups with $G$ compact and $T$ abelian. Any multiplicative $T$-gerbe $(\mathcal{G},m,\alpha)$ over $G$ determines an $Ad$-invariant symmetric bilinear form $\langle \cdot,\cdot \rangle: \mathfrak{g} \otimes \mathfrak{g} \rightarrow \mathfrak{t}$, a connective structure $(\nabla,\nabla_m)$ on $(\mathcal{G},m,\alpha)$ well-defined up to isomorphism and a curving $\Theta^L$ on $\mathcal{G}_{\nabla}$ with curvature $\mu$ and such that $\nabla_m$ has curvature $\nu$, where $\mu$, $\nu$ are as in \eqref{eq:murho}.
\end{corollary}

\begin{proof}
    Straightforward from Proposition \ref{prop:classconstr} and Theorem \ref{th:MaurerCartan}.
\end{proof}

\begin{example}\label{ex:stringconstr}
    Consider the multiplicative $U(1)$-gerbe String$(G)$ over a compact simple Lie group $G$ from Example \ref{ex:string}. By Corollary \ref{cor:compact}, this determines a pairing $\langle \cdot,\cdot \rangle: \mathfrak{g} \otimes \mathfrak{g} \rightarrow \mathbb R$, a connective structure and a Maurer-Cartan curving on String$(G)$. Since any $Ad$-invariant pairing on $\mathfrak{g}$ is a multiple of the Killing form, $\langle \cdot,\cdot \rangle$ is a multiple such that $[\mu] \in H^3(G,\mathbb Z) = \mathbb Z$ is a generator of $H^3(G,\mathbb Z)$. Moreover, \cite{Mein} provides an explicit description of cocycle data for a connective structure and a curving on the gerbe String$(G) \rightarrow G$ with curvature $\mu$ (without a connection on $m$, as there is no known explicit description of $m$ itself).
\end{example}

In \cite{WaldMultCon} a \emph{connection} on a multiplicative gerbe is defined as the same piece of structure as in Definition \ref{def:constr} but with an additional curving on $\mathcal{G}$ (not necessarily preserved by $m$); however, as we see in Theorem \ref{th:MaurerCartan}, the data of a connective structure already determines a canonical curving. 

\subsection{Connections on multiplicative gerbes and adjustments}\label{sec:adj}

This section is a digression of independent interest from the main topic of this article. Here we characterize, in terms of some algebraic data, the category of connective structures on a multiplicative gerbe that arises from a crossed module as in Example \ref{ex:crossedmultger}. The answer relates such connective structures with \emph{adjustments} \cite{RistSaeWolf}. 

\begin{definition}
    Let $(\tilde{G},H,f,\triangleright)$ be a central Lie crossed module and let $G := \tilde{G}/Im(f)$, $T:= Ker(f)$ by $\triangleright$. A \emph{strong adjustment} on $(\tilde{G},H,f,\triangleright)$ is a pair $(s,\kappa)$, where
    \begin{enumerate}
        \item $s: \mathfrak{h} \rightarrow \mathfrak{t}$ is a linear splitting of the inclusion map $\mathfrak{t} \rightarrow \mathfrak{h}$.
        \item $\kappa: \tilde{G} \times \tilde{\mathfrak{g}} \rightarrow \mathfrak{t}$ is linear on $\tilde{\mathfrak{g}}$ and satisfies
        \begin{align}
        \kappa(g_1g_2,v) &= \kappa(g_2, Ad(g_1^{-1})v) + \kappa(g_1,v),\\
        \kappa(f(h),v) &= s(h^{-1} \cdot v \triangleright h), \\
        \kappa(g,f(u)) &= s(g^{-1} \triangleright u - u).
        \end{align}
    \end{enumerate}
    Given two strong adjustments $(s_1,\kappa_1)$, $(s_2,\kappa_2)$, an \emph{isomorphism} between them is a linear map $\phi:\tilde{\mathfrak{g}} \rightarrow \mathfrak{t}$ such that $s_2(u) - s_1(u) = \phi f(u)$ and $\kappa_2(g,v) - \kappa_1(g,v) = \phi(Ad(g^{-1})v-v)$.
\end{definition}

\begin{proposition}\label{prop:adj}
    Let $(\tilde{G},H,f,\triangleright)$ be a central Lie crossed module and let $\mathcal{G}$ be the corresponding multiplicative $T$-gerbe over $G$ as in Example \ref{ex:crossedmultger}. Then the category of connective structures on $\mathcal{G}$ is equivalent to the category of strong adjustments on $(\tilde{G},H,f,\triangleright)$.
\end{proposition}
\begin{proof}
    By unwinding the definitions in this case we see that a connective structure on $\mathcal{G}$ is precisely the data of $\nabla \in \Omega^1(\tilde{G} \times H, \mathfrak{t})$ and $\tau \in \Omega^1(\tilde{G} \times \tilde{G},\mathfrak{t})$ satisfying the following relations (here $\theta^T \in \Omega^1(T,\mathfrak{t})$ is the Maurer-Cartan form on $T$):
    \begin{align}
    \begin{split}
        (g,ht)^{\ast}\nabla - (g,h)^{\ast}\nabla &= t^{\ast}\theta^T, \\
        (g,h)^{\ast}\nabla + (f(h)g,h')^{\ast}\nabla &= (g,h'h)^{\ast}\nabla, \\
        (f(h_1)g_1,f(h_2)g_2)^{\ast}\tau - (g_1,g_2)^{\ast}\tau &= (g_1g_2,h_1 g_1 \triangleright h_2)^{\ast}\nabla - (g_1,h_1)^{\ast}\nabla - (g_2,h_2)^{\ast}\nabla, \\
        (g_1,g_2)^{\ast}\tau + (g_1g_2,g_3)^{\ast}\tau &= (g_1,g_2g_3)^{\ast}\tau + (g_2,g_3)^{\ast}\tau
    \end{split}
    \end{align}
    Moreover, two such connective structures $(\nabla,\tau)$ and $(\nabla',\tau')$ are isomorphic whenever there exists $\sigma \in \Omega^1(\tilde{G},\mathfrak{t})$ such that
    \begin{align}
    \begin{split}
        \nabla' - \nabla = (f(h)g)^{\ast}\sigma - g^{\ast}\sigma, \\
        \tau' - \tau = (g_1g_2)^{\ast}\sigma - g_1^{\ast}\sigma - g_2^{\ast}\sigma.
    \end{split}
    \end{align}
    Then let $(s,\kappa)$ be a strong adjustment. We claim that $\nabla_{(g,h)}^s(v_g+v_h) := s(g^{-1} \triangleright h^{-1}v_h)$, $\tau^{s,\kappa}_{(g_1,g_2)}(v_{g_1}+v_{g_2}) := \kappa(g_2,g_1^{-1}v_{g_1})$ defines a connective structure on the multiplicative gerbe, which boils down to straightforward computations. Similarly, one checks that if $\phi$ is an isomorphism $(s_1,\kappa_1) \rightarrow (s_2,\kappa_2)$ then $\sigma^{\phi}_{g}(v_g) := \phi(g^{-1}v_g)$ is an isomorphism $(\nabla^{s_1},\tau^{s_1,\kappa_1}) \rightarrow (\nabla^{s_2},\tau^{s_2,\kappa_2})$. In fact, if $\sigma: (\nabla^{s_1},\tau^{s_1,\kappa_1}) \rightarrow (\nabla^{s_2},\tau^{s_2,\kappa_2})$ is an arbitrary isomorphism, then $\phi^{\sigma}(v) := \sigma_1(v)$ is an isomorphism $(s_1,\kappa_1) \rightarrow (s_2,\kappa_2)$ and $\sigma = \sigma^{\phi^{\sigma}}$, so we have defined a fully faithful functor from the category of strong adjustments to the category of connective structures. It is also essentially surjective: since any two connective structures on a gerbe are always isomorphic, we may restrict our attention to those $(\nabla,\tau)$ such that $\nabla = \nabla^s$ for a given splitting $s: \mathfrak{h} \rightarrow \mathfrak{t}$. Then one can check that for $\kappa(g,v) := \tau_{(g^{-1},1)}(0+v) + \tau_{(g^{-1},g)}(g^{-1}v + 0)$ we have that $(s,\kappa)$ is a strong adjustment with $(\nabla^s,\tau^{\kappa})$ isomorphic to $(\nabla^s,\tau)$.
\end{proof}

\begin{remark}
    Proposition \ref{prop:adj} restricted to the case of the crossed module $(G,T,f,\triangleright)$ with trivial $f$ and trivial $\triangleright$ coincides precisely with the isomorphism $H^2(BG_{\bullet},\Omega^1_{\mathfrak{t}}) \rightarrow H^1_{gr,cont}(G,\mathfrak{g}^{\ast}\otimes\mathfrak{t})$ from Lemma \ref{lem:bott}. 
\end{remark}

In particular, Proposition \ref{prop:adj} and Theorem \ref{th:MaurerCartan} imply that a strong adjustment $(s,\kappa)$ on a central Lie crossed module $(\tilde{G},H,f,\alpha)$ gives an $Ad$-invariant pairing $\langle \cdot,\cdot \rangle: \mathfrak{g} \otimes \mathfrak{g} \rightarrow \mathfrak{t}$ and a Maurer-Cartan curving $\Theta^L \in \Omega^2(\tilde{G},\mathfrak{t})$, which we can compute using the formulas in Remark \ref{rk:pairingformula}. For this we write first $\partial_g\kappa: T_g\tilde{G} \times \tilde{\mathfrak{g}} \rightarrow \mathfrak{t}$ for the partial derivative of $\kappa$ at $g \in \tilde{G}$, whose main properties are
\begin{align}
\begin{split}\label{eq:delkappa}
        \partial_g\kappa(v_g,v) &= \partial_1\kappa(g^{-1}v_g,Ad(g^{-1})v),\\
        \partial_1\kappa(Ad(g^{-1})u,Ad(g^{-1})v) &= \partial_1\kappa(u,v) - \kappa(g,[u,v]),\\
        \partial_1\kappa(f(u),v) &= s(v \triangleright u) = - \partial_1\kappa(v,f(u)).
\end{split}
\end{align}
Here we are writing $v \triangleright u := \frac{d}{dt}_{|t=0} exp(tv) \triangleright u$, $v \in \tilde{\mathfrak{g}}$, $u \in \mathfrak{h}$ for the Lie algebra action of $\tilde{\mathfrak{g}}$ on $\mathfrak{h}$, which satisfies $f(v \triangleright u) = [v,f(u)]$ and $f(u_1) \triangleright u_2 = [u_1,u_2]$. Then,
    \begin{align}
        \langle u,v \rangle &= \frac{1}{2}(\partial_1 \kappa(u,v) +\partial_1\kappa(v,u)) \label{eq:adjpair}\\
        \Theta_{g}^L(u_g,v_g) &= -\frac{1}{2}(\partial_1 \kappa(g^{-1}u_g,g^{-1}v_g) - \partial_1\kappa(g^{-1}v_g,g^{-1}u_g)) .\label{eq:adjcurv}
    \end{align}
Although $\partial_1\kappa$ is in principle defined over $\tilde{\mathfrak{g}} \otimes \tilde{\mathfrak{g}}$, \eqref{eq:adjpair} is well-defined over $\mathfrak{g} \otimes \mathfrak{g}$ by \eqref{eq:delkappa}. Note that $\langle \cdot,\cdot \rangle$ and $-\Theta^L_1$ are the symmetric and skew-symmetric parts of the tensor $\partial_1\kappa$. In terms of the Maurer-Cartan \emph{enhanced} curving we obtain the simple formula
\begin{equation}
    \Theta^{L,en}_g(u_g,v_g) = -\partial_1\kappa(g^{-1}u_g,g^{-1}v_g) = -\partial_g\kappa(u_g,v_gg^{-1}).
\end{equation}

\begin{remark}\label{rk:coradj}
    In \cite{RistSaeWolf}, an \emph{adjustment} in a general Lie crossed module $(\tilde{G},H,f,\triangleright)$ is defined as a map $\tilde{\kappa}: \tilde{G} \times \tilde{\mathfrak{g}} \rightarrow \mathfrak{h}$, linear in $\tilde{\mathfrak{g}}$, such that
    \begin{align}
        \tilde{\kappa}(g_1g_2,v) &= \tilde{\kappa}(g_2,Ad(g_1^{-1})v) + g_2^{-1} \triangleright \tilde{\kappa}(g_1,v) - \tilde{\kappa}(g_2,f\tilde{\kappa}(g_1,v)),\\
        \tilde{\kappa}(f(h),v) &= h^{-1} \cdot v \triangleright h.
    \end{align}
    for $v \in \tilde{\mathfrak{g}}$, $g_1, \, g_2 \in \tilde{G}$, $h \in H$. If the Lie crossed module is central and we are given a strong adjustment $(s,\kappa)$, then we may obtain an adjustment $\tilde{\kappa}$ as follows. Consider the exact sequence of Lie algebras
    \begin{align}
    0 \rightarrow \mathfrak{t} \rightarrow \mathfrak{h} \stackrel{f}{\rightarrow} \tilde{\mathfrak{g}} \stackrel{\pi}{\rightarrow} \mathfrak{g} \rightarrow 0
    \end{align}
    and the induced short exact sequences
    \begin{align}
        0 \rightarrow \mathfrak{t} \rightarrow \mathfrak{h} \stackrel{f}{\rightarrow} Im(f) \rightarrow 0, \label{eq:crossedses1} \\
        0 \rightarrow Im(f) \rightarrow \tilde{\mathfrak{g}} \stackrel{\pi}{\rightarrow} \mathfrak{g} \rightarrow 0. \label{eq:crossedses2}
    \end{align}
    Choose a linear splitting $l: \mathfrak{g} \rightarrow \tilde{\mathfrak{g}}$ of \eqref{eq:crossedses2} and note that there is a unique linear map $r: \tilde{\mathfrak{g}} \rightarrow \mathfrak{h}$ such that
    \begin{align}
        0 \leftarrow \mathfrak{t} \stackrel{s}{\leftarrow} \mathfrak{h} \stackrel{r}{\leftarrow} \tilde{\mathfrak{g}} \stackrel{l}{\leftarrow} \mathfrak{g} \leftarrow 0. \label{eq:splittinglie2alg}
    \end{align}
    is exact. Then one can easily check that $\tilde{\kappa}(g,v) := \kappa(g,v) + r(Ad(g^{-1})v-v)$ is an adjustment. In fact, any adjustment $\tilde{\kappa}$ such that $\tilde{\kappa}(g,f(u)) = g^{-1} \triangleright u - u$ and such that there exist splittings $s, \,l$ of \eqref{eq:crossedses1}, \eqref{eq:crossedses2} with $rf\kappa(g,v) = r(Ad(g^{-1})v-v)$ arises from a strong adjustment. All the explicit adjustments in \cite{KimSaeTDual} and \cite{RistSaeWolf} satisfy these conditions.
\end{remark}

\subsection{Complex Lie 2-groups}

\emph{Complex Lie 2-groups} are defined as in Definition  \ref{def:2grp}, demanding that $\mathfrak{G}_0$, $\mathfrak{G}_1$ are complex manifolds and that all the maps involved are holomorphic. If $G$ and $T$ are complex Lie groups with $T$ abelian, then central extensions of $G$ by $BT$ as complex Lie 2-groups are equivalent to \emph{holomorphic multiplicative $T$-gerbes over $G$} by a straightforward generalization of Theorem \ref{th:ScPr}. Here we discuss the interaction between holomorphic multiplicative gerbes and connective structures, as well as the complexification of smooth multiplicative gerbes. We refer to Appendix \ref{sec:cxgerbes} for background on holomorphic gerbes and connections on them.

\begin{definition}\label{def:multgercx}
    For $G$, $T$ complex Lie groups with $T$ abelian, a \emph{holomorphic multiplicative $T$-gerbe} over $G$ is a multiplicative $T$-gerbe $(\mathcal{G},m,\alpha)$ over $G$ such that $\mathcal{G}$, $m$ and $\alpha$ are holomorphic. A \emph{compatible (resp. holomorphic) connective structure} on it is a compatible (resp. holomorphic) connective structure on the gerbe $\mathcal{G}$ with a compatible (resp. holomorphic) connection on the isomorphism of gerbes $m$ such that $\alpha$ is a flat 2-isomorphism of gerbes.
\end{definition}

Recall that, for complex vector spaces $(V,J_V)$, $(W,J_W)$, we say that a skew-symmetric $\mathbb R$-multilinear map $\phi:\Lambda^kV \rightarrow W$ is of \emph{type $(p,k-p)$} if its $\mathbb C$-linear extension $\phi^{\mathbb C}:\Lambda^k(V \otimes \mathbb C) \rightarrow W$ is zero outside $\Lambda^p(V^{1,0}) \otimes \Lambda^{k-p}(V^{0,1}) \subset \Lambda^k(V \otimes \mathbb C)$, where $V^{1,0}$ and $V^{0,1}$ are the $i$ and $-i$ eigenspaces of $J_V^{\mathbb C}: V \otimes \mathbb C \rightarrow V \otimes \mathbb C$, respectively. The type of a symmetric map is similarly defined.

In terms of cocycle data \eqref{eq:cocycle1} in a good semi-simplicial cover of $BG_{\bullet}$, a holomorphic multiplicative $T$-gerbe over $G$ is a multiplicative $T$-gerbe for which $\lambda_{i_1j_1k_1}$, $m_{i_2j_2}$, $\alpha_{i_3}$ can be chosen to be holomorphic. In terms of cocycle data \eqref{eq:cocycle2}, a compatible connective structure is one for which $A_{i_1j_1}$, $M_{i_2}$ can be chosen to be of type $(1,0)$ and a holomorphic connective structure is one for which they can be chosen to be of type $(1,0)$ and to satisfy $\bar{\partial} A_{i_1j_1} = 0$, $\bar{\partial} M_{i_2} = 0$.

For a smooth multiplicative gerbe $\mathcal{G}$, we say that it \emph{admits} a holomorphic structure (with compatible or holomorphic connective structure) if it is isomorphic as a smooth multiplicative gerbe to the underlying smooth multiplicative gerbe of a holomorphic multiplicative gerbe (with compatible or holomorphic connective structure). The following proposition characterizes existence of holomorphic structures on multiplicative gerbes.


\begin{proposition}[\cite{Upm}]\label{prop:class2grpscx}
    Let $G$, $T$ be complex Lie groups with $T$ abelian and connected, let $Z = \ker \, exp_T$ and let $\mathcal{G}$ be a smooth multiplicative $T$-gerbe  over $G$. Then
    \begin{enumerate}
        \item $\mathcal{G}$ admits a holomorphic structure if and only if its de Rham class \eqref{eq:drclass} admits a representative $(\tau_3,\tau_2,\tau_1,0)$ with $\tau_3^{0,3} = 0$, $\tau_2^{0,2} = 0$, $\tau_1^{0,1} = 0$.
        \item $\mathcal{G}$ admits a holomorphic structure with compatible connective structure if and only if its de Rham class \eqref{eq:drclass} admits a representative $(\tau_3,\tau_2,\tau_1,0)$ with $\tau_3^{0,3} = 0$, $\tau_2^{0,2} = 0$, $\tau_1 = 0$.
        \item $\mathcal{G}$ admits a holomorphic structure with holomorphic connective structure if and only if its de Rham class \eqref{eq:drclass} admits a representative $(\tau_3,\tau_2,\tau_1,0)$ with $\tau_3^{1,2+0,3} = 0$, $\tau_2^{1,1+0,2} = 0$, $\tau_1 = 0$.
    \end{enumerate}
\end{proposition}
\begin{proof}
    Straightforward by Proposition \ref{prop:dolbeaultgerbes}.
\end{proof}

As in the smooth case, the pairing associated to a connective structure plays an important role, allowing for an intuitive classification of holomorphic multiplicative gerbes. From now on we write $J_{\mathfrak{g}}: \mathfrak{g} \rightarrow \mathfrak{g}$ and $J_{\mathfrak{t}}: \mathfrak{t} \rightarrow \mathfrak{t}$ for the complex structures on the Lie algebras of $G$ and $T$, respectively.
\begin{proposition}\label{prop:pairingcx}
Let $G$, $T$ be complex Lie groups with $T$ abelian and let $(\mathcal{G}_{\nabla},m_{\nabla},\alpha)$ be a smooth multiplicative $T$-gerbe with connective structure over $G$. Then
    \begin{enumerate}
        \item\label{it:pairingcx1} $(\mathcal{G}_{\nabla},m_{\nabla},\alpha)$ admits a holomorphic structure with compatible connective structure if and only if the pairing $\langle \cdot,\cdot \rangle:\mathfrak{g} \otimes \mathfrak{g} \rightarrow \mathfrak{t}$ from Theorem \ref{th:MaurerCartan} satisfies $\langle \cdot,\cdot \rangle^{0,2} = 0$. In this case, there is a unique such holomorphic structure with compatible connective structure up to holomorphic isomorphism with compatible connection and the Maurer-Cartan curving is compatible with it. 
        \item\label{it:pairingcx2} $(\mathcal{G}_{\nabla},m_{\nabla},\alpha)$ admits a holomorphic structure with holomorphic connective structure if and only if the pairing $\langle \cdot,\cdot \rangle:\mathfrak{g} \otimes \mathfrak{g} \rightarrow \mathfrak{t}$ from Theorem \ref{th:MaurerCartan} satisfies $\langle \cdot,\cdot \rangle^{1,1+0,2} = 0$. In this case, there is a unique such holomorphic structure with holomorphic connective structure up to holomorphic isomorphism with holomorphic connection and the Maurer-Cartan curving is holomorphic. 
    \end{enumerate}
\end{proposition}
\begin{proof}
    We only prove part \ref{it:pairingcx1}, as part \ref{it:pairingcx2} is similar. If $(\mathcal{G}_{\nabla},m_{\nabla},\alpha)$ is holomorphic with compatible connective structure, then take cocycle data \eqref{eq:cocycle2} such that $A_{i_1j_1}$, $M_{i_2}$ are of type $(1,0)$ and compute $\Theta^L$, $\langle \cdot,\cdot \rangle$ as in Remark \ref{rk:pairingformula}; then $(\Theta^L_{i_1})^{0,2} = 0$ and $\langle \cdot,\cdot \rangle^{0,2} = 0$ as $dA_{i_1j_1}^{0,2} = 0$, $dM_{i_2}^{0,2} = 0$. In particular, the holomorphic structure with compatible connective structure of $(\mathcal{G}_{\nabla},m_{\nabla},\alpha)$ is recovered from smooth cocycle data \eqref{eq:cocycle2} for it, as this determines $\Theta^L$. This is because the $(0,2)$-part of $\Theta^L$ and the $(0,1)$-part of $M_{i_2}$ in any smooth trivialization yield a 1-semiconnection for $\mathcal{G}$ and a 1-semiconnection for $m$ inducing by Proposition \ref{prop:dolbeaultgerbes} the holomorphic structure that we started with and with which the connective structure is compatible. So a smooth multiplicative gerbe with connective structure can admit at most one holomorphic structure with compatible connective structure, and the vanishing of the $(0,2)$-part of its pairing $\langle \cdot,\cdot \rangle$ is a necessary condition for this to happen. It is also sufficient by Proposition \ref{prop:class2grps}, as in this case the corresponding $\mu$, $\nu$ from \eqref{eq:murho} satisfy $\mu^{0,3} = 0$ and $\nu^{0,2} = 0$.
\end{proof}


Recall how the fibrewise complexification $\mathcal{L}^{\mathbb C}$ of a gerbe $\mathcal{L}$ is defined in \ref{def:complexificationgerbes}. Upmeier \cite{Upm} used the theory of Stein manifolds to construct a holomorphic $\mathbb C^{\ast}$-gerbe with holomorphic connective structure over $GL(n,\mathbb C)$ that restricts over $U(n)$ to the fibrewise complexification of the $U(1)$-gerbe String$(U(n))$. The following theorem provides a generalization of that construction.

\begin{theorem}\label{th:complexification}
Let $K$, $T_{\mathbb R}$ be compact, connected Lie groups with $T_{\mathbb R}$ abelian and let $j_T:T_{\mathbb R} \rightarrow T$, $j_K:K \rightarrow G$ be their complexifications. For $\mathcal{K}$ any $T_{\mathbb R}$-multiplicative gerbe over $K$ there is a unique holomorphic multiplicative $T$-gerbe with holomorphic connective structure $\mathcal{G}_{\nabla}$ over $G$ such that $j_K^{\ast}\mathcal{G} = \mathcal{K}^{\mathbb C}$ as smooth multiplicative $T$-gerbes over $K$. We call $\mathcal{G}_{\nabla}$ the \emph{complexification} of $\mathcal{K}$.
\end{theorem}
\begin{proof}
    Let $\langle \cdot,\cdot \rangle: \mathfrak{k} \otimes \mathfrak{k} \rightarrow \mathfrak{t}_{\mathbb R}$ be the pairing associated to $\mathcal{K}$ via Corollary \ref{cor:compact}, and let $\langle \cdot,\cdot \rangle_{\mathbb C}: \mathfrak{g} \otimes \mathfrak{g} \rightarrow \mathfrak{t}$ be its complexification. If $\mathcal{G}_{\nabla}$ as in the theorem exists, then by Proposition \ref{prop:pairingcx} its pairing $\mathfrak{g} \otimes \mathfrak{g} \rightarrow \mathfrak{t}$ must be $\langle \cdot,\cdot \rangle_{\mathbb C}$, as it must be $\mathbb C$-linear and restrict to $\langle \cdot,\cdot \rangle$ on $\mathfrak{k}$, so let us show that there is a $\mathcal{G}_{\nabla}$ with pairing $\langle\cdot,\cdot \rangle_{\mathbb C}$. By Theorem \ref{th:MaurerCartan}, this happens if and only if the forms $\mu_{\mathbb C} := \frac{1}{6}\langle \theta^L \wedge [\theta^L \wedge \theta^L ] \rangle_{\mathbb C}$, $\nu_{\mathbb C} := -\langle g_1^{\ast}\theta^L \wedge g_2^{\ast}\theta^R \rangle_{\mathbb C}$ determine $[\mu_{\mathbb C},-\nu_{\mathbb C},0,0] \in H^4(BG,\mathfrak{t})$ such that $exp([\mu_{\mathbb C},-\nu_{\mathbb C},0,0]) = 0 \in H^4(BG,\underline{T})$. Now the inclusion $j_K: K \rightarrow G$ is a homotopy equivalence, therefore there is a commutative diagram
    \begin{equation*}
        \begin{tikzcd}
             H^{4}(BG,\mathfrak{t}) \ar[r,"exp"] \ar[d,"{j_K^{\ast}}",swap] 
             & H^{4}(BG,\underline{T}) \ar[d,"{j_K^{\ast}}",swap]  \\
             H^{4}(BK,\mathfrak{t}) \ar[r,"exp"] 
             & H^{4}(BK,\underline{T})
        \end{tikzcd}
        \end{equation*}
    where the vertical arrows are isomorphisms, and it is clear that $j_K^{\ast}[\mu_{\mathbb C},-\nu_{\mathbb C},0,0] = [\mu,-\nu,0,0]$ for $\mu$, $\nu$ given by \eqref{eq:murho}. Since $exp([\mu,-\nu,0,0]) = 0$ it follows that $exp([\mu_{\mathbb C},\nu_{\mathbb C},0,0]) = 0$ and so there is a smooth multiplicative $T$-gerbe with connective structure $\mathcal{G}_{\nabla}$ over $G$ whose associated pairing is $\langle \cdot,\cdot \rangle_{\mathbb C}$. Moreover, it follows from Proposition \ref{prop:pairingcx} that $\mathcal{G}_{\nabla}$ has one and only one holomorphic structure with holomorphic connective structure. In principle, Theorem \ref{th:MaurerCartan} implies that $\mathcal{G}_{\nabla}$ as a smooth gerbe is only determined by $\langle \cdot,\cdot \rangle_{\mathbb C}$ up to flat gerbes, i.e., classes in $H^3(BG,\underline{T})$, but again since $j^{\ast}_K: H^3(BG,\underline{T}) \rightarrow H^3(BK,\underline{T})$ is an isomorphism, this dependence is fixed by imposing $j^{\ast}\mathcal{G} = \mathcal{K}^{\mathbb C}$.
\end{proof}

\begin{remark}
    Theorem \ref{th:complexification} is also true for $K$ non compact and not connected and $T$ non compact, as long as $K$ and $T_{\mathbb R}$ are Lie groups admitting complexifications $G$ and $T$ with $j_K: K \rightarrow G$ a homotopy equivalence and $\mathcal{K}$ is a multiplicative $T_{\mathbb R}$-gerbe with connective structure over $K$.
\end{remark}

The main example of interest in this article is the complexification of the 2-groups String$(K)$.

\begin{example}\label{ex:stringcx}
    For $G$ a complex reductive Lie group with compact form $K$, Brylinski \cite{BryCx2Grp} constructs, for each $l \in \mathbb Z = H^4(BK,\mathbb Z)$, a holomorphic $\mathbb C^{\ast}$-gerbe over $G$ that restricts over $K$ to the fibrewise complexification of the equivariant bundle gerbe at level $l$ from Example \ref{ex:string} (see also the construction in \cite{Upm}). It follows from Theorem \ref{th:complexification} that this gerbe admits a unique holomorphic multiplicative structure with holomorphic connective structure, which is the complexification of String$(K)$. As in the smooth case (cf. Example \ref{ex:string}), there is no known explicit cocycle description of the multiplicative structure but there is an explicit equivariant structure in the original work of Brylinski \cite{BryCx2Grp}.
\end{example}

\section{Principal 2-bundles and connections}\label{sec:2pbsandcon}

\subsection{Principal 2-bundles}\label{sec:2pbs}
Along the lines of, for example, \cite{NikWal2pbs}, this is the natural definition of a principal bundle for a Lie 2-group:

\begin{definition}\label{def:2pb}
    Let $\mathfrak{G}$ be a Lie 2-group and let $\mathfrak{P}$ be a Lie groupoid. An \emph{action} of $\mathfrak{G}$ on $\mathfrak{P}$ is a smooth anafunctor $\rho: \mathfrak{P} \times \mathfrak{G} \rightarrow \mathfrak{P}$ with a smooth transformation $\alpha^{\rho}: \rho \circ (\rho \times id) \Rightarrow \rho \circ (id \times m): \mathfrak{P} \times \mathfrak{G}  \times \mathfrak{G} \rightarrow \mathfrak{P}$ such that 
    \begin{equation}
        \begin{tikzcd}
            \mathfrak{P} \times \mathfrak{G} \times \mathfrak{G} \times \mathfrak{G} \ar[rr,bend left=70,"{((pg_1)g_2)g_3}"{name=F}] 
              \ar[rr,bend left=20,"{(p(g_1g_2))g_3}"{name=G,2cell}]
              \ar[rr,bend right=20,"{p((g_1g_2)g_3)}"{name=H,2cell}]
              \ar[rr,bend right=70,"{p(g_1(g_2g_3))}"{name=E,swap}] & & 
                \mathfrak{P} \\
             & \ar[Rightarrow,from=F,to=G,"{\alpha^{\rho} }"{swap,pos=0.3},shorten >=1.5pt] \ar[Rightarrow,from=G,to=H,"{\alpha^{\rho}}"{swap}] \ar[Rightarrow,from=H,to=E,"{\alpha}"{swap}] &
        \end{tikzcd} = 
        \begin{tikzcd}
            \mathfrak{P} \times \mathfrak{G} \times \mathfrak{G} \times \mathfrak{G} \ar[rr,bend left=50,"{((pg_1)g_2)g_3}"{name=F}] 
              \ar[rr,"{(pg_1)(g_2g_3)}"{name=G,2cell}]
              \ar[rr,bend right=50,"{p(g_1(g_2g_3))}"{name=H,swap}] & & 
                \mathfrak{P} \\
             & \ar[Rightarrow,from=F,to=G,"{\alpha^{\rho} }"{swap,pos=0.3},shorten >=1.5pt] \ar[Rightarrow,from=G,to=H,"{\alpha^{\rho}}"{swap}] & 
        \end{tikzcd}.
        \end{equation}
   A \emph{principal 2-bundle} with structure 2-group $\mathfrak{G}$ over a manifold $M$ is a Lie groupoid $\mathfrak{P}$ with a smooth functor $\pi: \mathfrak{P} \rightarrow M$ that is a surjective submersion on objects and an action $(\rho,\alpha^{\rho})$ of $\mathfrak{G}$ on $\mathfrak{P}$ such that $\pi \circ \rho = \pi \circ p_1: \mathfrak{P} \times \mathfrak{G} \rightarrow M$ and such that the anafunctor $p_2 \times \rho: \mathfrak{P} \times \mathfrak{G} \rightarrow \mathfrak{P} \times_M \mathfrak{P}$ has a quasi-inverse. \emph{Isomorphisms and 2-isomorphisms} of principal 2-bundles are defined similarly as in Definition \ref{def:2grp}, yielding a bicategory.
\end{definition}

In this article we only consider Lie 2-groups arising from multiplicative gerbes and so we give an equivalent definition in this context, analogous to Definition \ref{def:multger} and based on the definitions of string structures as trivializations of Chern-Simons 2-gerbes in \cite{Bunke,CarMultGer,WaldString}.

\begin{definition}\label{def:2pbmult}
Let $(\mathcal{G},m,\alpha)$ be a multiplicative $T$-gerbe over $G$. A \emph{principal $\mathcal{G}$-bundle} $(P,\mathcal{P},\rho,\alpha^{\rho})$ over a manifold $M$ is the following data: 
\begin{enumerate}
    \item A principal $G$-bundle $P \rightarrow M$.
    \item A $T$-gerbe $\mathcal{P} \rightarrow P$.
    \item An isomorphism of $T$-gerbes $\rho: p^{\ast}\mathcal{P} \otimes g^{\ast}\mathcal{G} \rightarrow (pg)^{\ast}\mathcal{P}$ over $P \times G$.
    \item A 2-isomorphism of $T$-gerbes over $P \times G \times G$
    \begin{equation}
    \begin{tikzcd}
             p^{\ast}\mathcal{P} \otimes g_1^{\ast}\mathcal{G} \otimes g_2^{\ast}\mathcal{G} \ar[r,"{(p,g_1)^*\rho}",{name=U}] \ar[d,"{(g_1,g_2)^{\ast}m}",swap] 
             & (pg_1)^{\ast}\mathcal{P}\otimes g_2^{\ast}\mathcal{G} \ar[d,"{(pg_1,g_2)^{\ast}\rho}"] \ar[Rightarrow, dl, "\alpha^{\rho}"]\\
             p^{\ast}\mathcal{P} \otimes (g_1g_2)^{\ast}\mathcal{G} \ar[r,"{(p,g_1g_2)^{\ast}\rho}",{name=D},swap] 
             & (pg_1g_2)^{\ast}\mathcal{P}
    \end{tikzcd}
\end{equation}
such that, over $P \times G \times G \times G$, we have 
\begin{equation}
\adjustbox{scale=0.85,center}{
        \begin{tikzcd}
            p^{\ast}\mathcal{P} \otimes g_1^{\ast}\mathcal{G} \otimes g_2^{\ast}\mathcal{G} \otimes g_3^{\ast}\mathcal{G} \ar[rr,bend left=70,"{\rho(((pg_1)g_2)g_3)}"{name=F}] 
              \ar[rr,bend left=20,"{\rho((p(g_1g_2))g_3)}"{name=G,2cell}]
              \ar[rr,bend right=20,"{\rho(p((g_1g_2)g_3))}"{name=H,2cell}]
              \ar[rr,bend right=70,"{\rho(p(g_1(g_2g_3)))}"{name=E,swap}] & & 
                (pg_1g_2g_3)^{\ast}\mathcal{P} \\
             & \ar[Rightarrow,from=F,to=G,"{ \alpha^{\rho}}"{swap,pos=0.3},shorten >=1.5pt] \ar[Rightarrow,from=G,to=H,"{\alpha^{\rho}}"{swap}] \ar[Rightarrow,from=H,to=E,"{\alpha}"{swap}] &
        \end{tikzcd} = 
        \begin{tikzcd}
            p^{\ast}\mathcal{P} \otimes g_1^{\ast}\mathcal{G} \otimes g_2^{\ast}\mathcal{G} \otimes g_3^{\ast}\mathcal{G} \ar[rr,bend left=50,"{\rho(((pg_1)g_2)g_3)}"{name=F}] 
              \ar[rr,bend left = 20,"{\rho((pg_1)(g_2g_3))}"{name=G,2cell}]
              \ar[rr,bend right=20,"{\rho(p(g_1(g_2g_3)))}"{name=H,swap}] & & 
                (pg_1g_2g_3)^{\ast}\mathcal{P} \\
             & \ar[Rightarrow,from=F,to=G,"{\alpha^{\rho} }"{swap,pos=0.3},shorten >=1.5pt] \ar[Rightarrow,from=G,to=H,"{\alpha^{\rho}}"{swap}] & 
        \end{tikzcd}.
}
        \end{equation}
\end{enumerate}
Given $(P^i,\mathcal{P}^i,\rho^i,\alpha^{\rho_i})$, $i=1,2$, then an \emph{isomorphism of $\mathcal{G}$-bundles} is the following data:
\begin{enumerate}
    \item An equivariant map $u: P^1 \rightarrow P^2$ 
    \item An isomorphism of $T$-gerbes $\varphi:\mathcal{P}^1 \rightarrow u^{\ast}\mathcal{P}_2$ over $P_1$
    \item A $2$-isomorphism of $T$-gerbes over $P^1 \times G$
    \begin{equation}
    \begin{tikzcd}
             (p^1)^{\ast}\mathcal{P}^1 \otimes g^{\ast}\mathcal{G} \ar[r,"{(p^1,g)^*\rho^1}",{name=U}] \ar[d,"(p^1)^{\ast}\varphi",swap] 
             & (p^1g)^{\ast}\mathcal{P}^1 \ar[d,"{(p^1g)^{\ast}\varphi}"] \ar[Rightarrow, dl, "\alpha^{\varphi}"]\\
             u^{\ast}\mathcal{P}^2 \otimes g^{\ast}\mathcal{G} \ar[r,"{(u,g)^{\ast}\rho^2}",{name=D},swap] 
             & (u(p^1)g)^{\ast}\mathcal{P}^2
    \end{tikzcd}
    \end{equation}
    such that, over $P^1 \times G \times G$, 
    \begin{equation}
    \adjustbox{scale=0.85,center}{
        \begin{tikzcd}
            p_1^{\ast}\mathcal{P}^1 \otimes g_1^{\ast}\mathcal{G} \otimes g_2^{\ast}\mathcal{G} \ar[rr,bend left=70,"{\varphi((p_1g_1)g_2)}"{name=F}] 
              \ar[rr,bend left=20,"{\varphi(p_1g_1)g_2}"{name=G,2cell}]
              \ar[rr,bend right=20,"{(\varphi(p_1)g_1)g_2}"{name=H,2cell}]
              \ar[rr,bend right=70,"{\varphi(p_1)(g_1g_2)}"{name=E,swap}] & & 
                (u(p_1)g_1g_2)^{\ast}\mathcal{P}^2 \\
             & \ar[Rightarrow,from=F,to=G,"{\alpha^{\varphi}}"{swap,pos=0.3},shorten >=1.5pt] \ar[Rightarrow,from=G,to=H,"{\alpha^{\varphi}}"{swap}] \ar[Rightarrow,from=H,to=E,"{\alpha^{\rho_2}}"{swap}] &
        \end{tikzcd} = 
        \begin{tikzcd}
            p_1^{\ast}\mathcal{P}^1 \otimes g_1^{\ast}\mathcal{G} \otimes g_2^{\ast}\mathcal{G} \ar[rr,bend left=50,"{\varphi((p_1g_1)g_2)}"{name=F}] 
              \ar[rr,bend left = 20,"{\varphi(p_1(g_1g_2))}"{name=G,2cell}]
              \ar[rr,bend right=20,"{\varphi(p_1)(g_1g_2)}"{name=H,swap}] & & 
                (u(p_1)g_1g_2)^{\ast}\mathcal{P}^2  \\
             & \ar[Rightarrow,from=F,to=G,"{\alpha^{\rho_1}}"{swap,pos=0.3},shorten >=1.5pt] \ar[Rightarrow,from=G,to=H,"{\alpha^{\varphi}}"{swap}] & 
        \end{tikzcd}.
    }
    \end{equation}
\end{enumerate}
Given $(u,\varphi,\alpha^{\varphi}), (u',\varphi',\alpha^{\varphi'}): (P^1,\mathcal{P}^1,\rho^1,\alpha^{\rho^1}) \rightarrow (P^2,\mathcal{P}^2,\rho^2,\alpha^{\rho^2})$, then a \emph{$2$-isomorphism} between them can only exist if $u = u'$ and is then given by a 2-isomorphism $\psi: \varphi \Rightarrow \varphi'$ such that, over $P_1 \times G$,  
\begin{equation}
\adjustbox{scale=0.9,center}{
        \begin{tikzcd}
            p_1^{\ast}\mathcal{P}^1 \otimes g^{\ast}\mathcal{G} \ar[rr,bend left=70,"{\varphi(p_1g)}"{name=F}] 
              \ar[rr,bend left=20,"{\varphi(p_1)g}"{name=G,2cell}]
              \ar[rr,bend right=20,"{\varphi'(p_1)g}"{name=H,2cell}]
             & & (u(p_1)g)^{\ast}\mathcal{P}^2 \\ \ar[Rightarrow,from=F,to=G,"{\alpha^{\varphi}}"{swap,pos=0.3},shorten >=1.5pt] \ar[Rightarrow,from=G,to=H,"{\psi}"{swap}] &
        \end{tikzcd} = 
        \begin{tikzcd}
            p_1^{\ast}\mathcal{P}^1 \otimes g^{\ast}\mathcal{G} \ar[rr,bend left=50,"{\varphi(p_1g)}"{name=F}] 
              \ar[rr,bend left = 20,"{\varphi'(p_1g)}"{name=G,2cell}]
              \ar[rr,bend right=20,"{\varphi'(p_1)g}"{name=H,swap}] & & 
                (u(p_1)g_1g_2)^{\ast}\mathcal{P}^2  \\
              \ar[Rightarrow,from=F,to=G,"{\psi}"{swap,pos=0.5},shorten >=1.5pt] \ar[Rightarrow,from=G,to=H,"{\alpha^{\varphi'}}"{swap}] & 
        \end{tikzcd}.
}
\end{equation}

\end{definition}

We often abbreviate all the data $(P,\mathcal{P},\rho,\alpha^{\rho})$ of a $\mathcal{G}$-bundle by $(P,\mathcal{P})$ or $\mathcal{P}$.

\begin{proposition}\label{prop:2pbmult}
    Let $\mathcal{G}$ be a multiplicative $T$-gerbe over $G$ with associated Lie 2-group $\mathfrak{G}$. There is a canonical equivalence of bicategories between $\mathcal{G}$-bundles in the sense of Definition \ref{def:2pbmult} and $\mathfrak{G}$-bundles in the sense of Definition \ref{def:2pb}.
\end{proposition}
\begin{proof}
    The proof is analogous to that of Theorem \ref{th:ScPr}.
\end{proof}

There are two ways of giving cocycle data for principal 2-bundles: in a good semi-simplicial cover of the simplicial manifold $P/\!/G_{\bullet}$ obtained as the nerve of the action groupoid $P/\!/G$, and in a good cover of $M$. We present both descriptions, since each one is useful for different purposes. In either case, we must choose first cocycle data \eqref{eq:cocycle1} $\lambda_{i_1j_1k_1}$, $m_{i_2j_2}$, $\alpha_{i_3}$ for the multiplicative gerbe in a good semi-simplicial cover $\mathcal{U}_{\bullet}=(\{U^n_{i_n}\}_{i_n \in I_n},\{d^n_j\}_{n \in \mathbb N, 0 \leq j \leq n})$ of $BG_{\bullet}$.

\subsubsection{Cocycle data on a good semi-simplicial cover of $P/\!/G_{\bullet}$}\label{sec:cocycle2pbspg}
Consider the semi-simplicial manifold $(P /\!/ G)_{\bullet}$ that is obtained as the nerve of the action groupoid $P/\!/G$. This is a semi-simplicial manifold whose geometric realization is $M$, defined by $(P /\!/ G)_{n} := P \times G^n$ with face maps $d^n_j: P \times G^n \rightarrow P \times  G^{n-1}$, $0 \leq j \leq n$ given by 
\begin{align*}
    d^n_0(p,g_1,...,g_n) &= (pg_1,...,g_n),\\
    d^n_1(p,g_1,...,g_n) &= (p,g_1g_2,g_3,...,g_n),\\
    &...\\
    d^n_{n-1}(p,g_1,...,g_n) &= (p,g_1,...,g_{n-1}g_n),\\
    d^n_n(p,g_1,...,g_n) &= (p,g_1,...,g_{n-1}).
\end{align*}
There is a map of semi-simplicial manifolds $(P/\!/G)_{\bullet} \rightarrow BG_{\bullet}$ given at each level simply by the projections $P \times G^n \rightarrow G^n$. Then it follows directly from Definition \ref{def:2pbmult} that a $\mathcal{G}$-bundle $(P,\mathcal{P},\rho,\alpha^{\rho})$ can be described with the following data:
\begin{enumerate}
    \item A good semi-simplicial cover $\mathcal{V}_{\bullet}$ of $(P /\!/ G)_{\bullet}$, that is, a family $\{\mathcal{V}_n\}_{n \in \mathbb N}$ of good covers $\mathcal{V}_n = \{V^n_{r_n}\}_{r_n \in R_n}$ of $P \times G^{n}$ indexed by sets $R_n$ and with maps $d^n_j: R_n \rightarrow R_{n-1}$, $0 \leq j \leq n$ such that $d^n_j(V^n_{r_n}) \subset V^{n-1}_{d^n_j(r_n)}$ and that $(\{R_n\}_{n \in \mathbb N},\{d^n_j\}_{n \in \mathbb N, 0 \leq j \leq n})$ forms a semi-simplicial set.
    \item A map of semi-simplicial manifolds $\pi:\mathcal{V} \rightarrow \mathcal{U}$ covering the map $(P /\!/ G)_{\bullet} \rightarrow BG_{\bullet}$, that is, maps $\pi_n: R_n \rightarrow I_n$ such that $\pi_n(V_{r_n}^n) \subset U^n_{\pi_n(r_n)}$ and $d^n_j \circ \pi_n = \pi_n \circ d^n_j$.
    \item Cocycle data 
    \begin{equation}
        b_{r_0s_0t_0}: V_{r_0s_0t_0}^0 \rightarrow T, \quad act_{r_1s_1}:V_{r_1s_1}^1 \rightarrow T, \quad \beta_{r_2}: V_{r_2}^2 \rightarrow T
    \end{equation} such that
    \begin{align}
    \begin{split}
        &b_{r_0s_0t_0}(p)b_{r_0s_0u_0}^{-1}(p) b_{r_0t_0u_0}(p)b_{s_0t_0u_0}^{-1}(p) = 1,\\
        &act_{r_1s_1}(p,g)act_{r_1t_1}^{-1}(p,g)act_{s_1t_1}(p,g) \\
        &\quad \quad \quad \quad = b_{d_1(r_1)d_1(s_1)d_1(t_1)}(p)b_{d_0(r_1)d_0(s_1)d_0(t_1)}^{-1}(pg) \lambda_{\pi(r_1)\pi(s_1)\pi(t_1)}(g),  \\
        &\beta_{r_2}(p,g_1,g_2)\beta_{s_2}^{-1}(p,g_1,g_2) \\
        &\quad \quad \quad \quad = act_{d_2(r_2)d_2(s_2)}^{-1}(p,g_1) act_{d_1(r_2)d_1(s_2)}(p,g_1g_2) act_{d_0(r_2)d_0(s_2)}^{-1}(pg_1,g_2) m_{\pi(r_2)\pi(s_2)}(g_1,g_2), \\
        &\beta_{d_3(r_3)}(p,g_1,g_2)\beta_{d_2(r_3)}^{-1}(p,g_1,g_2g_3)\beta_{d_1(r_3)}(p,g_1g_2,g_3)\beta_{d_0(r_3)}^{-1}(pg_1,g_2,g_3)\alpha_{\pi(r_3)}(g_1,g_2,g_3) = 1.
    \end{split}
    \end{align}
\end{enumerate}

\begin{remark}
    A good semi-simplicial cover $\mathcal{V}_{\bullet}$ of $P/\!/G$ with map $\pi$ as above always exists. Namely, if $\{M_a\}_{a \in \Lambda}$ is a good cover of $M$ where we have sections $s_a:M_a \rightarrow P$ and $\mathcal{U} = \{\{U^n_{i^n}\}_{i_n \in I^n}\}_{n \in \mathbb N}$ is a good semisimplicial cover of $BG_{\bullet}$, then we can let $R^n := \Lambda \times I^{n+1}$ and 
    $V^n_{(a,i_{n+1})} := \{(s_a(x)g_0,g_1,...,g_n) \:|\: x \in M_a, \, (g_0,...,g_n) \in U^{n+1}_{i_{n+1}}\} \cong M_a \times U^{n+1}_{i_{n+1}}$, with maps $d^n_j(a,i_{n+1},s_a(x)g_0,g_1,...,g_n) = (a,d^{n+1}_{j+1}(i_{n+1}),d^n_j(s_a(x)g_0,g_1,...,g_n))$ and $\pi^n(a,i_{n+1},s_a(x)g_0,g_1,...,g_n) = (d^{n+1}_0(i_{n+1}),g_1,...,g_n)$.
\end{remark}

This cocycle data can be used to prove Proposition \ref{prop:pontryagin1} below, which describes which $G$-bundles can be lifted to $\mathcal{G}$-bundles. In order to state the proposition we need to recall first that a \emph{connection} on a $G$-bundle $P \rightarrow M$ is a $A \in \Omega^1(P,\mathfrak{g})$ such that, over $P \times G$,
\begin{equation}\label{eq:connection}
    act^{\ast}A = Ad(g(\cdot)^{-1})p^{\ast}A + g^{\ast}\theta^L.
\end{equation}
By taking vectors tangent to either $P$ or $G$ in this equation, we see that this is equivalent to the two conditions 
\begin{align*}
    A_{pg}(v_pg) = g^{-1}A_p(v_p)g, \quad A_{p\cdot 1}(pu) = u, \quad \text{ for } \quad p \in P,\, v_p \in T_pP, \, g \in G, \,u \in T_1G.
\end{align*} 

We need the following lemma.

\begin{lemma}\label{lem:chernsimons}
    Let $P$ be a $G$-bundle, let $A \in \Omega^1(P,\mathfrak{g})$ be a connection on it and let $\langle \cdot,\cdot \rangle: \mathfrak{g} \otimes \mathfrak{g} \rightarrow \mathfrak{t}$ be an $Ad$-invariant symmetric map. Then the forms 
    $$CS(A) := \langle dA \wedge A \rangle + \frac{1}{3}\langle A \wedge [A \wedge A] \rangle \in \Omega^3(P,\mathfrak{t}), \quad R(A) := \langle p^{\ast}A \wedge g^{\ast}\theta^R \rangle \in \Omega^2(P \times G,\mathfrak{t})$$
    satisfy
            \begin{align}\label{eq:chernsimons}
                 dCS(A) = \langle F_A \wedge F_A \rangle, \quad dR(A) - \delta CS(A) = g^{\ast}\mu, \quad \delta R(A) = g^{\ast}\nu,
           \end{align}
    where $\mu, \, \nu$ are defined by \eqref{eq:murho}, $F_A := dA + \frac{1}{2}[A \wedge A]$ and $\delta$ is the simplicial differential of $P/\!/ G_{\bullet}$.
\end{lemma}
\begin{proof}
    It follows from a straightforward computation.
\end{proof}

\begin{proposition}\label{prop:pontryagin1}
Let $\mathcal{G}$ be a multiplicative $T$-gerbe over $G$ with $T$ connected and let $P \rightarrow M$ be a $G$-bundle. Then
\begin{enumerate}
    \item\label{it:pontryagin1} There is a characteristic class $c(P) \in H^4(M,Z)$ which vanishes precisely when $P$ can be lifted to a $\mathcal{G}$-bundle.
    \item\label{it:pontryagin2} If $\mathcal{G}$ admits a connective structure with associated pairing $\langle \cdot,\cdot \rangle:\mathfrak{g} \otimes \mathfrak{g} \rightarrow \mathfrak{t}$, then the image of $c(P)$ in $H^4(M,\mathfrak{t})$ is represented in de Rham cohomology by $\langle F_A \wedge F_A \rangle \in \Omega^4_{cl}(M,\mathfrak{t})$, for $A \in \Omega^1(P,\mathfrak{g})$ any connection and $F_A = dA + \frac{1}{2}[A \wedge A] \in \Omega^2(M,ad\,P)$ its curvature.
\end{enumerate}
\end{proposition}
\begin{proof}
Recall from Proposition \ref{prop:class2grps} that $\mathcal{G}$, as a multiplicative $T$-gerbe over $G$, is classified by an element in $H^3(BG_{\bullet},C^{\infty}_T)$ described in a \v{C}ech resolution of this sheaf by $\lambda_{i_1j_1k_1}, \,m_{i_2j_2}, \, \alpha_{i_3}$. Consider the map of semi-simplicial manifolds $\pi:(P/\!/G)_{\bullet} \rightarrow BG_{\bullet}$. By inspection of the above cocycle data we see that the data of a $\mathcal{G}$-bundle lifting $P$ is equivalent to a trivialization of $-\pi^{\ast}[\lambda_{i_1j_1k_1}, \,m_{i_2j_2}, \, \alpha_{i_3}] \in H^3(P/\!/G_{\bullet},C^{\infty}_T) = H^3(M,C^{\infty}_T) = H^4(M,Z)$, which proves part \ref{it:pontryagin1}. For part \ref{it:pontryagin2}, recall from Theorem \ref{th:MaurerCartan} that the image of the class $[\lambda_{i_1j_1k_1}, \,m_{i_2j_2}, \, \alpha_{i_3}]$ in $H^4(BG_{\bullet},\mathfrak{t})$ is represented in semi-simplicial de Rham cohomology by $(0,\mu,-\nu,0,0)$, where $\mu$, $\nu$ are defined by \eqref{eq:murho}. Then \eqref{eq:chernsimons} implies precisely that $-\pi^{\ast}(0,\mu,-\nu,0,0)$ and $(\langle F_A \wedge F_A \rangle,0,0,0)$ differ in the semi-simplicial de Rham complex of $P/\!/G_{\bullet}$ by the total derivative of $(CS(A),R(A))$; hence, they define the same class in $H^4(P/\!/G_{\bullet},\mathfrak{t})$. This concludes the proof, as the isomorphism $H^4(M,\mathfrak{t}) = H^4(P/\!/G_{\bullet},\mathfrak{t})$ in de Rham cohomology sends any closed $4$-form $\sigma \in \Omega^4(M,\mathfrak{t})$ to the element $(\pi^{\ast}\sigma,0,0,0)$ of the de Rham complex of $P/\!/G_{\bullet}$.

\end{proof}

\subsubsection{Cocycle data on a good cover of $M$}\label{sec:cocycle2pbsm}
For $\{M_a\}_{a \in \Lambda}$ a good cover of $M$, a $\mathcal{G}$-bundle $(P,\mathcal{P},\rho,\alpha^{\rho})$ is described by:
\begin{enumerate}
    \item $g_{ab}: M_{ab} \rightarrow G$ with $g_{ab}g_{bc} = g_{ac}$.
    \item $\sigma^{ab}$ trivializations of $g_{ab}^{\ast}\mathcal{G}$; that is, $s^{ab}_{i_1j_1}: g_{ab}^{-1}(U_{i_1j_1}^1) \rightarrow T$ with \begin{equation}
        s^{ab}_{i_1j_1}(x)s^{ab}_{j_1k_1}(x)\lambda_{i_1j_1k_1}(g_{ab}(x)) = s^{ab}_{i_1k_1}(x).
    \end{equation}
    \item $\tau^{abc}: m(\sigma^{ab},\sigma^{bc}) \rightarrow \sigma^{ac}$ isomorphisms of trivializations such that
    \begin{equation}
        \begin{tikzcd}
            \text{(Trivial)} \ar[rr,bend left=70,"{(\sigma_{ab}\sigma_{bc})\sigma_{cd}}"{name=F}] 
              \ar[rr,bend left=20,"{\sigma_{ab}(\sigma_{bc}\sigma_{cd})}"{name=G,2cell}]
              \ar[rr,bend right=20,"{\sigma_{ab}\sigma_{bd}}"{name=H,2cell}]
              \ar[rr,bend right=70,"{\sigma_{ad}}"{name=E,swap}] & & 
                (g_{ad})^{\ast}\mathcal{G} \\
             & \ar[Rightarrow,from=F,to=G,"{\alpha}"{swap,pos=0.3},shorten >=1.5pt] \ar[Rightarrow,from=G,to=H,"{\tau^{bcd}}"{swap}] \ar[Rightarrow,from=H,to=E,"{\tau^{abd}}"{swap}] &
        \end{tikzcd} = 
        \begin{tikzcd}
            \text{(Trivial)} \ar[rr,bend left=50,"{(\sigma_{ab}\sigma_{bc})\sigma_{cd}}"{name=F}] 
              \ar[rr,bend left = 20,"{\sigma_{ac}\sigma_{cd}}"{name=G,2cell}]
              \ar[rr,bend right=20,"{\sigma_{ad}}"{name=H,swap}] & & 
                (g_{ad})^{\ast}\mathcal{G} \\
             & \ar[Rightarrow,from=F,to=G,"{\tau^{abc}}"{swap,pos=0.3},shorten >=1.5pt] \ar[Rightarrow,from=G,to=H,"{\tau^{acd}}"{swap}] & 
        \end{tikzcd};
    \end{equation}    
    that is, $t^{abc}_{i_2}: (g_{ab},g_{bc})^{-1}(U^2_{i_2}) \rightarrow T$ such that 
    \begin{align}
    \begin{split}
        t^{abc}_{i_2}(x)s^{ab}_{d_2(i_2)d_2(j_2)}(x)s^{bc}_{d_0(i_2)d_0(j_2)}(x)m_{i_2j_2}(g_{ab}(x),g_{bc}(x)) &= s^{ac}_{d_1(i_2)d_1(j_2)}(x)t^{abc}_{j_2}(x),\\
        t^{abc}_{d_3(i_3)}(x)t^{acd}_{d_1(i_3)}(x)\alpha_{i_3}(g_{ab}(x),g_{bc}(x),g_{cd}(x)) &= t^{bcd}_{d_0(i_3)}(x)t^{abd}_{d_2(i_3)}(x).
    \end{split}
    \end{align}
\end{enumerate}
Indeed, given a $\mathcal{G}$-bundle $(P,\mathcal{P},\rho,\alpha^{\rho})$, one obtains this data by taking local sections $s_a: M_a \rightarrow P$ and trivializations $\eta_a$ of $s_a^{\ast}\mathcal{P} \rightarrow M_a$. Then $g_{ab}$ are defined as the unique functions such that $s_ag_{ab} = s_b$, while $\sigma^{ab}$ are defined as the composition 
$$\text{(Trivial)} \stackrel{\eta_b}{\rightarrow} s_b^{\ast}\mathcal{P} \stackrel{(s_a,g_{ab})^{\ast}\rho^{-1}}{\longrightarrow} s_a^{\ast}\mathcal{P} \otimes g_{ab}^{\ast}\mathcal{G} \stackrel{\eta_a^{-1}}{\rightarrow} g_{ab}^{\ast}\mathcal{G}$$ and similarly $\tau^{abc} = (s_a,g_{ab},g_{bc})^{\ast}(\alpha^{\rho})^{-1}$. Conversely, given such data one constructs 
$$P := \sqcup_a M_a \times G/ \sim$$
with $(a,x,g_{ab}(x)g) \sim (b,x,g)$ and defines $r_a: \pi^{-1}(M_a) \subset P \rightarrow G$ to be $[a,x,g] \mapsto g$; these satisfy $g_{ab}r_b = r_a$. Then $\mathcal{P} \rightarrow P$ is constructed by gluing the gerbes $r_a^{\ast}\mathcal{G} \rightarrow \pi^{-1}(M_a)$ with the isomorphisms $$r_a^{\ast}\mathcal{G} \stackrel{(g_{ab},r_b)^{\ast}m^{-1}}{\rightarrow} g_{ab}^{\ast}\mathcal{G} \otimes r_b^{\ast}\mathcal{G} \stackrel{\sigma_{ab}^{-1}}{\rightarrow} r_b^{\ast}\mathcal{G}$$ and the 2-isomorphisms $(g_{ab},g_{bc},r_c)^{\ast}\alpha^{-1} \otimes \tau^{abc}$.

\subsection{Connections on principal 2-bundles}\label{sec:connections}

Connections on String$(G)$-principal bundles are defined in \cite{WaldString} as trivializations of an associated Chern-Simons 2-gerbe with connection. The existence of this 2-gerbe with connection relies essentially on the fact that String$(G)$ has a canonical enhanced curving. Hence, Theorem \ref{th:MaurerCartan} allows us to generalize loc. cit. to define connections for $\mathcal{G}$-bundles, where $\mathcal{G}$ is any multiplicative gerbe equipped with a connective structure. So from now on we fix such choice of connective structure and write $\langle \cdot,\cdot \rangle: \mathfrak{g} \otimes \mathfrak{g} \rightarrow \mathfrak{t}$, $\Theta^L$ for the associated pairing and curving from Theorem \ref{th:MaurerCartan}. 

The following definition of \emph{connection} is obtained by unwinding what a \emph{trivialization of the Chern-Simons 2-gerbe with connection} means in this context, but we also expand it by introducing \emph{enhanced connections}. Recall that connections on $G$-bundles are defined by \eqref{eq:connection} and that enhanced curvings are defined in Definition \ref{def:enhcurv}.

\begin{definition}\label{def:2pbcon}
Let $(P,\mathcal{P},\rho,\alpha^{\rho})$ be a principal $\mathcal{G}$-bundle. A \emph{connective structure} on it is the following data:
\begin{enumerate}
    \item A connective structure $\nabla$ on the gerbe $\mathcal{P} \rightarrow P$
    \item A connection $\nabla_{\rho}$ on the isomorphism of gerbes $\rho: p^{\ast}\mathcal{P}_{\nabla} \otimes g^{\ast}\mathcal{G}_{\nabla} \rightarrow (pg)^{\ast}\mathcal{P}_{\nabla}$ such that $\alpha^{\rho}$ is a flat 2-isomorphism of gerbes.
\end{enumerate}
We often write $(P,\mathcal{P}_{\nabla},\rho_{\nabla},\alpha^{\rho})$ for a principal $\mathcal{G}$-bundle with connective structure. An \emph{isomorphism of connective structures on a $\mathcal{G}$-bundle} $(\nabla_1,\nabla_{\rho_1}) \rightarrow (\nabla_2,\nabla_{\rho_2})$ is an isomorphism of connective structures on gerbes $\phi:\nabla_1 \rightarrow \nabla_2$ such that the following is a commutative diagram of isomorphisms of gerbes with connective structures
    \begin{equation}
        \begin{tikzcd}
             p^{\ast}\mathcal{P}_{\nabla_1} \otimes g^{\ast}\mathcal{G}_{\nabla} \ar[r,"{(\rho,\nabla_{\rho_1})}",{name=U}] \ar[d,"{(id,p^{\ast}\phi \otimes id)}",swap] 
             & (pg)^{\ast}\mathcal{P}_{\nabla_1}\ar[d,"{(id,(pg)^{\ast}\phi)}"] \\
             p^{\ast}\mathcal{P}_{\nabla_2} \otimes g^{\ast}\mathcal{G}_{\nabla} \ar[r," {(\rho,\nabla_{\rho_2})}",{name=D},swap] 
             & (pg)^{\ast}\mathcal{P}_{\nabla_2}.
        \end{tikzcd}
        \end{equation}
A \emph{connection} (resp. \emph{enhanced connection}) on $(P,\mathcal{P},\rho,\alpha^{\rho})$ is 
\begin{enumerate}
    \item A connective structure $(\nabla,\nabla_{\rho})$ on $(P,\mathcal{P},\rho,\alpha^{\rho})$.
    \item A $G$-connection $A \in \Omega^1(P,\mathfrak{g})$ on $P$.
    \item\label{it:defcon3} A curving (resp. \emph{enhanced curving}) $B$ on $\mathcal{P}_{\nabla} \rightarrow P$ such that $\nabla_{\rho}$ has curvature $-\langle p^{\ast}A \wedge g^{\ast}\theta^R \rangle \in \Omega^2(P \times G,\mathfrak{t})$ (resp. $-2\langle p^{\ast}A \otimes g^{\ast}\theta^R \rangle \in \Gamma(T^{\ast}(P \times G) \otimes T^{\ast}(P \times G) \otimes \mathfrak{t})$) with respect to $p^{\ast}B\otimes g^{\ast}\Theta^L$ (resp. $p^{\ast}B\otimes g^{\ast}\Theta^{L,en}$) on $p^{\ast}\mathcal{P}_{\nabla} \otimes g^{\ast}\mathcal{G}_{\nabla}$ and $(pg)^{\ast}B$ on $(pg)^{\ast}\mathcal{P}_{\nabla}$.
\end{enumerate}
An \emph{isomorphism of (enhanced) connections} $(\nabla_1,\nabla_{\rho_1},A_1,B_1) \rightarrow (\nabla_2,\nabla_{\rho_2},A_2,B_2)$ can only exist if $A_1 = A_2$ and is then given by an isomorphism $\phi:(\nabla_1,\nabla_{\rho_1}) \rightarrow (\nabla_2,\nabla_{\rho_2})$, flat with respect to $B_1$, $B_2$. We write $\mathcal{A}(\mathcal{P})$ for the groupoid of connections on $\mathcal{P}$ and  $\mathcal{A}^{en}(\mathcal{P})$ for the groupoid of enhanced connections on $\mathcal{P}$. Whenever a connective structure $(\nabla,\nabla_{\rho})$ on $\mathcal{P}$ is fixed, we write $\mathcal{A}(\mathcal{P}_{\nabla})$ for the set of connections with such connective structure and $\mathcal{A}^{en}(\mathcal{P}_{\nabla})$ for the set of enhanced connections with such connective structure.
\end{definition}

\begin{remark}
    Recall how the curvature of a connection on an isomorphism of gerbes with respect to curvings is defined in Definition \ref{def:gerbes}. Condition \ref{it:defcon3} in Definition \ref{def:2pbcon} above can be thought of as an analog of the equation \eqref{eq:connection} that defines connections on $G$-bundles, as it is more clear in the cocycle equation \eqref{eq:cocyclecon1e} below.
\end{remark}

\begin{remark}\label{rk:enhcon}
An enhanced connection can be equivalently defined as a pair $((\nabla,\nabla_{\rho},A,B),h)$ of a connection $(\nabla,\nabla_{\rho},A,B)$ and a symmetric tensor $h \in \Gamma(S^2 T^{\ast}M \otimes \mathfrak{t})$. This is because condition \ref{it:defcon3} in Definition \ref{def:2pbcon} states that an $h^P \in \Gamma(S^2T^{\ast}P \otimes \mathfrak{t})$ is the symmetric part of an enhanced connection with underlying $G$-connection $A$ if and only if $h := h^P + \frac{1}{2}\langle A \odot A \rangle \in \Gamma(S^2T^{\ast}P \otimes \mathfrak{t})$ is basic, which can thus be identified with a symmetric tensor on $M$. Note the relation with the Kaluza-Klein mechanism \cite{KimSaeTDual}.
\end{remark}


\begin{lemma}\label{lem:curvature}
     Let $(P,\mathcal{P}_{\nabla},\rho_{\nabla},\alpha^{\rho},A,B)$ be a principal $\mathcal{G}$-bundle with connection. Then the curvature $\hat{H} \in \Omega^3(P,\mathfrak{t})$ of $B$ as a curving on $\mathcal{P}_{\nabla} \rightarrow P$ is of the form $\hat{H} := \pi^{\ast}H - CS(A)$ for some $H \in \Omega^3(M,\mathfrak{t})$, where $CS(A)$ is as in Lemma \ref{lem:chernsimons}.
\end{lemma}
\begin{proof}
It follows from Theorem \ref{th:MaurerCartan} and the general properties of curvings on gerbes that the curvature of $B$ is some $\hat{H} \in \Omega^3(P,\mathfrak{t})$ with $\delta \hat{H} - g^{\ast}\mu = -d\langle p^{\ast}A \wedge g^{\ast}\theta^R \rangle$, so $\hat{H} + CS(A)$ is a basic $3$-form on $P$ by Lemma \ref{lem:chernsimons}, as we wanted to show.
\end{proof}

\begin{definition}\label{def:curvature}
The \emph{curvature} of a connection $(\nabla,\nabla_{\rho},A,B)$ on a $\mathcal{G}$-bundle $(P,\mathcal{P},\rho,\alpha^{\rho})$ is the pair $(F_A,H) \in \Omega^2(M,ad\,P) \oplus \Omega^3(M,\mathfrak{t})$, where $F_A = dA + \frac{1}{2}[A \wedge A]$ and $H$ is the three-form in Lemma \ref{lem:curvature}. A $\mathcal{G}$-bundle with connection is \emph{flat} if $F_A = 0$ and $H = 0$.
\end{definition}

The following two propositions provide proofs of concept for Definitions \ref{def:2pbcon} and \ref{def:curvature}. Proposition \ref{prop:bianchi} states that the curvature of a connection on a $\mathcal{G}$-bundle satisfies the \emph{Green-Schwartz anomally cancellation equation} \cite{Kil,Witten} which is expected in string theory from the field strength of a string with $\mathcal{G}$-symmetry. Proposition \ref{prop:flatcocycle} characterizes flat $\mathcal{G}$-bundles in terms of locally constant cocycle data.

\begin{proposition}[Bianchi Identity]\label{prop:bianchi}
    Let $(P,\mathcal{P}_{\nabla},\rho_{\nabla},\alpha^{\rho},A,B)$ be a principal $\mathcal{G}$-bundle with connection. Then its curvature $(F_A,H) \in \Omega^2(M,ad\,P) \oplus \Omega^3(M,\mathfrak{t})$ satisfies
    \begin{align}
        d_AF_A &= 0, \\
        dH - \langle F_A \wedge F_A \rangle &= 0.
    \end{align}
\end{proposition}
\begin{proof}
    The equation $d_AF_A = 0$ is the classical Bianchi identity for connections on $G$-bundles. On the other hand, the curvature $\hat{H} \in \Omega^3(P,\mathfrak{t})$ of $B$ as a curving on the gerbe $\mathcal{P}_{\nabla} \rightarrow P$ satisfies $d\hat{H} = 0$ and so Lemma \ref{lem:chernsimons} implies $dH - \langle F_A \wedge F_A \rangle = 0$.
\end{proof}

\begin{proposition}\label{prop:flatcocycle}
    A $\mathcal{G}$-bundle $(P,\mathcal{P},\rho,\alpha^{\rho})$ admits locally constant cocycle data $g_{ab}, \,s^{ab}_{i_1j_1}, \, t^{abc}_{i_2}$ as in Section \ref{sec:cocycle2pbsm} if and only if it admits a flat connection.
\end{proposition}
\begin{proof}
    If $(P,\mathcal{P},\rho,\alpha^{\rho})$ admits locally constant cocycle data $g_{ab}, \, s^{ab}_{i_1j_1}, \, t^{abc}_{i_2}$ then using the cocycle description of connections from Section \ref{sec:cocyclecon} below we see that $\sigma^{ab}_{i_1} = 0$, $A_a = 0$, $B_a=0$ is clearly a flat connection. Conversely, if $(P,\mathcal{P},\rho,\alpha^{\rho})$ has a flat connection then recall from Section \ref{sec:cocycle2pbsm} the method for obtaining cocycle data $(g_{ab},\sigma^{ab},\tau^{abc})$ from local sections $s_a:M_a \rightarrow P$ and trivializations $\eta_a$ of $s_a^{\ast}\mathcal{P}$. Since the $G$-bundle with connection $(P,A)$ is flat, it is a classical result that we can find $s_a:M_a \rightarrow P$ such that $s_a^{\ast}A = 0$ with corresponding transition functions $g_{ab}:M_{ab} \rightarrow G$ locally constant. Then $s_a^{\ast}\mathcal{P} \rightarrow M_a$ is a gerbe with connective structure and curving of curvature $s_a^{\ast}(H-CS(A)) = 0$; hence we can take flat trivializations with connections $\eta_{a,\nabla}$ by part \ref{it:flatgerbes} of Proposition \ref{prop:gerbes}. The corresponding $\sigma^{ab}_{\nabla}$ are then trivializations with connections whose curvature is $-Curv(\eta_{a,\nabla}) + (s_a,g_{ab})^{\ast}R(A) + Curv(\eta_{b,\nabla}) = 0$; hence they can be represented by locally constant transitions $s^{ab}_{i_1j_1}$. Finally, since $\alpha^{\rho}$ is a flat 2-isomorphism then $\tau^{abc} = (s_a,g_{ab},g_{bc})^{\ast}\alpha^{\rho}$ is locally constant.
\end{proof}

\subsubsection{Cocycle data for connections on principal 2-bundles}\label{sec:cocyclecon}

We give now cocycle descriptions of connections adapted to each of the two cocycle descriptions of $\mathcal{G}$-bundles from Section \ref{sec:2pbs}. On a semi-simplicial cover $\mathcal{V}_{\bullet}$ of $(P /\!/G)_{\bullet}$ where the $\mathcal{G}$-bundle is given by $b_{r_0s_0t_0},\, act_{r_1s_1},\, \beta_{r_2}$, it follows directly from the definitions that a connection is given by an ordinary $G$-connection $A \in \Omega^1(P,\mathfrak{g})$ and
\begin{align}
    \sigma_{r_0s_0} \in \Omega^1(V^0_{r_0s_0},\mathfrak{t}), \quad \eta_{r_1} \in \Omega^1(V^1_{r_1},\mathfrak{t}), \quad B_{r_0} \in \Omega^2(V^0_{r_0},\mathfrak{t})
\end{align}
such that 
\begin{subequations}
\begin{align}
        &\sigma_{r_0s_0} - \sigma_{r_0t_0} + \sigma_{s_0t_0} = b_{r_0s_0t_0}^{\ast}\theta^T, \\
        &\eta_{r_1} +  d_0^{\ast}\sigma_{d_0(r_1)d_0(s_1)} + act_{r_1s_1}^{\ast}\theta^T =\pi^{\ast}A_{\pi(r_1)\pi(s_1)} + d_1^{\ast}\sigma_{d_1(r_1)d_1(s_1)} + \eta_{s_1}, \\
        &\beta_{r_2}^{\ast}\theta^T + \pi^{\ast}M_{\pi(r_2)} + d_1^{\ast}\eta_{d_1(r_2)} = d_0^{\ast}\eta_{d_0(r_2)} + d_2^{\ast}\eta_{d_2(r_2)}, \\
        &B_{r_0} - B_{s_0} = d\sigma_{r_0s_0},\\
        &-d\eta_{r_1} + \pi^{\ast}\Theta^L_{\pi(r_1)} + d_1^{\ast}B_{d_1(r_1)} - d_0^{\ast}B_{d_0(r_0)} = \langle p^{\ast}A \wedge g^{\ast}\theta^R \rangle. \label{eq:cocyclecon1e}
    \end{align}
\end{subequations}
On the other hand, if the $\mathcal{G}$-bundle is given on a cover $\{M_a\}_{a \in \Lambda}$ of $M$ by $g_{ab}$, $s^{ab}_{i_1j_1}$, $t^{abc}_{i_2}$, then a connection is given by 
\begin{align}
    \sigma^{ab}_{i_1} \in \Omega^1(g_{ab}^{-1}(U_{i_1}^1),\mathfrak{t}), \quad A_a \in \Omega^1(M_a,\mathfrak{g}), \quad B_a \in \Omega^2(M_a,\mathfrak{t})
\end{align}
such that
\begin{subequations}
\begin{align}
    &\sigma^{ab}_{i_1} + (s^{ab}_{i_1j_1})^{\ast}\theta^T + g_{ab}^{\ast}A_{i_1j_1} = \sigma^{ab}_{j_1},\\
    &(t^{abc}_{i_2})^{\ast}\theta^T - \sigma^{bc}_{d_0(i_2)} + \sigma^{ac}_{d_1(i_2)} - \sigma^{ab}_{d_2(i_2)} - (g_{ab},g_{bc})^{\ast}M_{i_2} = 0,\\
    &A_b - Ad(g_{ab}^{-1})A_a = g_{ab}^{\ast}\theta^L,\\
    &B_b - B_a + \langle A_a \wedge g_{ab}^{\ast}\theta^R \rangle = d\sigma^{ab}_{i_1} + g_{ab}^{\ast}\Theta_{i_1}^L.
\end{align}
\end{subequations}
In this case $H \in \Omega^3(M,\mathfrak{t})$ is given locally by $H_{|M_a} := dB_a + \langle dA_a \wedge A_a \rangle + \frac{1}{3}\langle A_a \wedge [A_a \wedge A_a] \rangle$. If $(\sigma^{ab}_{i_1},A_a,B_a)$ and $(\hat{\sigma}^{ab}_{i_1},\hat{A}_a,\hat{B}_a)$ are two connections, then an isomorphism between them can only exist if $A_a = \hat{A}_a$ and is then given by $\tau_a \in \Omega^1(M_a,\mathfrak{t})$ such that
\begin{align}
\begin{split}
    &\hat{\sigma}^{ab}_{i_1} - \sigma^{ab}_{i_1} = \tau_b - \tau_a,\\
    &\hat{B}_a - B_a = d\tau_a.
\end{split}
\end{align}
This cocycle data agrees with \cite{ShengXuZhu}. It can be obtained, as in Section \ref{sec:2pbs}, by taking sections $s_a: M_a \rightarrow P$ and trivializations with connections $\eta_{a,\nabla}$ of $s_a^{\ast}\mathcal{P}_{\nabla}$. Then $\sigma^{ab}_{i_1}$ are the 1-forms defining the connection on the isomorphism of gerbes
$$\sigma_{ab}: \text{(Trivial)} \stackrel{\eta_b}{\rightarrow} s_b^{\ast}\mathcal{P} \stackrel{(s_a,g_{ab})^{\ast}\rho^{-1}}{\longrightarrow} s_a^{\ast}\mathcal{P} \otimes g_{ab}^{\ast}\mathcal{G} \stackrel{\eta_a^{-1}}{\rightarrow} g_{ab}^{\ast}\mathcal{G},$$
while $A_a := s_a^{\ast}A$ and $B_a := Curv(\eta_a)$. Conversely, from this data, recall that $(P,A)$ is constructed as $P := \sqcup_a M_a \times G/ \sim$ with $(a,x,g_{ab}(x)g) \sim (b,x,g)$, and that we define $r_a: \pi^{-1}(M_a) \subset P \rightarrow G$ to be $[a,x,g] \mapsto g$. Then the connection $A$ is constructed as $A_{|\pi^{-1}(M_a)} := r_a^{\ast}\theta^L + Ad(r_a^{-1})A_a$, while $(\mathcal{P}_{\nabla},B)$ is obtained by gluing the gerbes with connective structure and curving $$(r_a^{\ast}\mathcal{G}_{\nabla},r_a^{\ast}\Theta + B_a - \langle A_a \wedge r_a^{\ast}\theta^R \rangle) \rightarrow \pi^{-1}(M_a)$$
with the isomorphisms 
$$r_a^{\ast}\mathcal{G}_{\nabla} \stackrel{(g_{ab},r_b)^{\ast}m^{-1}}{\rightarrow} g_{ab}^{\ast}\mathcal{G} \otimes r_b^{\ast}\mathcal{G} \stackrel{\sigma_{ab}^{-1}}{\rightarrow} r_b^{\ast}\mathcal{G}$$
and the 2-isomorphisms $(g_{ab},g_{bc},r_c)^{\ast}\alpha^{-1} \otimes \tau^{abc}$.


This cocycle data can be used to prove Proposition \ref{prop:connections} below, which describes the groupoids $\mathcal{A}(\mathcal{P})$, $\mathcal{A}^{en}(\mathcal{P})$ of connections and enhanced connections on a $\mathcal{G}$-bundle introduced in Definition \ref{def:2pbcon}. First we need some notation: for $P \rightarrow M$ a $G$-bundle and $ad \, P \rightarrow M$ its associated $\mathfrak{g}$-bundle we write $\Omega^1(ad \,P) \times_{\langle \cdot,\cdot \rangle} \Omega^2(M,\mathfrak{t})$ for the group with underlying set $\Omega^1(ad \,P) \times \Omega^2(M,\mathfrak{t})$ but with product
\begin{equation}\label{eq:torsorproduct}
    (a_1,b_1) \cdot (a_2,b_2) = (a_1+a_2,b_1+b_2+\langle a_1 \wedge a_2 \rangle)
\end{equation}
i.e., it is a non-trivial central extension of $\Omega^1(ad \,P)$ by $\Omega^2(M,\mathfrak{t})$\footnote{The space $\Omega^1(M,\mathfrak{g}) \times_{\langle \cdot,\cdot \rangle} \Omega^2(M,\mathfrak{t})$ is the natural space of 1-forms with values in $\mathfrak{g} \times B\mathfrak{t}$, if this is thought of as a 2-vector space with a braiding for the addition functor given precisely by $\langle \cdot,\cdot \rangle$.}.


\begin{proposition}\label{prop:connections}
Let $(P,\mathcal{P},\rho,\alpha^{\rho}) \rightarrow M$ be a $\mathcal{G}$-bundle.
\begin{enumerate}
    \item\label{it:con1} Connective structures on $(P,\mathcal{P},\rho,\alpha^{\rho})$ always exist and any two are isomorphic; the set of isomorphisms between them is a torsor for $\Omega^1(M,\mathfrak{t})$.
    \item\label{it:con2} For any fixed connective structure $(\nabla,\nabla_{\rho})$ on $(P,\mathcal{P},\rho,\alpha^{\rho})$, the set $\mathcal{A}(\mathcal{P}_{\nabla})$ is a right torsor for $\Omega^1(ad \,P) \times_{\langle \cdot,\cdot \rangle} \Omega^2(M,\mathfrak{t})$. In particular, it is an affine $\Omega^2(M,\mathfrak{t})$-bundle over the space of $G$-connections on $P$.
    \item\label{it:con3} For any fixed connective structure $(\nabla,\nabla_{\rho})$ on $(P,\mathcal{P},\rho,\alpha^{\rho})$, the set $\mathcal{A}^{en}(\mathcal{P}_{\nabla})$ is a right torsor for $\Omega^1(ad \,P) \times_{\langle \cdot,\cdot \rangle} \Omega^2(M,\mathfrak{t}) \times \Gamma(S^2T^{\ast}M \otimes \mathfrak{t})$. In particular, it is an affine $\Gamma(T^{\ast}M \otimes T^{\ast}M \otimes \mathfrak{t})$-bundle over the space of $G$-connections on $P$.
\end{enumerate}
\end{proposition}
\begin{proof}
We prove existence of connective structures and connections using the cocycle description on a cover $\{M_a\}_{a \in \Lambda}$ of $M$. Given cocycle data $g_{ab}$, $s^{ab}_{i_1j_1}$, $t^{abc}_{i_2}$ for the $\mathcal{G}$-bundle, one checks immediately that for fixed $a, \,b \in \Lambda$, $(s^{ab}_{i_1j_1})^{\ast}\theta^T + g_{ab}^{\ast}A_{i_1j_1}$ is a cocycle of 1-forms in $M_{ab}$ with respect to the covering $\{g_{ab}^{-1}(U_{i_1})\}_{i_1 \in I_1}$ and so there exist $\tilde{\sigma}^{ab}_{i_1}$ with $\tilde{\sigma}^{ab}_{i_1} + (s^{ab}_{i_1j_1})^{\ast}\theta^T + g_{ab}^{\ast}A_{i_1j_1} = \tilde{\sigma}^{ab}_{j_1}$. Then for fixed $a, \, b, \, c \in \Lambda$ one can use the cocycle conditions to see that there is a 1-form $\eta_{abc} \in \Omega^1(M_{abc},\mathfrak{t})$ that restricts on each $(g_{ab},g_{bc})^{-1}(U_{i_2})$ to $\eta_{abc} := (t^{abc}_{i_2})^{\ast}\theta^T - \tilde{\sigma}^{bc}_{d_0(i_2)} + \tilde{\sigma}^{ac}_{d_1(i_2)} - \tilde{\sigma}^{ab}_{d_2(i_2)} + (g_{ab},g_{bc})^{\ast}M_{i_2}$; moreover, one can check that $\eta_{abc}$ form a cocycle of 1-forms with respect to the cover $\{M_a\}_a$ of $M$; hence there exist $\xi_{ab}$ with $\eta_{abc} = \xi_{ab} - \xi_{ac} + \xi_{bc}$ and so $\sigma^{ab}_{i_1} := \tilde{\sigma}^{ab}_{i_1} + \xi_{ab}$ gives a connective structure. Given two connective structures $\sigma^{ab}_{i_1}$, $\hat{\sigma}^{ab}_{i_1}$ then their difference is a cocycle of one-forms on $M$ with respect to $\{M_a\}_{a \in \Lambda}$, as it follows from their cocycle equations, hence an isomorphism always exists and the set of such is indeed a torsor for $\Omega^1(M,\mathfrak{t})$.

Once a connective structure $\sigma^{ab}_{i_1}$ is chosen, the existence of the 1-forms $A_a$ is simply the classical existence of connections on $G$-bundles and then the existence of $B_a$ follows from the fact that $d\sigma^{ab}_{i_1} + g_{ab}^{\ast}\Theta_{i_1} \in \Omega^2(M_{ab},\mathfrak{t})$ is independent of $i_1$ and that $-\langle A_a \wedge g_{ab}^{\ast}\theta^R \rangle + d\sigma^{ab}_{i_1} + g_{ab}^{\ast}\Theta_{i_1}$ is a cocycle of $2$-forms in $M$ for any choice of $A$. This shows that the space of $\mathcal{G}$-connections with fixed connective structure is an affine $\Omega^2(M,\mathfrak{t})-$bundle over the space of $G$-connections. To obtain the other, more precise description, note that if $(A,B) \in \mathcal{A}(\mathcal{P}_{\nabla})$ is a connection and $(a,b) \in \Omega^1(M,ad \,P)\times \Omega^2(M,\mathfrak{t})$ then $(A+a,B+\langle A \wedge a \rangle + b)$ is a connection. In fact, all connections are of this form as they are necessarily of the form $(A+a,B+\tilde{b})$ with $a \in \Omega^1(M,ad \,P)$, $\tilde{b} \in \Omega^2(P,\mathfrak{t})$ satisfying $act^{\ast}\tilde{b} - p^{\ast}\tilde{b} = - \langle a \wedge g^{\ast}\theta^R \rangle$ over $P \times G$, so $b := \tilde{b} - \langle A \wedge a \rangle$ is a basic form. One can easily check that this defines a right action of $\Omega^1(ad\,P) \times_{\langle\cdot,\cdot\rangle} \Omega^2(M,\mathfrak{t})$. Then part \ref{it:con3} follows directly from Remark \ref{rk:enhcon}.
\end{proof}

\begin{remark}\label{rk:spaceconnections}
    Parts \ref{it:con1} and \ref{it:con2} from Proposition \ref{prop:connections} can be summarized in a single, more abstract statement by saying that the groupoid $\mathcal{A}(\mathcal{P})$ is a torsor (i.e., a principal 2-bundle over a point) for the Lie 2-group associated to the crossed module 
    $$(\Omega^1(ad\,P) \times_{\langle \cdot,\cdot \rangle} \Omega^2(M,\mathfrak{t}),\Omega^1(M,\mathfrak{t}),f,\triangleright)$$
    with $\triangleright$ trivial and $f:\Omega^1(M,\mathfrak{t}) \rightarrow \Omega^1(ad\,P) \times_{\langle \cdot,\cdot \rangle} \Omega^2(M,\mathfrak{t})$, $\eta \mapsto (0,d\eta)$. Similarly, $\mathcal{A}^{en}(\mathcal{P})$ is a torsor for $(\Omega^1(ad\,P) \times_{\langle \cdot,\cdot \rangle} \Omega^2(M,\mathfrak{t}) \times \Gamma(S^2T^{\ast}M \otimes \mathfrak{t}),\Omega^1(M,\mathfrak{t}),f,\triangleright)$. These results are analogous to the classical fact that the space of connections on a $G$-bundle is a torsor for $\Omega^1(M, \,ad P)$. 
\end{remark}

\subsubsection{Cocycle data for the gauge action on connections}\label{sec:gauge}

\begin{definition}
    For $\mathcal{P}_{\nabla}^1$, $\mathcal{P}_{\nabla}^2$ two $\mathcal{G}$-bundles with connective structure and $(u,\varphi,\alpha^{\varphi})$ an isomorphism betwen the underlying $\mathcal{G}$-bundles (cf. Definition \ref{def:2pbmult}), we define a \emph{connection} on $(u,\varphi,\alpha^{\varphi})$ to be a connection on the isomorphism of gerbes $\varphi$ such that that $\alpha^{\varphi}$ is a flat 2-isomorphism of gerbes. If $\psi:(u,\varphi,\alpha^{\varphi}) \Rightarrow (u,\varphi',\alpha^{\varphi'})$ is a 2-isomorphism of $\mathcal{G}$-bundles and $(u,\varphi,\alpha^{\varphi}), \, (u,\varphi',\alpha^{\varphi'})$ are equipped with connections, we say that $\psi$ is \emph{flat} if $\psi$ is flat as a 2-isomorphism of gerbes. This defines the \emph{bicategory of $\mathcal{G}_{\nabla}$-bundles}.
\end{definition}

We proceed to decribe isomorphisms and 2-isomorphisms of $\mathcal{G}_{\nabla}$-bundles, and their action on connections, in terms of cocycle data on a good cover $\{M_a\}_a$ of $M$ (cf. Sections \ref{sec:cocycle2pbsm} and \ref{sec:cocyclecon}). This will be useful in Section \ref{sec:chern} in order to give a local formula for the Chern connection of a holomorphic $\mathcal{G}_{\nabla}$-bundle with a reduction to a $\mathcal{K}_{\nabla}$-bundle.

If the $\mathcal{G}_{\nabla}$-bundles $\mathcal{P}_{\nabla}^1$ and $\mathcal{P}_{\nabla}^2$ are described by $(g_{ab}^1,\sigma_{ab,\nabla}^1,\tau_{abc}^1)$, $(g_{ab}^2,\sigma_{ab,\nabla}^2,\tau_{abc}^2)$, where $\sigma_{ab,\nabla}^i$ are trivializations with connections of $(g_{ab}^i)^{\ast}\mathcal{G}_{\nabla}$ and $\tau_{abc}^i: m(\sigma_{ab,\nabla}^i,\sigma_{bc,\nabla}^i) \rightarrow \sigma_{ac,\nabla}^i$ are flat isomorphisms of trivializations as in Sections \ref{sec:cocycle2pbsm} and \ref{sec:cocyclecon}, then an isomorphism $\mathcal{P}_{\nabla}^1 \rightarrow \mathcal{P}_{\nabla}^2$ is given by:
\begin{enumerate}
    \item Functions $\varphi_a: M_a \rightarrow G$ with $\varphi_ag_{ab}^1 = g_{ab}^2\varphi_b$.
    \item Trivializations with connection $\Phi_{a,\nabla}$ of $\varphi_a^{\ast}\mathcal{G}_{\nabla}$.
    \item Flat isomorphisms of trivializations $\psi_{ab}: m(\Phi_{a,\nabla},\sigma_{ab,\nabla}^1) \rightarrow m(\sigma_{ab,\nabla}^2,\Phi_{b,\nabla})$ satisfying 
\end{enumerate}
\begin{equation}
        \begin{tikzcd}
            \text{(Trivial)} \ar[rr,bend left=100,"{m((\Phi_a\sigma_{ab}^1)\sigma_{bc}^1)}"{name=F}] 
              \ar[rr,bend left=40,"{m(\Phi_a(\sigma_{ab}^1\sigma_{bc}^1))}"{name=G,2cell}]
              \ar[rr,"{m(\Phi_a\sigma_{ac}^1)}"{name=H,2cell}]
              \ar[rr,bend right=30,"{m(\sigma_{ac}^2\Phi_c)}"{name=E,2cell}]
              \ar[rr,bend right=100,"{m((\sigma_{ab}^2\sigma_{bc}^2)\Phi_c)}"{name=J,swap}] & & 
                (g_{ac}^2\varphi_c)^{\ast}\mathcal{G}
             & \ar[Rightarrow,from=F,to=G,"{\alpha}"{swap}] \ar[Rightarrow,from=G,to=H,"{\tau^{abc}_1}"{swap,pos=0.9}] \ar[Rightarrow,from=H,to=E,"{\psi_{ac}}"{swap,pos=0.1}] \ar[Rightarrow,from=E,to=J,"{(\tau^{abc}_2)^{-1}}"{swap,pos=0.2}] 
        \end{tikzcd} = 
        \begin{tikzcd}
            \text{(Trivial)} \ar[rr,bend left=100,"{m((\Phi_a\sigma_{ab}^1)\sigma_{bc}^1)}"{name=F}] 
              \ar[rr,bend left=40,"{m((\sigma_{ab}^2\Phi_b)\sigma_{bc}^1)}"{name=G,2cell}]
              \ar[rr,"{m(\sigma_{ab}^2(\Phi_b\sigma_{bc}^1))}"{name=H,2cell}]
              \ar[rr,bend right=40,"{m(\sigma_{ab}^2(\sigma_{bc}^2\Phi_c))}"{name=E,2cell}]
              \ar[rr,bend right=100,"{m((\sigma_{ab}^2\sigma_{bc}^2)\Phi_c)}"{name=J,swap}] & & 
                (g_{ac}^2\varphi_c)^{\ast}\mathcal{G} .
             & \ar[Rightarrow,from=F,to=G,"{\psi_{ab}}"{swap}] \ar[Rightarrow,from=G,to=H,"{\alpha}"{swap,pos=0.5}] \ar[Rightarrow,from=H,to=E,"{\psi_{bc}}"{swap,pos=0.1}] \ar[Rightarrow,from=E,to=J,"{\alpha}"{swap,pos=0.5}] 
        \end{tikzcd}
    \end{equation}   
In particular, each trivialization with connection $\Phi_{a,\nabla}$ has a curvature $Curv(\Phi_{a,\nabla}) \in \Omega^2(M_a,\mathfrak{t})$ satisfying 
\begin{align*}
    dCurv(\Phi_{a,\nabla}) &= \varphi_a^{\ast}\mu,\\
    Curv(\Phi_{a,\nabla}) - Curv(\Phi_{b,\nabla}) &= Curv(\sigma_{ab,\nabla}^2) - \langle (g_{ab}^2)^{\ast}\theta^L \wedge \varphi_b^{\ast}\theta^R \rangle - Curv(\sigma_{ab,\nabla}^1) + \langle \varphi_a^{\ast}\theta^L \wedge (g_{ab}^1)^{\ast}\theta^R \rangle.    
\end{align*}
If $(\varphi_a,\Phi_{a,\nabla},\psi_{ab})$, $(\varphi_a',\Phi_{a,\nabla}',\psi_{ab}')$ are two isomorphisms $\mathcal{P}_{\nabla}^1 \rightarrow \mathcal{P}_{\nabla}^2$, then a 2-isomorphism between them can only exist if $\varphi_a = \varphi_a'$ and in that case it is given by flat 2-isomorphisms $\chi_a: \Phi_{a,\nabla} \Rightarrow \Phi'_{a,\nabla}$ such that
\begin{equation}
        \begin{tikzcd}
            \text{(Trivial)} 
              \ar[rr,bend left=20,"{m(\Phi_a\sigma_{ab}^1)}"{name=G}]
              \ar[rr,bend right=20,"{m(\Phi_a'\sigma_{ab}^1)}"{name=H,2cell}]
              \ar[rr,bend right=70,"{m(\sigma_{ab}^2\Phi_b')}"{name=E,swap}] & & 
                (g_{ab}^2\varphi_b)^{\ast}\mathcal{G} \\
             &\ar[Rightarrow,from=G,to=H,"{\chi_a}"{swap}] \ar[Rightarrow,from=H,to=E,"{\psi_{ab}'}"{swap}] &
        \end{tikzcd} = 
        \begin{tikzcd}
            \text{(Trivial)} 
              \ar[rr,bend left=20,"{m(\Phi_a\sigma_{ab}^1)}"{name=G}]
              \ar[rr,bend right=20,"{m(\sigma_{ab}^2\Phi_b)}"{name=H,2cell}]
              \ar[rr,bend right=70,"{m(\sigma_{ab}^2\Phi_b')}"{name=E,swap}] & & 
                (g_{ab}^2\varphi_b)^{\ast}\mathcal{G} \\
             &\ar[Rightarrow,from=G,to=H,"{\psi_{ab}}"{swap}] \ar[Rightarrow,from=H,to=E,"{\chi_{b}}"{swap}] &
        \end{tikzcd}.
    \end{equation} 
In particular, the existence of such $\chi_a$ implies $Curv(\Phi_{a,\nabla}) = Curv(\Phi_{a,\nabla}')$. Now if $(A_a^1,B_a^1)$ is a connection for $\mathcal{P}_{\nabla}^1$, then $(\varphi_a,\Phi_{a,\nabla},\psi_{ab})$ sends it to the following connection for $\mathcal{P}_{\nabla}^2$:
\begin{align}
    A_a^2 &:= Ad(\varphi_a)A_a^1 - \varphi_a^{\ast}\theta^R,\\
    B_a^2 &:= B_a - \langle \varphi_a^{\ast}\theta^L \wedge A_a^1 \rangle - Curv(\Phi_{a,\nabla}).
\end{align}
Note that if there exists a 2-isomorphism between $(\varphi_a,\Phi_{a,\nabla},\psi_{ab})$ and $(\varphi_a',\Phi_{a,\nabla}',\psi_{ab}')$, then $\varphi_a = \varphi_a'$ and $Curv(\Phi_{a,\nabla}) = Curv(\Phi_{a,\nabla}')$; hence, their actions on $(A_a^1,B_a^1)$ coincide.

\subsection{Comparison with adjusted connections}\label{sec:adjcon}

This section is a digression from the main topic of this article. As explained in Section \ref{sec:adj}, some multiplicative gerbes can be described by crossed modules and in this case connective structures correspond to a certain class of adjustments that we call \emph{strong}. Since there exists a theory of connections for adjusted Lie crossed modules (\cite{KimSaeTransport,RistSaeWolf}), it seems natural to check if it agrees with our Definition \ref{def:2pbcon}, whenever they can be compared. This is our next task. As far as we know, this has not been carried out explicitly anywhere in the literature. 

We start by collecting the cocycle data for Lie crossed module bundles and adjusted connections on them following \cite{KimSaeTransport,RistSaeWolf}. Let $(\tilde{G},H,f,\triangleright,\tilde{\kappa})$ be such an adjusted Lie crossed module, as in Example \ref{ex:crossed} and Remark \ref{rk:coradj}. Then a \emph{bundle} for it is given in a cover $\{M_a\}_a$ of $M$ by maps $\tilde{g}_{ab}: M_{ab} \rightarrow \tilde{G}$, $h_{abc}: M_{abc} \rightarrow H$ such that
\begin{align}
\begin{split}
    t(h_{abc})\tilde{g}_{ab}\tilde{g}_{bc} = \tilde{g}_{ac},\\
    h_{acd}h_{abc} = h_{abd} \tilde{g}_{ab} \triangleright h_{bcd}.
\end{split}
\end{align}
An \emph{adjusted connection} is given by $\Lambda_{ab} \in \Omega^1(M_{ab},\mathfrak{h})$, $\tilde{A}_a \in \Omega^1(M_a,\tilde{\mathfrak{g}})$ and $\tilde{B}_a \in \Omega^2(M_a,\mathfrak{h})$ satisfying
\begin{align}
\begin{split}\label{eq:cocycleadjcon}
    \Lambda_{bc} + g_{bc}^{-1} \triangleright \Lambda_{ab} &= \Lambda_{ac} - \tilde{g}_{ac}^{-1} \triangleright (h_{abc}^{\ast}\theta^R + \tilde{A}_a \triangleright h_{abc} \cdot h_{abc}^{-1} ), \\
    \tilde{A}_b - Ad(\tilde{g}_{ab}^{-1})A_a &= \tilde{g}_{ab}^{\ast}\theta^L - f(\Lambda_{ab}),\\
    \tilde{g}_{ab}^{-1} \triangleright \tilde{B}_a - \tilde{B}_b &= -d\Lambda_{ab} - \frac{1}{2}[\Lambda_{ab} \wedge \Lambda_{ab}] - \tilde{A}_b \triangleright \Lambda_{ab} + \tilde{\kappa}\left(\tilde{g}_{ab}, d\tilde{A}_a + \frac{1}{2}[\tilde{A}_a \wedge \tilde{A}_a] + f(\tilde{B}_a) \right).
\end{split}
\end{align}
An \emph{isomorphism of adjusted connections} $(\Lambda_{ab},\tilde{A}_a,\tilde{B}_a) \rightarrow (\Lambda_{ab}',\tilde{A}_a',\tilde{B}_a')$ is described by $\lambda_a \in \Omega^1(M_a,\mathfrak{h})$ such that
\begin{align}
\begin{split}
    \Lambda_{ab}' - \Lambda_{ab} &= \lambda_b - \tilde{g}_{ab}^{-1} \triangleright \lambda_a,\\
    \tilde{A}'_a - \tilde{A}_a &= -f(\lambda_a),\\
    \tilde{B}'_a - \tilde{B}_a &= d\lambda_a + \frac{1}{2}[\lambda_a \wedge \lambda_a] + \tilde{A}'_a \triangleright \lambda_a.
\end{split}
\end{align}
An important property of adjusted connections that can be derived directly from the above cocycle data is the following equivariance relation:
\begin{align}
    d\tilde{A}_b + \frac{1}{2}[\tilde{A}_b \wedge \tilde{A}_b] = Ad(\tilde{g}_{ab}^{-1})(d\tilde{A}_a + \frac{1}{2}[\tilde{A}_a \wedge \tilde{A}_a]) -f\left( d\Lambda_{ab} + \frac{1}{2}[\Lambda_{ab} \wedge \Lambda_{ab}] + \tilde{A}_b \triangleright \Lambda_{ab} \right) \label{eq:adjeq}
\end{align}
If $r: \tilde{\mathfrak{g}} \rightarrow \mathfrak{h}$ is the map corresponding to splittings $(s,l)$ as in \eqref{eq:splittinglie2alg}, then an adjusted connection $(\Lambda_{ab},\tilde{A}_a,\tilde{B}_a)$ is said to be \emph{fake-flat with respect to $r$} if
\begin{equation}
    r(d\tilde{A}_a + \frac{1}{2}[\tilde{A}_a \wedge \tilde{A}_a] + f\tilde{B}_a) = 0.
\end{equation} 

Assume now that $(\tilde{G},H,f,\triangleright,\tilde{\kappa})$ is a central Lie crossed module, let $G := \tilde{G}/Im(t)$, $T := \ker (t)$ and assume that $\tilde{\kappa}$ is obtained from a strong adjustment $(s,\kappa)$ and a splitting $l: \mathfrak{g} \rightarrow \tilde{\mathfrak{g}}$ as in Remark \ref{rk:coradj}, so that the associated multiplicative gerbe $\mathcal{G}$ has a connective structure by Proposition \ref{prop:adj}. Then we have:

\begin{proposition}\label{prop:adjcon}
    Under the above conditions, there is an equivalence of bicategories between the bicategory of $\mathcal{G}$-bundles with connection in the sense of Definition \ref{def:2pbcon} and the bicategory of bundles for the crossed module with adjusted connections that are fake-flat with respect to $r$.
\end{proposition}
\begin{proof}
    First, there is an equivalence of bicategories between the bicategory of $\mathcal{G}$-bundles and the bicategory of crossed module bundles, because both are equivalent to the bicategory of principal 2-bundles (as in Definition \ref{def:2pb}) for the corresponding Lie 2-group: the case of crossed modules is proven in \cite{NikWal2pbs}, while the other one is Proposition \ref{prop:2pbmult}. This means that we only need to prove that the categories of connections are equivalent in both approaches. For this we use first the cocycle description for connections on $\mathcal{G}$-bundles in the sense of Definition \ref{def:2pbcon} adapted to the multiplicative gerbe with connective structure from Proposition \ref{prop:adj} and Example \ref{ex:crossedmultger}. Unwinding the definitions and using the formulas for $\langle \cdot,\cdot \rangle$ and $\Theta^L$, this is: $\sigma_{ab} \in \Omega^1(M_{ab},\mathfrak{t})$, $A_a \in \Omega^1(M_a,\mathfrak{g})$, $B_a \in \Omega^2(M_a,\mathfrak{t})$ such that:
    \begin{subequations}
    \label{eq:cocycleadj}
    \begin{align}
        \sigma_{ab} - \sigma_{ac} + \sigma_{bc} &= - \kappa(\tilde{g}_{bc},\tilde{g}_{ab}^{\ast}\theta^L) - s(\tilde{g}_{ac}^{-1} \triangleright h_{abc}^{\ast}\theta^R), \label{eq:cocycleadj1}\\
        A_b &= Ad(g_{ab}^{-1})A_a + g_{ab}^{\ast}\theta^L, \label{eq:cocycleadj2}\\
        B_b - B_a &= d\sigma_{ab} - \frac{1}{2}(\partial_1\kappa(\tilde{g}_{ab}^{\ast}\theta^L \wedge \tilde{g}_{ab}^{\ast}\theta^L) + \partial_1\kappa(A_a \wedge \tilde{g}_{ab}^{\ast}\theta^R) - \partial_1\kappa( \tilde{g}_{ab}^{\ast}\theta^R \wedge A_a)  ) \label{eq:cocycleadj3},
    \end{align}
    \end{subequations}
    where $g_{ab} := \pi(\tilde{g}_{ab}): M_{ab} \rightarrow G$. An isomorphism of connections $(\sigma_{ab},A_a,B_a) \rightarrow (\sigma_{ab}',A_a',B_a')$ can only exist if $A_a = A_a'$ and is then given by $\lambda_a \in \Omega^1(M_a,\mathfrak{t})$ such that $\sigma'_{ab} - \sigma_{ab} = \lambda_b - \lambda_a$, $B'_a - B_a = d\lambda_a$. Now letting $\xi_{ab} := (\kappa(\tilde{g}_{ab},l(A_a)) + r(Ad(\tilde{g}_{ab}^{-1})l(A_a) + \tilde{g}_{ab}^{\ast}\theta^L))$, a quick computation reveals that 
    \begin{align}
        \kappa(\tilde{g}_{bc},\tilde{g}_{ab}^{\ast}\theta^L) &= \tilde{g}_{bc}^{-1} \triangleright \xi_{ab} - \xi_{ac} + \xi_{bc} + \tilde{g}_{ac}^{-1} \triangleright (l(A_a) \triangleright h_{abc} \cdot h_{abc}^{-1}) + rf(\tilde{g}_{ac}^{-1} \triangleright h_{abc}^{\ast}\theta^R),\\
        s(d\xi_{ab}) &= \partial_1\kappa(\tilde{g}_{ab}^{\ast}\theta^L \wedge Ad(\tilde{g}_{ab}^{-1})l(A_a)) + \kappa(\tilde{g}_{ab},dl(A_a)),
    \end{align}
    which can be used to see that 
    \begin{align}
        \Lambda_{ab} &:= \xi_{ab} + \sigma_{ab},\\
        \tilde{A}_a &:= l(A_a),\\
        \tilde{B}_a &:= B_a + \frac{1}{2}\partial_1\kappa(l(A_a) \wedge l(A_a)) - \frac{1}{2}r[l(A_a) \wedge l(A_a)]
    \end{align}
    is an adjusted connection which is fake-flat with respect to $r$. It is then clear that this map can be extended to a fully faithful functor. It is also essentially surjective, since for $(\Lambda_{ab},\tilde{A}_a,\tilde{B}_a)$ a $r$-fake-flat adjusted connection, one has that $\sigma_{ab} := s(\Lambda_{ab} - \kappa(\tilde{g}_{ab},\tilde{A}_a))$, $A_a := \pi(\tilde{A}_a)$, $B_a := s(\tilde{B}_a - \frac{1}{2}\partial_1\kappa(\tilde{A}_a \wedge \tilde{A}_a))$ is a $\mathcal{G}$-connection that maps under the above functor to an adjusted connection isomorphic to the original one. An isomorphism is given by the one-forms $\lambda_a := r(\tilde{A}_a)$ (here the fake-flatness condition is used to prove that this is indeed an isomorphism).
\end{proof}

\begin{remark}
    In \cite{RistSaeWolf} there is no definition of a \emph{connective structure} on a bundle for an adjusted crossed module. Moreover, in the cocycle equations \eqref{eq:cocycleadjcon}, $\Lambda_{ab}$ and $\tilde{A}_a$ are coupled, so in principle it is unclear how one could write such definition. Our Proposition \ref{prop:adjcon} decouples the cocycle data, so in particular it gives the first notion of \emph{connective structure} on a bundle for an adjusted central crossed module. This is important because the Chern correspondence in Theorem \ref{th:chern} is a statement about holomorphic connective structures.
\end{remark}

\section{Holomorphic principal 2-bundles and the Chern correspondence}\label{sec:2pbshol}

Let $X$ be a complex manifold and let $\mathfrak{G}$ be a complex Lie 2-group. Then \emph{holomorphic} $\mathfrak{G}$-bundles are defined as in Definition \ref{def:2pb}, demanding that all the manifolds and maps are holomorphic. In particular, a holomorphic $\mathfrak{G}$-bundle has an underlying smooth $\mathfrak{G}$-bundle and one may study the space of holomorphic structures on a given smooth $\mathfrak{G}$-bundle. In this section we provide different results that describe this space in terms of \emph{semiconnections} when $\mathfrak{G}$ arises from a holomorphic multiplicative $T$-gerbe with holomorphic connective structure. In this case one can also study the space of \emph{holomorphic connective structures} on a smooth $\mathfrak{G}$-bundle, which we also describe.

\subsection{Semiconnections on principal 2-bundles}

Fix $X$ a complex manifold, $T$, $G$ complex Lie groups with $T$ abelian and $(\mathcal{G}_{\nabla},m_{\nabla},\alpha)$ a holomorphic multiplicative $T$-gerbe over $G$ with holomorphic connective structure as in Definition \ref{def:multgercx}. In this section we will define \emph{semiconnections} on smooth $\mathcal{G}$-bundles and use them to describe holomorphic structures and holomorphic structures with holomorphic connective structures on them.

\begin{definition}
\begin{enumerate}
    \item A \emph{holomorphic principal $\mathcal{G}$-bundle} over $X$ is a principal $\mathcal{G}$-bundle $(P,\mathcal{P},\rho,\alpha^{\rho})$ over $X$ such that $P$ is a holomorphic manifold with the action of $G$ on $P$ holomorphic, $\mathcal{P} \rightarrow P$ is a holomorphic gerbe and $\rho$, $\alpha^{\rho}$ are an isomorphism and a 2-isomorphism of holomorphic gerbes. A \emph{compatible connection} $(\nabla,\nabla_{\rho},A,B)$ is a connection as in Definition \ref{def:2pbcon} for which $(\nabla,B)$ is compatible with the holomorphic structure of the gerbe $\mathcal{P} \rightarrow P$, $\nabla_{\rho}$ is a compatible connection on $\rho$ and $A \in \Omega^1(P,\mathfrak{g})$ is of type $(1,0)$. Note $F_A^{0,2} = 0$, $H^{0,3} = 0$ in this case.
    \item A \emph{holomorphic principal $\mathcal{G}_{\nabla}$-bundle} over $X$ is a holomorphic $\mathcal{G}$-bundle $(P,\mathcal{P},\rho,\alpha^{\rho})$ over $X$ with a connective structure $(\nabla,\nabla_{\rho})$ as in Definition \ref{def:2pbcon} such that $\nabla$ is a holomorphic connective structure on $\mathcal{P} \rightarrow P$ and $\nabla_{\rho}$ is a holomorphic connection on $\rho$. A \emph{compatible connection} on a holomorphic $\mathcal{G}_{\nabla}$-principal 2-bundle is a connection $(\nabla,\nabla_{\rho},A,B)$ whose connective structure is the given one and such that $A \in \Omega^1(P,\mathfrak{g})$ is of type $(1,0)$ and $B$ is compatible with $\nabla$. Note $F_A^{0,2} = 0$, $H^{1,2+0,3} = 0$ in this case.
\end{enumerate}
\end{definition}

In terms of the cocycle data from Sections \ref{sec:cocycle2pbsm} and \ref{sec:cocyclecon}, a holomorphic $\mathcal{G}$-bundle is one for which $g_{ab}$, $s^{ab}_{i_1j_1}$ and $t^{abc}_{i_2}$ can be chosen to be holomorphic. In that case, a compatible connection is one for which $\sigma^{ab}_{i_1}$, $A_a$ can be chosen to be of type $(1,0)$ and $B_a$ can be chosen to be of type $(2,0)+(1,1)$. On the other hand, a holomorphic $\mathcal{G}_{\nabla}$-bundle is one for which $g_{ab}$, $s^{ab}_{i_1j_1}$, $t^{abc}_{i_2}$, $\sigma^{ab}_{i_1}$ can be chosen to be holomorphic and in that case a compatible connection is one for which $A_a$ can be chosen to be of type $(1,0)$ and $B_a$ can be chosen to be of type $(2,0)$. It can be proven similarly as in Proposition \ref{prop:connections} that compatible connections on holomorphic $\mathcal{G}$-bundles and on holomorphic $\mathcal{G}_{\nabla}$-bundles always exist.

Recall that for any complex Lie group $G$ the category of holomorphic $G$-bundles is equivalent to the category of smooth $G$-bundles with \emph{semiconnections}. Here a semiconnection on a smooth $G$-bundle $P$ may be defined as an equivalence class of smooth $G$-connections $A \in \Omega^1(P,\mathfrak{g})$ such that $F_A^{0,2} = 0$, where we identify $A_1 \sim A_2$ if $A_1-A_2 \in \Omega^{1,0}(X,ad \,P)$. Recall also that we proved the equivalent statement for gerbes in Proposition \ref{prop:dolbeaultgerbes}, based on Definition \ref{def:dolbgerbes}, which justifies the following definitions.


\begin{definition}\label{def:dol2pbs}
    The bicategory of smooth $\mathcal{G}$-bundles with 1-semiconnections is defined in the following way.
    \begin{enumerate}
        \item A \emph{smooth $\mathcal{G}$-bundle with 1-semiconnection} $(P,\mathcal{P},\rho,\alpha^{\rho},D_A,D_B,D^{\rho})$ is a smooth $\mathcal{G}$-bundle $(P,\mathcal{P},\rho,\alpha^{\rho})$ with:
        \begin{itemize}
            \item A semiconnection $D_A$ on the $G$-bundle $P \rightarrow X$.
            \item A 1-semiconnection $D_B$ on the gerbe $\mathcal{P} \rightarrow P$, where we see $P$ as a complex manifold with the complex structure induced by $D_A$.
            \item A 1-semiconnection $D^{\rho}$ on the isomorphism of gerbes with 1-semiconnection $\rho:(p^{\ast}\mathcal{P} \otimes g^{\ast}\mathcal{G},p^{\ast}D_B \otimes g^{\ast}(\Theta^L)) \rightarrow ((pg)^{\ast}\mathcal{P},(pg)^{\ast}D_B)$ that is preserved by $\alpha^{\rho}$. Here we see $\Theta^L$ as a 1-semiconnection by Proposition \ref{prop:pairingcx}.
        \end{itemize} 
        \item An \emph{isomorphism of smooth $\mathcal{G}$-bundles with 1-semiconnections} $$(u,\phi,\alpha^{\phi},D^{\phi}):(P^1,\mathcal{P}^1,\rho^1,\alpha^{\rho,1},D_A^1,D_B^1,D^{\rho^1}) \rightarrow (P^2,\mathcal{P}^2,\rho^2,\alpha^{\rho,2},D_A^2,D_B^2,D^{\rho^2})$$ is an isomorphism of smooth $\mathcal{G}$-bundles $(u,\phi,\alpha^{\phi})$ such that $u^{\ast}D^2_A = D^1_A$ with a 1-semiconnection $D^{\phi}$ on the isomorphism of gerbes with 1-semiconnections $\phi:(\mathcal{P}^1,D_1) \rightarrow (u^{\ast}\mathcal{P}^2,u^{\ast}D_2)$ that is preserved by $\alpha^{\phi}$.
        \item A \emph{2-isomorphism of smooth $\mathcal{G}$-bundles with 1-semiconnections} $\psi:(u,\phi,\alpha^{\phi},D^{\phi}) \Rightarrow (u',\phi',\alpha^{\phi'},D^{\phi'})$ is a 2-isomorphism $\psi$ of smooth $\mathcal{G}$-bundles such that $\psi: \phi \Rightarrow \phi'$ preserves the 1-semiconnections $D^{\phi}$, $D^{\phi'}$. 
    \end{enumerate} 
    We define in an analogous way the bicategory of smooth $\mathcal{G}$-bundles with 2-semiconnections. 
    For a fixed smooth $\mathcal{G}_{\nabla}$-bundle $\mathcal{P}_{\nabla}$ we write $\mathcal{D}(\mathcal{P}_{\nabla})$ for the set of 1-semiconnections $(D_A,D_B,D^{\rho})$ such that $\{(D_B)_{ij}$, $D^{\rho}_i\}$ is the $(0,1)$-part of the given connective structure and $\mathcal{D}'(\mathcal{P}_{\nabla})$ for the set of 2-semiconnections whose underlying connective structure is the given one. 
\end{definition}

The following proposition follows directly from Definition \ref{def:dol2pbs} and Proposition \ref{prop:dolbeaultgerbes}.

\begin{proposition}\label{prop:dolbeault2pbs}
The bicategory of holomorphic $\mathcal{G}$-bundles is equivalent to the bicategory of smooth $\mathcal{G}$-bundles with 1-semiconnections. The bicategory of holomorphic $\mathcal{G}_{\nabla}$-bundles is equivalent to the bicategory of smooth $\mathcal{G}$-bundles with 2-semiconnections.
\end{proposition}

For practical purposes, it is convenient to describe 1-semiconnections and 2-semiconnections on $\mathcal{G}$-bundles in a more straightforward way than the one in Definition \ref{def:dol2pbs}. For this, recall from Proposition \ref{prop:connections} that the space $\mathcal{A}(\mathcal{P}_{\nabla})$ of connections on a $\mathcal{G}_{\nabla}$-bundle over $X$ is a torsor for the group $\Omega^1(ad \,P) \times_{\langle \cdot,\cdot \rangle} \Omega^2(X,\mathfrak{t})$ with product given by \eqref{eq:torsorproduct} and define similarly subgroups $\Omega^{1,0}(ad\,P) \times_{\langle \cdot,\cdot \rangle} \Omega^{2,0+1,1}(X,\mathfrak{t})$ and $\Omega^{1,0}(ad\,P) \times_{\langle \cdot,\cdot \rangle} \Omega^{2,0}(X,\mathfrak{t})$.

\begin{proposition}\label{prop:dolbtorsor}
    Let $\mathcal{P}_{\nabla}$ be a smooth $\mathcal{G}_{\nabla}$-bundle. There are canonical bijections
    \begin{align*}
        \mathcal{D}(\mathcal{P}_{\nabla}) &\rightarrow \{(A,B) \in \mathcal{A}(\mathcal{P}_{\nabla}) \:|\: F_A^{0,2} = 0, \, H^{0,3} = 0\}/\Omega^{1,0}(ad\,P) \times_{\langle \cdot,\cdot \rangle} \Omega^{2,0+1,1}(M,\mathfrak{t}), \\
        \mathcal{D}'(\mathcal{P}_{\nabla}) &\rightarrow \{(A,B) \in \mathcal{A}(\mathcal{P}_{\nabla})  \:|\: F_A^{0,2} = 0, \, H^{1,2+0,3} = 0\}/\Omega^{1,0}(ad\,P) \times_{\langle \cdot,\cdot \rangle} \Omega^{2,0}(M,\mathfrak{t}).
    \end{align*}
    These maps send a 1-semiconnection (resp. a 2-semiconnection) $(D_A,D_B,D^{\rho})$ to the set of all connections on $\mathcal{P}_{\nabla}$ that are compatible with the holomorphic structure (resp. holomorphic structure with holomorphic connective structure) induced by $(D_A,D_B,D^{\rho})$. 
\end{proposition}

\begin{proof}
    We prove the case of 2-semiconnections, as the other one is similar. It follows from Definition \ref{def:dol2pbs} and Remark \ref{rk:dolbgerbes} that an element of $\mathcal{D}(\mathcal{P}_{\nabla})$ is a pair $([A],[B])$ of an equivalence class of connections $A$ on the $G$-bundle $P$ such that $F_A^{0,2} = 0$ up to addition of $\Omega^{1,0}(ad \,P)$ (yielding a holomorphic structure on $P$ such that $A \in \Omega^{1,0}(P,\mathfrak{g})$ and which we use in the following), and an equivalence class of curvings $B$ on the gerbe with connective structure $\mathcal{P}_{\nabla} \rightarrow P$ whose curvature $\hat{H}$ is of type $(3,0)+(2,1)$ and such that the curvature of $\rho_{\nabla}:p^{\ast}\mathcal{P}_{\nabla} \otimes g^{\ast}\mathcal{G}_{\nabla} \rightarrow (pg)^{\ast}\mathcal{P}_{\nabla}$ with respect to $p^{\ast}B \otimes g^{\ast}\Theta^L$, $(pg)^{\ast}B$ is some $\tau_B \in \Omega^2(P \times G,\mathfrak{t})$ of type $(2,0)$, up to addition of $\Omega^{2,0}(P,\mathfrak{t})$. For such pair $([A],[B])$ choose a representing connection $A \in [A]$. We claim that we may always choose a representing curving $B \in [B]$ such that $\tau_B = -R(A)$, for $R(A)$ as in Lemma \ref{lem:chernsimons}. This is because for any representing $B \in [B]$ we have that $\tau_B + R(A) \in \Omega^{2,0}(P \times G,\mathfrak{t})$ satisfies $\delta(\tau_B + R(A)) = 0$, hence there is $b \in \Omega^{2,0}(P,\mathfrak{t})$ with $\delta b = \tau_B + R(A)$. Then for such a choice of $B$ we have $(A,B) \in \mathcal{A}(\mathcal{P}_{\nabla})$ with $F_A^{0,2} = 0$ and $H^{1,2+0,3} = \hat{H}^{1,2+0,3} + CS(A)^{1,2+0,3} = 0$; it is easy to check that this gives a bijection as above.
\end{proof}


\subsubsection{Cocycle data for 1-semiconnections and 2-semiconnections}

We proceed to give cocycle data for 1-semiconnections and 2-semiconnections. If $g_{ab}$, $s^{ab}_{i_1j_1}$, $t^{abc}_{i_2}$ is cocycle data for a smooth $\mathcal{G}$-bundle in a cover $\{X_a\}_{a \in A}$ of $X$ as in Section \ref{sec:cocycle2pbsm}, then a 1-semiconnection is given by 
\begin{equation}
    \sigma^{ab}_{i_1} \in \Omega^{0,1}(g_{ab}^{-1}(U_{i_1}^1),\mathfrak{t}), \quad  A_a \in \Omega^{0,1}(X_a,\mathfrak{g}), \quad D_a \in \Omega^{0,2}(X_a,\mathfrak{t})
\end{equation} 
such that
\begin{align}
\begin{split}
    &\sigma^{ab}_{i_1} + ((s^{ab}_{i_1j_1})^{\ast}\theta^T + g_{ab}^{\ast}A_{i_1j_1})^{0,1} = \sigma^{ab}_{j_1},\\
    &((t^{abc}_{i_2})^{\ast}\theta^T  + (g_{ab},g_{bc})^{\ast}M_{i_2})^{0,1} - \sigma^{bc}_{d_0(i_2)} + \sigma^{ac}_{d_1(i_2)} - \sigma^{ab}_{d_2(i_2)} = 0,\\
    &A_b - Ad(g_{ab}^{-1})A_a = (g_{ab}^{\ast}\theta^L)^{0,1},\\
    &D_b - D_a + \langle A_a \wedge (g_{ab}^{\ast}\theta^R)^{0,1} \rangle = (d\sigma^{ab}_{i_1} + g_{ab}^{\ast}\Theta_{i_1})^{0,2},\\
    &\bar{\partial} A_a + \frac{1}{2}[A_a \wedge A_a] = 0,\\
    &\bar{\partial} D_a + \langle \bar{\partial} A_a \wedge A_a \rangle + \frac{1}{3}\langle A_a \wedge [A_a \wedge A_a] \rangle = 0.
\end{split}
\end{align}
On the other hand, a 2-semiconnection is given by
\begin{equation}
    \sigma^{ab}_{i_1} \in \Omega^{0,1}(g_{ab}^{-1}(U_{i_1}^1),\mathfrak{t}), \quad  A_a \in \Omega^{0,1}(X_a,\mathfrak{g}), \quad D_a \in \Omega^{1,1+0,2}(X_a,\mathfrak{t})
\end{equation}
such that
\begin{align}
\begin{split}
    &\sigma^{ab}_{i_1} + (s^{ab}_{i_1j_1})^{\ast}\theta^T + g_{ab}^{\ast}A_{i_1j_1} = \sigma^{ab}_{j_1},\\
    &(t^{abc}_{i_2})^{\ast}\theta^T  + (g_{ab},g_{bc})^{\ast}M_{i_2} - \sigma^{bc}_{d_0(i_2)} + \sigma^{ac}_{d_1(i_2)} - \sigma^{ab}_{d_2(i_2)} = 0,\\
    &A_b - Ad(g_{ab}^{-1})A_a = (g_{ab}^{\ast}\theta^L)^{0,1},\\
    &D_b - D_a + \langle A_a \wedge g_{ab}^{\ast}\theta^R \rangle - \langle (g_{ab}^{\ast}\theta^L)^{1,0} \wedge A_b \rangle = (d\sigma^{ab}_{i_1} + g_{ab}^{\ast}\Theta_{i_1})^{1,1+0,2}, \\
    &\bar{\partial} A_a + \frac{1}{2}[A_a \wedge A_a] = 0,\\
    &(dD_a)^{1,2+0,3} + \langle d A_a \wedge A_a \rangle + \frac{1}{3}\langle A_a \wedge [A_a \wedge A_a] \rangle = 0.
\end{split}
\end{align}
This data is obtained by observing that, if $(\sigma^{ab}_{i_1},A_a,B_a)$ is a connection with $F_A^{0,2} = 0$ and $H^{1,2+0,3} = 0$, then $(\sigma^{ab}_{i_1},A_a^{0,1},B_a^{1,1+0,2}+\langle A_a^{1,0} \wedge A_a^{0,1} \rangle)$ satisfies the equations above, given data as above it can always be completed to a connection, and two connections differ by $a \in \Omega^{1,0}(ad \,P)$, $b \in \Omega^{2,0}(X,\mathfrak{t})$ if and only if the corresponding data $(\sigma^{ab}_{i_1},A_a^{0,1},B_a^{1,1+0,2}+\langle A_a^{1,0} \wedge A_a^{0,1} \rangle)$ coincides (and similarly for $1$-semiconnections). 

Similarly to what we discussed in Section \ref{sec:gauge}, isomorphisms of $\mathcal{G}_{\nabla}$-bundles act on 1-semiconnections and 2-semiconnections. We proceed to give a formula for this action in terms of cocycle data. This can be obtained from the formula for the action on connections of Section \ref{sec:gauge} and the maps above that send connections to semiconnections.

Let $\mathcal{P}_{\nabla}^1$ and $\mathcal{P}_{\nabla}^2$ be $\mathcal{G}_{\nabla}$-bundles described by cocycle data $(g_{ab}^1,\sigma_{ab,\nabla}^1,\tau_{abc}^1)$, $(g_{ab}^2,\sigma_{ab,\nabla}^2,\tau_{abc}^2)$. An isomorphism $\mathcal{P}_{\nabla}^1 \rightarrow \mathcal{P}_{\nabla}^2$ described by $(\varphi_a,\Phi_{a,\nabla},\psi_{ab})$ as in Section \ref{sec:gauge} acts on a 1-semiconnection $(A_a^1,D_a^1)$ for $\mathcal{P}_{\nabla}^1$ by sending it to the 1-semiconnection $(A_a^2,D_a^2)$ for $\mathcal{P}_{\nabla}^2$ with
\begin{align}
    A_a^2 &= Ad(\varphi_a)A_a^1 - (\varphi_a^{\ast}\theta^R)^{0,1},\\
    D_a^2 &= D_a^1 - \langle (\varphi_a^{\ast}\theta^L)^{0,1} \wedge A_a^1 \rangle  - Curv(\Phi_{a,\nabla})^{0,2},     
\end{align}
and it acts on a 2-semiconnection $(A_a^1,D_a^1)$ by sending it to the 2-semiconnection $(A_a^2,D_a^2)$ for $\mathcal{P}_{\nabla}^2$ with
\begin{align}
    A_a^2 &= Ad(\varphi_a)A_a^1 - (\varphi_a^{\ast}\theta^R)^{0,1},\\
    D_a^2 &= D_a^1 - \langle (\varphi_a^{\ast}\theta^L)^{0,1} \wedge A_a^1 \rangle + \langle A_a^2 \wedge (\varphi_a^{\ast}\theta^R)^{1,0} \rangle - \langle (\varphi_a^{\ast}\theta^L)^{1,0} \wedge A_a^1 \rangle  - Curv(\Phi_{a,\nabla})^{1,1+0,2}.     
\end{align}

\subsection{Complexification of principal 2-bundles and Chern correspondence}\label{sec:chern}

In this section we fix compact connected Lie groups $K$, $T_{\mathbb R}$ with $T_{\mathbb R}$ abelian, a multiplicative $T_{\mathbb R}$-gerbe $\mathcal{K}$ over $K$ and a complex manifold $X$. We write $G$, $T$ and $\mathcal{G}_{\nabla}$ for the complexifications of $K$, $T_{\mathbb R}$, $\mathcal{K}$, respectively, as in Theorem \ref{th:complexification}, and we write $j: K \rightarrow G$ for the complexification map. We will show that there is a complexification functor fom $\mathcal{K}$-bundles to $\mathcal{G}$-bundles and we will prove a relation between connections on a $\mathcal{K}$-bundle and holomorphic structures (with or without holomorphic connective structures) on its complexification. 

If $P_h \rightarrow X$ is a $K$-bundle, then its \emph{fibrewise complexification} is the smooth $G$-bundle $P_h^{\mathbb C} := P_h \times G / K$, where $k \in K$ acts as $(p,g) \cdot k = (pk,k^{-1}g)$. Note there is a canonical $K$-equivariant map $l:P_h \rightarrow P_h^{\mathbb C}$. This defines a functor from $K$-bundles to $G$-bundles that is easily improved to a functor from $K$-bundles with connection to $G$-bundles with connection which we denote by $(P_h,A_h) \mapsto (P_h^{\mathbb C},A_h^{\mathbb C})$. Recall also the fibrewise complexification of gerbes from Definition  \ref{def:complexificationgerbes}. The following proposition generalizes these two construction to $\mathcal{K}$-bundles.

\begin{proposition}\label{prop:complexification2pb}
Let $(P_h,\mathcal{P}_h,\rho_h,\alpha^{\rho_h}) \rightarrow X$ be a $\mathcal{K}$-bundle. Then, there is a unique smooth $\mathcal{G}$-bundle $(P,\mathcal{P},\rho,\alpha^{\rho}) \rightarrow X$ such that $P = P_h^{\mathbb C}$ and 
\begin{enumerate}
    \item There is an isomorphism of $T$-gerbes $\phi: \mathcal{P}_h^{\mathbb C} \rightarrow l^{\ast}\mathcal{P}$ over $P_h$.
    \item There is a 2-isomorphism of $T$-gerbes over $P_h \times K$
    \begin{equation*}
        \begin{tikzcd}[column sep = 6em]
             p^{\ast}\mathcal{P}_{h}^{\mathbb C} \otimes k^{\ast}\mathcal{K}^{\mathbb C} \ar[r,"{\rho_h^{\mathbb C}}",{name=U}] \ar[d,"{p^{\ast}\phi \otimes id}",swap] 
             & (pk)^{\ast}\mathcal{P}_{h}^{\mathbb C} \ar[d,"(pk)^{\ast}\phi"] \ar[Rightarrow, dl, "\psi"]\\
             l(p)^{\ast}\mathcal{P} \otimes j(k)^{\ast}\mathcal{G} \ar[r," {(l(p),j(k))^{\ast}\rho}",{name=D},swap] 
             & (l(p)j(k))^{\ast}\mathcal{P}
        \end{tikzcd}
    \end{equation*}
    intertwining $\alpha^{\rho_h}$ with $\alpha^{\rho}$.
\end{enumerate} 
We call $(P,\mathcal{P})$ the \emph{fibrewise complexification} of $(P_h,\mathcal{P}_h)$ and denote it by $(P_h,\mathcal{P}_h)^{\mathbb C}$.
\end{proposition}
\begin{proof}
If such $\mathcal{G}$-bundle exists then it is unique up to isomorphism because if $\mathcal{P}^1, \,\mathcal{P}^2 \rightarrow P$ are any two such bundles then $\rho_1 \otimes\rho_2^{\ast}$, $\alpha^{\rho_1} \otimes \alpha^{\rho_2^{\ast}}$ equips the $T$-gerbe $\mathcal{P}^1 \otimes (\mathcal{P}^2)^{\ast} \rightarrow P$ with the necessary descent data to give a $T$-gerbe over $X$ which is trivial precisely when $\mathcal{P}^1$ isomorphic to $\mathcal{P}^2$ as $\mathcal{G}$-bundles. But since $K \rightarrow G$ is injective this $T$-gerbe can also be obtained by descending $l^{\ast}\mathcal{P}^1 \otimes l^{\ast}\mathcal{P}_2^{\ast}$ and the maps $\phi$, $\psi$ give precisely a trivialization of the descent of this $T$-gerbe.
In order to construct the $\mathcal{G}$-bundle, choose cocycle data $\{g_{ab,h},\sigma^{ab,h},\tau^{abc,h}\}$ for $(P_h,\mathcal{P}_{\nabla,h})$ over a cover $\{X_a\}_{a \in A}$ of $X$; i.e., $g_{ab}^h: X_{ab} \rightarrow K$, $\sigma^{ab,h}$ are trivializations of $(g_{ab}^h)^{\ast}\mathcal{K}$ as a $T_{\mathbb R}$-gerbe and $\tau^{abc,h}$ are isomorphisms of trivializations of $(g_{ac}^h)^{\ast}\mathcal{K}$ $m(\sigma^{ab,h},\sigma_{bc,h}) \rightarrow \sigma^{ac,h}$ satisfying a cocycle condition. Then we define $g_{ab} := j \circ g_{ab}^h: X_{ab} \rightarrow G$ and we note that $g_{ab}^{\ast}\mathcal{G} = (g_{ab}^h)^{\ast}j^{\ast}\mathcal{G} = (g_{ab}^h)^{\ast}\mathcal{K}^{\mathbb C}$; hence $(g_{ab}, (\sigma^{ab,h})^{\mathbb C},(\tau^{abc,h})^{\mathbb C})$ is cocycle data for a $\mathcal{G}$-bundle with the above property.
\end{proof}

\begin{remark}
    Fibrewise complexification can be easily enhanced to a functor of bicategories $C: \langle \text{Smooth } \mathcal{K}-\text{bundles} \rangle \rightarrow \langle \text{Smooth } \mathcal{G}-\text{bundles} \rangle$. Similarly, there are complexification functors between the bicategories of bundles with connective structures and between the bicategories of bundles with connections which we denote by $(P_h,\mathcal{P}_{h,\nabla_h},A_h,B_h) \mapsto (P_h^{\mathbb C},\mathcal{P}_{h,\nabla_h}^{\mathbb C},A_h^{\mathbb C},B_h^{\mathbb C})$.
\end{remark}


The classical Chern correspondence \cite{Sandon} is the statement that, for $K$ a compact Lie group and $P_h$ a $K$-bundle, holomorphic structures on $P_h^{\mathbb C}$ are in bijection with connections on $P_h$ such that $F_{A_h}^{0,2} = 0$, and that this bijection is obtained by sending each holomorphic structure on $P_h^{\mathbb C}$ to the unique connection on $P_h$ whose complexification is compatible with the holomorphic structure. Theorem \ref{th:chern} provides two analogs of this correspondence for $\mathcal{K}$-bundles; in the case $K = \{1\}$ the result for $1$-semiconnections was already observed in \cite{Chat}, while the result for $2$-semiconnections appears in \cite{Gualt}. There is also a close relationship with the Chern correspodence for Courant algebroids in \cite{GarRuTi}. We introduce first some notation.

Recall enhanced connections from Definition \ref{def:2pbcon} and, using Remark \ref{rk:enhcon}, write $((A,B),g) \in \mathcal{A}^{en}(\mathcal{P}_{\nabla})$ for an enhanced connection thought of as a pair of a connection $(A,B) \in \mathcal{A}(\mathcal{P}_{\nabla})$ and a $g \in \Gamma(S^2T^{\ast}X \otimes \mathfrak{t})$. Write $\Gamma(S^{1,1}T^{\ast}X \otimes \mathfrak{t}_{\mathbb R})$ for the intersection of the symmetric $\mathfrak{t}$-valued tensors of type $(1,1)$ on $X$ with the $\mathfrak{t}_{\mathbb R}$-valued tensors. A $g \in \Gamma(S^{1,1}T^{\ast}X \otimes \mathfrak{t}_{\mathbb R})$ determines an $\omega \in \Omega^{1,1}(X,\mathfrak{t})$ by $\omega = g(J\cdot,\cdot)$, where $J$ is the complex structure on $X$. Write $d^c: \Omega^p(X,\mathfrak{t}) \rightarrow \Omega^{p+1}(X,\mathfrak{t})$ for the operator $d^c := J_{\mathfrak{t}}(\overline{\partial}-\partial)$ and note that it preserves $\mathfrak{t}_{\mathbb R}$-valued forms. 

\begin{theorem}\label{th:chern}
    Let $(\mathcal{P}_h,\nabla_h)$ be a $\mathcal{K}$-bundle with connective structure over $X$ and let $(\mathcal{P},\nabla)$ be its fibrewise complexification. Then
    \begin{enumerate}
        \item\label{it:chern1} There is a canonical bijection
        $$ \mathcal{D}(\mathcal{P}_{\nabla}) \rightarrow \{(A_h,B_h) \in \mathcal{A}(\mathcal{P}_{h,\nabla_h}) \:|\: F_{A_h}^{0,2} = 0, \,H_h^{0,3} = 0\}/\Omega^{1,1}(X,\mathfrak{t}_{\mathbb R}). $$
        This map sends a 1-semiconnection $(D_A,D_B,D^{\rho})$ to the set of connections on $\mathcal{P}_{h,\nabla_h}$ whose complexification is compatible with the holomorphic structure on $\mathcal{P}_{\nabla}$ determined by $(D_A,D_B,D^{\rho})$.
        \item\label{it:chern2} There is a canonical bijection
        $$ \mathcal{D}'(\mathcal{P}_{\nabla}) \rightarrow \{((A_h,B_h),g) \in \mathcal{A}^{en}(\mathcal{P}_{h,\nabla_h}) \:|\: g^{0,2} = 0, \, F_{A_h}^{0,2} = 0, \, \,H_h = d^c(g(J\cdot,\cdot))\}. $$
        This map sends  a 2-semiconnection $(D_A,D_B,D^{\rho})$ to the unique enhanced connection $((A_h,B_h),g) \in \mathcal{A}^{en}(\mathcal{P}_{h,\nabla_h})$ such that $g^{0,2} = 0$ and that the connection $(A_h^{\mathbb C},B_h^{\mathbb C}- J_{\mathfrak{t}}g(J\cdot,\cdot))$ on $\mathcal{P}_{\nabla}$ is compatible with the holomorphic structure with holomorphic connective structure determined by $(D_A,D_B,D^{\rho})$.
    \end{enumerate}
\end{theorem}
\begin{proof}
    We only show the proof of part \ref{it:chern2}, as the other one is similar. We shall prove that the map
    \begin{align*}
        \phi:\{(A_h,(B_h,g)) \in \mathcal{A}^{en}(\mathcal{P}_{h,\nabla_h}) \,|\,g^{0,2} = 0 \} &\rightarrow \mathcal{A}(\mathcal{P}_{\nabla}) /\Omega^{1,0}(ad\,P)  \times_{\langle \cdot,\cdot \rangle} \Omega^{2,0}(X,\mathfrak{t})\\
        ((A_h,B_h),g)  &\mapsto [A_h^{\mathbb C},B_h^{\mathbb C}-J_{\mathfrak{t}}g(J\cdot,\cdot)]
    \end{align*}
    is a bijection. The theorem will follow then from Proposition \ref{prop:dolbtorsor} and the fact that the curvature $(F_{A_h},H_h)$ of $(A_h,B_h)$ satisfies $F_{A_h}^{0,2} = 0$, $H_h = d^cg(J\cdot,\cdot)$ if and only if the curvature $(F_A,H)$ of $(A_h^{\mathbb C},B_h^{\mathbb C}-J_{\mathfrak{t}}g(J\cdot,\cdot))$ satisfies $F_{A}^{0,2} = 0$, $H^{1,2+0,3} = 0$. To see that $\phi$ is a bijection, recall from part \ref{it:con3} of Proposition \ref{prop:connections} that $\{((A_h,B_h),g)  \in \mathcal{A}^{en}(\mathcal{P}_{h,\nabla_h}) \,|\,g^{0,2} = 0 \}$ is a $\Omega^2(X,\mathfrak{t}_{\mathbb R}) \times \Gamma(S^{1,1}T^{\ast}X \otimes \mathfrak{t}_{\mathbb R})$-bundle over $\mathcal{A}(P_h)$, while $\mathcal{A}(\mathcal{P}_{\nabla}) /\Omega^{1,0}(ad\,P)  \times_{\langle \cdot,\cdot \rangle} \Omega^{2,0}(X,\mathfrak{t})$ is a $\Omega^{2}(X,\mathfrak{t})/\Omega^{2,0}(X,\mathfrak{t})$-bundle over $\mathcal{A}(P)/\Omega^{1,0}(ad\,P)$. Now the map
    \begin{align*}
        \mathcal{A}(P_h) &\rightarrow \mathcal{A}(P)/\Omega^{1,0}(ad\,P)\\
        A_h &\mapsto [A_h^{\mathbb C}] 
    \end{align*}
    is a bijection, while the map
    \begin{align*}
         \Omega^2(X,\mathfrak{t}_{\mathbb R}) \times \Gamma(S^{1,1}T^{\ast}X \otimes \mathfrak{t}_{\mathbb R}) &\rightarrow \Omega^{2}(X,\mathfrak{t})/\Omega^{2,0}(X,\mathfrak{t})\\
        (b,g) &\mapsto [b-J_{\mathfrak{t}}g(J\cdot,\cdot)] 
    \end{align*}
    is an isomorphism of groups. Thus $\phi$ may be regarded as an equivariant map between affine bundles with the same fiber over the same space and so it is a bijection.
\end{proof}

We obtain the following corollary, which can be interpreted as saying that holomorphic $\mathcal{G}_{\nabla}$-bundles with a reduction to a $\mathcal{K}_{\nabla}$-bundle are the geometric objects prequantizing the Hermitian metrics proposed by Yau \cite{Yau} as a generalization of K\"ahler metrics with the potential to fulfill Reid's fantasy \cite{Reid}.

\begin{corollary}
    Let $(\mathcal{P}_h,\nabla_h)$ be a $\mathcal{K}$-bundle with connective structure over $X$. A holomorphic structure with holomorphic connective structure on its complexification determines an $\omega \in \Omega^{1,1}(X,\mathfrak{t}_{\mathbb R})$ such that
    \begin{equation}
        dd^c\omega - \langle F_h \wedge F_h \rangle = 0,
    \end{equation}
    where $F_h \in \Omega^2(ad\,P_h)$ is the curvature of the Chern connection on $P_h$.
\end{corollary}

\begin{definition}
For $(\mathcal{P}_{h},\nabla_h)$ a $\mathcal{K}$-bundle with connective structure and a choice of holomorphic structure with holomorphic connective structure on its complexification $(\mathcal{P},\nabla)$, the corresponding enhanced connection $((A_h,B_h),g)) \in \mathcal{A}^{en}(\mathcal{P}_{h,\nabla_h})$ from Theorem \ref{th:chern} is called the \emph{unitary Chern enhanced connection}, while the connection $(A_h^{\mathbb C},B_h^{\mathbb C}-J_{\mathfrak{t}}\omega) \in \mathcal{A}(\mathcal{P}_{\nabla})$ is called the \emph{complex Chern connection}.
\end{definition}

\begin{remark}\label{rk:chernlocalformula}
We give here a local formula for the Chern connections corresponding to a holomorphic structure with holomorphic connective structure on the complexification of a $\mathcal{K}_{\nabla}$-bundle $(P_h,\mathcal{P}_{\nabla,h})$. In particular, checking that these formulas satisfy the properties of Theorem \ref{th:chern} provides an alternative proof for it. We do it in terms of the cocycle description for 2-bundles in a cover $\{X_a\}_a$ of $X$ from Section \ref{sec:cocycle2pbsm}; 

Let $g_{ab}^h, \,\sigma^{ab,h}_{\nabla}, \, \tau^{abc,h}_{\nabla}$ be cocycle data for $(P_h,\mathcal{P}_{\nabla,h})$ in a cover $\{X_a\}_a$ of $X$ as in Sections \ref{sec:cocycle2pbsm} and \ref{sec:connections}, and let $g_{ab}, \, \sigma^{ab}_{\nabla}, \, \tau^{abc}_{\nabla}$ be holomorphic cocycle data for the holomorphic structure with holomorphic connective structure on its complexification. Then the two cocycle descriptions are related by a complex gauge transformation $(\{\varphi_a\},\{\Phi_{a,\nabla}\},\{\psi_{ab}\})$ as in Section \ref{sec:gauge}. This is
\begin{enumerate}
    \item $\varphi_a: M_a \rightarrow G$ satisfying $\varphi_a g_{ab} = g_{ab}^h \varphi_b$,
    \item $\Phi_{a,\nabla}$ trivializations with connection of $\varphi^{\ast}_a\mathcal{G}_{\nabla}$,
    \item $\psi_{ab}$ flat isomorphisms of trivializations $\psi_{ab}:m(\Phi_{a,\nabla},\sigma^{ab}_{\nabla}) \Rightarrow m(\sigma^{ab,h}_{\nabla},\Phi_{b,\nabla})$ satisfying a certain cocycle equation.
\end{enumerate}
Let $A^h_a$ be the local one-forms of the Chern connection of $P_h \rightarrow P$ in the unitary frame and let $A_a$ be the local one-forms in the holomorphic frame. The classical formula states that\footnote{Here we write $\overline{\cdot}: \mathfrak{g} \rightarrow \mathfrak{g}$, $\overline{\cdot}: \mathfrak{t} \rightarrow \mathfrak{t}$ for the $\mathbb C$-antilinear involutions that leave $J_{\mathfrak{g}}\mathfrak{k}$ and $J_{\mathfrak{t}}\mathfrak{t}_{\mathbb R}$ invariant, repectively, and we write $Im: \mathfrak{t} \rightarrow \mathfrak{t}_{\mathbb R}$, $Re: \mathfrak{t} \rightarrow J_{\mathfrak{t}}\mathfrak{t}_{\mathbb R}$ for the two projections, in analogy to the case $\mathfrak{t}_{\mathbb R} = i\mathbb R$, $\mathfrak{t} = \mathbb C$.}
\begin{align}
    A^h_a &= -(\varphi_a^{\ast}\theta^R)^{0,1} + \overline{(\varphi_a^{\ast}\theta^{R})^{0,1}},\\
    A_a &= (\varphi_a^{\ast}\theta^L)^{1,0} + \overline{(\varphi_a^{\ast}\theta^L)^{0,1}}.
\end{align}
On the other hand, the formulas for the unitary Chern enhanced connection $(B^h_a,g)$ (in the unitary frame, we write $\omega = g(J\cdot,\cdot)$) and the complex Chern connection $B_a$ (in the holomorphic frame) are
\begin{align}
    B^h_a &= -J_{\mathfrak{t}}Im\left( b_a^{1,1} + 2b_a^{0,2} - \langle (\varphi_a^{\ast}\theta^L)^{1,0} \wedge (\varphi_a^{\ast}\theta^L)^{0,1} \rangle - \langle (\overline{\varphi_a^{\ast}\theta^L)^{1,0}} \wedge (\varphi_a^{\ast}\theta^L)^{0,1} \rangle \right),\\
    \omega &= -J_{\mathfrak{t}}Re \left( b_a^{1,1} - \langle (\varphi_a^{\ast}\theta^L)^{1,0} \wedge (\varphi_a^{\ast}\theta^L)^{0,1} \rangle \right),\\
    B_a &= b_a^{2,0} + \overline{b_a^{0,2}} + \langle (\varphi_a^{\ast}\theta^L)^{1,0} \wedge \overline{(\varphi_a^{\ast}\theta^L)^{0,1}} \rangle,
\end{align}
where $b_a := Curv(\Phi_{a,\nabla})$. These formulas follow from noting that
$$(Curv(\sigma_{ab,\nabla}) - \langle A_a \wedge g_{ab}^{\ast}\theta^R \rangle) - (Curv(\sigma_{ab,\nabla}^h) - \langle A_a^h \wedge (g_{ab}^h)^{\ast}\theta^R \rangle)  =  (b_b + \langle \varphi_b^{\ast}\theta^L \wedge A_b \rangle) - (b_a + \langle \varphi_a^{\ast}\theta^L \wedge A_a \rangle)$$
and taking $-J_{\mathfrak{t}}Im((\cdot)^{1,1} + 2( \cdot)^{0,2})$, $(\cdot)^{2,0} + \overline{(\cdot)^{0,2}}$ and $-J_{\mathfrak{t}}Re((\cdot)^{1,1})$ on both sides.
\end{remark}

The classical Chern correspondence also states that, if $P_h^1$, $P_h^2$ are $K$-bundles with holomorphic structures on their complexifications determining Chern connections $A_h^1$, $A_h^2$, then an isomorphism $P_h^1 \rightarrow P_h^2$ is flat with respect to $A_h^1$, $A_h^2$ precisely when its complexification is holomorphic with respect to the corresponding holomorphic structures. Similarly, our Chern correspondence can  be improved to an equivalence of bicategories of which Theorem \ref{th:chern} is the result at the level of objects. To state this equivalence, consider the forgetful functors of bicategories
\begin{align*}
    F: \langle \text{Holomorphic } \mathcal{G}\text{-bundles} \rangle &\rightarrow \text{Smooth } \mathcal{G}\text{-bundles} \rangle,\\
    F_{\nabla}: \langle \text{Holomorphic } \mathcal{G}_{\nabla}\text{-bundles} \rangle &\rightarrow \text{Smooth } \mathcal{G}_{\nabla}\text{-bundles} \rangle
\end{align*}
and the complexification functors of bicategories
\begin{align*}
    C: \langle \text{Smooth } \mathcal{K}\text{-bundles} \rangle &\rightarrow \text{Smooth } \mathcal{G}\text{-bundles} \rangle,\\
    C_{\nabla}: \langle \text{Smooth } \mathcal{K}_{\nabla}\text{-bundles} \rangle &\rightarrow \text{Smooth } \mathcal{G}_{\nabla}\text{-bundles} \rangle.
\end{align*}
We define the following bicategories as categorical fibered products:
\begin{align*}
    \langle \text{Holomorphic } \mathcal{K}\text{-bundles} \rangle &:= \langle \text{Holomorphic } \mathcal{G}\text{-bundles} \rangle \prescript{}{F}{\times}^{}_C   \langle \text{Smooth } \mathcal{K}\text{-bundles} \rangle, \\
    \langle \text{Holomorphic } \mathcal{K}_{\nabla}\text{-bundles} \rangle &:= \langle \text{Holomorphic } \mathcal{G}_{\nabla}\text{-bundles} \rangle \prescript{}{F_{\nabla}}{\times}^{}_{C_{\nabla}}   \langle \text{Smooth } \mathcal{K}_{\nabla}\text{-bundles} \rangle.
\end{align*}

\begin{corollary}\label{cor:chern}
\begin{enumerate}
    \item\label{it:corchern1} The bicategory of holomorphic $\mathcal{K}$-bundles is equivalent to the bicategory $D$ with:
        \begin{itemize}
            \item An object in $D_0$ is an equivalence class of $\mathcal{K}$-bundles with connection $(\mathcal{P}_h,\nabla_h,A_h,B_h)$ whose curvature satisfies $F_{A_h}^{0,2} = 0$, $H_h^{0,3} = 0$ and we identify $(\mathcal{P}_h,\nabla_h,A_h,B_h) \sim (\mathcal{P}_h,\nabla_h,A_h,B_h+b)$ for any $b \in \Omega^{1,1}(X,\mathfrak{t}_{\mathbb R})$.
            \item Isomorphisms $(\mathcal{P}_h,\nabla_h,A_h,[B_h]) \rightarrow (\mathcal{P}_{h,2},\nabla_{h,2},A_{h,2},[B_{h,2}])$ in $D_1$ are isomorphisms of $\mathcal{K}$-bundles with connection $(\mathcal{P}_h,\nabla_h,A_h,B_h) \rightarrow (\mathcal{P}_{h,2},\nabla_{h,2},A_{h,2},B_{h,2})$ that are flat up to forms in $\Omega^{1,1}(X,\mathfrak{t}_{\mathbb R})$.
            \item 2-isomorphisms in $D_2$ are flat 2-isomorphisms between the corresponding isomorphisms of $\mathcal{K}$-bundles with connection.
        \end{itemize}
    \item\label{it:corchern2} The bicategory of holomorphic $\mathcal{K}_{\nabla}$-bundles is equivalent to the bicategory $D'$ of smooth $\mathcal{K}$-bundles with enhanced connections $((A_h,B_h),g))$ such that $g^{0,2} = 0$, $F_{A_h}^{0,2} = 0$, $H_h = d^cg(J\cdot,\cdot)$. 
\end{enumerate}
\end{corollary}
\begin{proof}
We show the proof of part \ref{it:corchern2}, as the other one follows similarly. First, it follows from Proposition \ref{prop:dolbeault2pbs} that the bicategory of holomorphic $\mathcal{K}_{\nabla}$-bundles is equivalent to the bicategory whose objects are $\mathcal{K}_{\nabla}$-bundles with 2-semiconnections in their complexifications (such that the connective structure of the 2-semiconnection is the complexification of the one in the $\mathcal{K}$-bundle); whose isomorphisms are isomorphisms of $\mathcal{K}_{\nabla}$-bundles that complexify to isomorphisms of 2-semiconnections, and whose 2-isomorphisms are 2-isomorphisms of $\mathcal{K}$-bundles preserving the 2-semiconnections.

Theorem \ref{th:chern} implies then that every object in $\langle\text{Holomorphic } \mathcal{K}_{\nabla}\text{-bundles} \rangle$ can be described by an object in $D'_0$, and conversely. An isomorphism in $\langle\text{Holomorphic } \mathcal{K}_{\nabla}\text{-bundles} \rangle $ between the $\mathcal{K}_{\nabla}$-bundles corresponding to $(\mathcal{P}_h,\nabla_h,A_h,B_h,g)$ and $(\mathcal{P}_{h,2},\nabla_{h,2},A_{h,2},B_{h,2},g_2)$ is then an isomorphism of $\mathcal{K}_{\nabla}$-bundles $(u,\varphi^u_{\nabla},\alpha^u): (\mathcal{P}_h,\nabla_h) \rightarrow (\mathcal{P}_h^2,\nabla_h^2)$ such that $u$ is flat (by the classical Chern correspondence) and that $(\varphi^u_{\nabla})^{\mathbb C}$ has curvature of type $(2,0)$ with respect to the curvings $B_{h,1} - J_{\mathfrak{t}}g_1(J\cdot,\cdot)$ and $u^{\ast}(B_{h,2} - J_{\mathfrak{t}}g_2(J\cdot,\cdot))$ (by Theorem \ref{th:chern} and Remark \ref{rk:dolbgerbes}). But its curvature with respect to these curvings equals its curvature with respect to $B_{h,1}$, $u^{\ast}B_{h,2}$ (which is $\mathfrak{t}_{\mathbb R}$-valued because so are $B_{h,1}$, $B_{h,2}$ and $\varphi^u_{\nabla}$) plus the $J_{\mathfrak{t}}\mathfrak{t}_{\mathbb R}$-valued form $J_{\mathfrak{t}}(g_1(J\cdot,\cdot) - g_2(J\cdot,\cdot))$; such a sum can only be of type $(2,0)$ when both terms are zero. This means that all isomorphisms in $\langle\text{Holomorphic } \mathcal{K}_{\nabla}\text{-bundles} \rangle $ are described by isomorphisms in $D'$, and conversely. The same holds trivially for 2-isomorphisms, which concludes the proof.
\end{proof}

\subsection{Examples of holomorphic principal 2-bundles}

In this section we provide some examples of holomorphic principal 2-bundles with holomorphic connective structure. These are particularly interesting in non-K\"ahler geometry because of the following result.

\begin{proposition}\cite{BryHol}
    Let $X$ be a complex manifold satisfying the $\partial\overline{\partial}$-lemma and let $\mathcal{L}_{\nabla} \rightarrow X$ be a holomorphic $T$-gerbe with holomorphic connective structure. Then $\mathcal{L}$ is flat.
\end{proposition}
\begin{proof}
    Any $T$-gerbe with $\mathfrak{t}$-connective structure admits a reduction to a $T_{\mathbb R}$-gerbe with $\mathfrak{t}_{\mathbb R}$-connective structure. Then we may represent the de Rham class of $\mathcal{L}$ by $[-2i\partial \omega] \in H^3(X,\mathfrak{t})$, where $\omega = g(J\cdot,\cdot) \in \Omega^{1,1}(X,\mathfrak{t}_{\mathbb R})$ is determined by the reduction and the holomorphic structure with holomorphic connective structure as in Theorem \ref{th:chern}. But then by the $\partial\overline{\partial}$-lemma we have $[-2i\partial \omega] = 0 \in H^3(X,\mathfrak{t})$.
\end{proof}

\emph{Hopf surfaces} provide the simplest examples of non K\"ahler compact complex manifolds. We will focus on the Hopf surface $X = \mathbb C^2 \setminus \{(0,0)\} / \mathbb Z$, where $\mathbb Z$ acts as $k \cdot (x_1,x_2) = (2^kx_1,2^kx_2)$. As a smooth manifold, it is diffeomorphic to $S^1 \times S^3$. We proceed to give some examples of holomorphic principal 2-bundles over $X$.

\subsubsection{$\mathbb C^{\ast}$-Gerbes over the Hopf surface}\label{sec:gerbeshopf}
The Hodge numbers $h^{0,2}$, $h^{1,1}$ and $h^{1,2}$ of $X$ vanish and so from the exact sequences
\begin{align*}
    H^2(X,\mathcal{O}_{\mathbb C}) &\rightarrow H^2(X,\mathcal{O}_{\mathbb C^{\ast}}) \rightarrow H^3(X,\mathbb Z) \rightarrow 0,\\
    \mathbb H^2(X, 0 \rightarrow \Omega^1_{\mathbb C,\overline{\partial}-cl}) \rightarrow \mathbb H^2(X, &\mathcal{O}_{\mathbb C^{\ast}} \rightarrow \Omega^1_{\mathbb C,\overline{\partial}-cl}) \rightarrow \mathbb H^2(X, \mathcal{O}_{\mathbb C^{\ast}} \rightarrow 0) \rightarrow \mathbb H^3(X, 0 \rightarrow \Omega^1_{\mathbb C,\overline{\partial}-cl}),
\end{align*}
we see that holomorphic $\mathbb C^{\ast}$-gerbes with holomorphic connective structure over $X$ are completely classified by $H^3(X,\mathbb Z) = \mathbb Z$. Theorem \ref{th:chern} implies that a generator of this group can be described either by giving cocycle data for a $U(1)$-gerbe with connection of curvature $H_h = d^c\omega$, for some $\omega \in \Omega^{1,1}(X,\mathbb R)$, or by giving cocycle data for a holomorphic $\mathbb C^{\ast}$-gerbe with holomorphic connective structure. The first is done in \cite{Gualt}, with
\begin{equation}
    \omega = \frac{1}{4\pi i} \frac{dx_1 \wedge \overline{dx_1} + dx_2 \wedge \overline{dx_2}}{|x_1|^2 + |x_2|^2},
\end{equation} 
the Hermitian form associated to the \emph{Boothby metric} on $X$. Here we give explicit holomorphic cocycle data for the same gerbe. For this we need some notation. Let 
\begin{align*}
    log_{+}: \{z \in \mathbb C^{\ast}: \: z/|z| \neq 1\} &\rightarrow \{\zeta \in \mathbb C: \: 0 < Im(\zeta) < 2\pi \}, \\
    log_{-}: \{z \in \mathbb C^{\ast}: \: z/|z| \neq -1\} &\rightarrow \{\zeta \in \mathbb C: \: -\pi < Im(\zeta) < \pi \}
\end{align*}
be two branches of the logarithm such that 
\begin{equation*}
    log_+(z) - log_-(z) =
        \begin{cases}
        0      & \quad \text{if } \quad z \in \mathbb H \\
        2\pi i & \quad \text{if } \quad \bar{z} \in \mathbb H,
        \end{cases}
\end{equation*}
where $\mathbb H = \{z \in \mathbb C: \: Im(z) > 0\}$. We write $[x_1,x_2] \in X$ for the class of $(x_1,x_2) \in \mathbb C^2 \setminus \{0\}$ and define the open cover
\begin{align*}
    U_{1^+} = \{[x_1,x_2] \in X \:|\: x_1 \neq 0, \, x_1/|x_1| \neq 1\}, \quad &U_{2^+} = \{[x_1,x_2] \in X \:|\: x_2 \neq 0, \, x_2/|x_2| \neq 1\}, \\
    U_{1^-} = \{[x_1,x_2] \in X \:|\: x_1 \neq 0, \, x_1/|x_1| \neq -1\}, \quad &U_{2^-} = \{[x_1,x_2] \in X \:|\: x_2 \neq 0, \, x_2/|x_2| \neq -1\}.
\end{align*}
The holomorphic gerbe with holomorphic connective structure is described in this cover by the following cocycle data:
\begin{align*}
    &\lambda_{1^{\pm}2^+2^-} = 1, \quad &\lambda_{1^{+}1^-2^{\pm}} = \begin{cases}
        1      & \quad \text{if } \quad x_1 \in \mathbb H \\
        x_1/x_2 & \quad \text{if } \quad \bar{x_1} \in \mathbb H
        \end{cases}, \\
        &A_{1^+1^-} = \begin{cases}
        0      & \quad \text{if } \quad x_1 \in \mathbb H \\
        dlog(x_1) & \quad \text{if } \quad \bar{x_1} \in \mathbb H
        \end{cases}, \quad &A_{2^+2^-} = \begin{cases}
        0      & \quad \text{if } \quad x_2 \in \mathbb H \\
        dlog(x_2) & \quad \text{if } \quad \bar{x_2} \in \mathbb H
        \end{cases}, \\
        &A_{12} = \frac{1}{2\pi i}\left(log(x_1) - log(x_2) \right)dlog(x_2).
\end{align*}
Here the last equation means $A_{1^+2^-} = \frac{1}{2\pi i}\left(log_+(x_1) - log_-(x_2) \right)dlog(x_2)$ and similarly for $A_{1^+2^+}$, $A_{1^-2^+}$, $A_{1^-2^-}$. 



\subsubsection{Higher extensions of $SL(2,\mathbb C)$-bundles over the Hopf surface}

Consider the compact Lie group $SU(2)$, whose complexification is $SL(2,\mathbb C)$. For each level $l \in \mathbb Z = H^3(SU(2),\mathbb Z)$, $l \neq 0$, we consider the corresponding central extension $BU(1) \rightarrow \mathcal{K}_l \rightarrow SU(2)$ from Example \ref{ex:string}, equipped with its unique connective structure whose corresponding pairing is a multiple of the Killing form by Example \ref{ex:stringconstr}. Let us study $\mathcal{K}_l$-bundles over the Hopf surface $X$. First, note that $SU(2)$-bundles over a 4-manifold are totally classified by their Pontryagin class, and it follows from Proposition \ref{prop:pontryagin1} that the characteristic class controlling the lift of $SU(2)$-bundles to $\mathcal{K}_l$-bundles is an integer multiple of the Pontryagin class. Since $H^4(X,\mathbb Z)$ has no torsion, this means that the unique $SU(2)$-bundle admitting a lift to a $\mathcal{K}_l$-bundle is the trivial one.

It follows by a descent argument similar to the one in the proof of Proposition \ref{prop:complexification2pb} that $\mathcal{K}_l$-bundles over $X$ whose underlying $SU(2)$-bundle is $X \times SU(2)$ are in bijection with $U(1)$-gerbes on $X$, where each gerbe $\mathcal{L}_h \rightarrow X$ is sent to the $\mathcal{K}_l$-bundle 
$$(X \times SU(2),\pi_X^{\ast}\mathcal{L}_h \otimes \pi_{SU(2)}^{\ast}\mathcal{G}_l, (g_1,g_2)^{\ast}m,(g_1,g_2,g_3)^{\ast}\alpha),$$
where $m$ and $\alpha$ are the product and associator of $\mathcal{K}_l$. We claim that, for any choice of $\mathcal{L}_h$ and any holomorphic structure on $X \times SL(2,\mathbb C)$, we can always lift it to a unique holomorphic structure with holomorphic connective structure on the complexification of the corresponding $\mathcal{K}_l$-bundle. Indeed, if such lift exists, then it must be unique because any two non-isomorphic choices would differ by a non-trivial holomorphic structure with holomorphic connective structure on the topologically trivial $\mathbb C^{\ast}$-gerbe over $X$, but there is no such structure by the arguments in Section \ref{sec:gerbeshopf}. To prove existence, take any connection $(A_h,B_h)$ on the $\mathcal{K}_l$-bundle such that $A_h$ is the Chern connection of the holomorphic structure (with curvature $F_h \in \Omega^{1,1}(X,\mathbb R)$). Then the curvature $3$-form $H_h \in \Omega^3(X,\mathbb R)$ satisfies $dH_h = \langle F_h \wedge F_h \rangle \in \Omega^{2,2}(X,\mathbb R)$, hence $\overline{\partial}H_h^{1,2} = 0$ and, since $h^{1,2} = 0$, we obtain $H^{1,2}_h = i\overline{\partial}\omega$ for some $\omega \in \Omega^{1,1}(X,\mathbb R)$. This yields $H_h = d^c\omega$ and so we obtain a holomorphic structure with holomorphic connective structure by Theorem \ref{th:chern}.

We can also consider higher extensions of topologically non-trivial $SL(2,\mathbb C)$-bundles by using the following toy model for the Hull-Strominger system. Given a pair $(P^1,P^2)$ of $SU(2)$-bundles with Pontryagin classes related by $p_1(P^1) = d \cdot p_1(P^2)$ for some $d \in \mathbb Z$, we construct an extension  $BU(1) \rightarrow \mathcal{K}_{1,d} \rightarrow SU(2) \times SU(2)$ by integrating as in Theorem \ref{th:MaurerCartan} the diagonal pairing 
$$\langle \cdot,\cdot \rangle: (\mathfrak{su}(2) \oplus \mathfrak{su}(2)) \otimes (\mathfrak{su}(2) \oplus \mathfrak{su}(2)) \rightarrow \mathbb R$$
that equals the Killing form multiplied by $1/8$ in the first $\mathfrak{su}(2)$ factor and the Killing form multiplied by $-d/8$ in the second $\mathfrak{su}(2)$ factor. Then Proposition \ref{prop:pontryagin1} implies that $P^1 \times P^2$ can be lifted to a smooth $\mathcal{K}_{1,d}$-bundle. By similar arguments as in the preceding paragraph, the set of such smooth lifts is a torsor for $H^3(X,\mathbb Z) = \mathbb Z$ and, once such a lift is chosen, each holomorphic structure in the complexification of $P^1 \times P^2$ can be lifted in a unique way to a holomorphic structure with holomorphic connective structure in the complexification of the $\mathcal{K}_{1,d}$-bundle. This relates the moduli spaces of holomorphic principal 2-bundles with holomorphic connective structures over the Hopf surface to the moduli spaces of holomorphic vector bundles studied in \cite{Moraru}.

\appendix

\section{Gerbes as cocycle data}\label{ap:gerbes}

\subsection{Smooth gerbes}
\emph{Gerbes} are higher analogs of line bundles: geometric representatives of cohomology classes in $H^3(M,\mathbb Z)$. They were first defined by Giraud \cite{Giraud} as certain \emph{sheaves of groupoids} and then studied in the work of Brylinski \cite{BryLoop,BryHol,BryGauge}, who in particular set up a working definition in terms of \v{C}ech cocycle data which has been extensively used since then, as well as the essentially equivalent definitions of \emph{Hitchin-Chatterjee} gerbes \cite{Chat,Hit} and \emph{bundle gerbes} \cite{Mur}, while the most recent approach is to treat them as principal 2-bundles with structure 2-group $BU(1)$ \cite{NikWal2pbs}. In this appendix we recall familiar facts from this theory in the \v{C}ech cocycle approach. These are \emph{working} definitions, in the sense that we describe geometric objects in terms of the cocycle data with which they can be constructed, without gluing this data. The gluing is performed in Proposition \ref{prop:2pbmult}.

Fix during this section a manifold $M$ and an abelian Lie group $T$ with Lie algebra $\mathfrak{t}$. In the presence of an open cover $\{U_i\}_{i \in I}$ of $M$ we will often write $U_{i_1...i_k} = \bigcap_{j= 1}^k U_{i_j}$ for the finite intersections of the open sets.

\begin{definition}\label{def:gerbes}
A \emph{$T$-gerbe} over $M$ is the data of an open cover $\{U_i\}_{i \in I}$ of $M$ and a $T$-valued \v{C}ech 2-cocycle in this cover; i.e., functions $\lambda_{ijk}:U_{ijk} \rightarrow T$ satisfying $\lambda_{ikl}\lambda_{ijk} = \lambda_{ijl}\lambda_{jkl}$, modulo refinements. A \emph{connective structure} on it is the data of 1-forms $A_{ij} \in \Omega^1(U_{ij},\mathfrak{t})$ such that $A_{ij} - A_{ik} + A_{jk} = \lambda_{ijk}^{\ast}\theta^T$, for $\theta^T \in \Omega^1(T,\mathfrak{t})$ the Maurer-Cartan form on $T$, modulo refinements. A \emph{curving} for it is a collection of 2-forms $B_i \in \Omega^2(U_i,\mathfrak{t})$ such that $B_i - B_j = dA_{ij}$, modulo refinements. A \emph{connection} for a gerbe is a connective structure with a curving. Its \emph{curvature} is $H \in \Omega^3_{cl}(M,\mathfrak{t})$ given locally by $H_{|U_i} = dB_i$. 

Given two gerbes described by $(\{U^1_i\}_{i\in I_1},\{\lambda^1_{ijk}\}_{i,j,k \in I_1})$ and $(\{U^2_i\}_{i\in I_2},\{\lambda^2_{ijk}\}_{i,j,k \in I_2})$, respectively, an \emph{isomorphism} between them is the data of $\{V_a\}_{a \in A}$ a common refinement of $\{U^1_i\}_{i\in I_1}$ and $\{U^2_i\}_{i\in I_2}$ (with refinement maps $i_1:A \rightarrow I_1$ and $i_2:A \rightarrow I_2$) and functions $s_{ab}: V_{ab} \rightarrow T$ satisfying $s_{ac}\lambda^1_{i_1(a)i_1(b)i_1(c)} = \lambda^2_{i_2(a)i_2(b)i_2(c)}s_{bc}s_{ab}$. If the gerbes have connective structures $A_{ij}^1$, $A_{ij}^2$, then we say that they are \emph{preserved} when $0 = A^1_{i_1(a)i_1(b)} - A^2_{i_2(a)i_2(b)} - s_{ab}^{\ast}\theta^T$; we also define a \emph{connection} on the isomorphism (even when it does not preserve them) as a collection of $1$-forms $A_a \in \Omega^1(V_a,\mathfrak{t})$ satisfying $A_a - A_b = A^1_{i_1(a)i_1(b)} - A^2_{i_2(a)i_2(b)} - s_{ab}^{\ast}\theta^T$. An isomorphism of gerbes with a connection is also called an \emph{isomorphism of gerbes with connective structure}, while a connection on the identity isomorphism is also called an \emph{isomorphism of connective structures}. If the gerbes also have curvings $B_i^1$, $B_i^2$ then the \emph{curvature} of the isomorphism with connection $(\{s_{ab}\}_{a,b},\{A_a\}_a)$ with respect to $B^1_i$, $B^2_i$ is $F \in \Omega^2(M,\mathfrak{t})$ defined locally by $F_{|V_a} := dA_a - B^1_{i_1(a)} + B^2_{i_2(a)}$; note that it satisfies $dF = H_2 - H_1$.

Given two isomorphisms $(\{V_a\},\{s_{ab}\})$, $(\{V'_a\},\{s'_{ab}\})$ between the same gerbes, a \emph{2-isomorphism} between them is the data of a common refinement $\{W_r\}_{r \in R}$ of $\{V_a\}_{a \in A}$ and $\{V'_a\}_{a \in A'}$ such that the refinement maps $R \rightarrow A \rightarrow I_n$, $R \rightarrow A' \rightarrow I_n$ coincide for $n = 1,\,2$ and functions $t_r: W_r \rightarrow T$ such that $t_rs_{a(r)a(s)} = s'_{a'(r)a'(s)}t_s$. If the gerbes have connective structures $A_{ij}^1$, $A_{ij}^2$ and the isomorphisms have connections $A_a$, $A_a'$ then this is a \emph{2-isomorphism of gerbes with connective structure} or a \emph{flat 2-isomorphism} when the 1-form $\eta \in \Omega^1(M,\mathfrak{t})$ given locally by $\eta = A'_{a'(r)} - A_{a(r)} + t_r^{\ast}\theta^T$ is actually $\eta = 0$. Gerbes (with or without connective structures and curvings), their isomorphisms and 2-isomorphisms form a bicategory.
\end{definition}

\begin{remark}\label{rk:bundlegerbes}
    The bicategory of Murray's \emph{bundle gerbes} \cite{Mur} is defined in a similar way but replacing the use of a cover of $M$ by a general surjective submersion $\pi: Y \rightarrow M$. For such a surjective submersion, write $Y^{[n]} := Y \times_M Y \times_M ... \times_M Y$ for the $n$th fibered product over $M$ and $d_j: Y^{[n]} \rightarrow Y^{[n-1]}$, $j=0, \,..., \,n-1$ for the map that forgets the $j$-th point. Then a $T$-gerbe over $M$ can be defined as a principal $T$-bundle $L \rightarrow Y^{[2]}$ with an isomorphism $\lambda:d_2^{\ast}L \otimes d_0^{\ast}L \rightarrow d_1^{\ast}L$ of $T$-bundles over $Y^{[3]}$ satisfying $d_1^{\ast}\lambda \circ d_3^{\ast}\lambda = d_2^{\ast}\lambda \circ d_0^{\ast}\lambda$ over $Y^{[4]}$; a connective structure on it is a $T$-connection $\nabla$ on $L \rightarrow Y^{[2]}$ such that $\lambda$ is flat and a curving for it is a 2-form $B \in \Omega^2(Y,\mathfrak{t})$ such that $d_1^{\ast}B - d_0^{\ast}B = Curv(L,\nabla)$. Isomorphisms and 2-isomorphisms of bundle gerbes can be defined similarly as in Definition \ref{def:gerbes}; the resulting bicategory is equivalent to the one from Definition \ref{def:gerbes}.
\end{remark}

Let $T$ be a connected abelian Lie group and let $Z \subset \mathfrak{t}$ be the kernel of the exponential map $exp:\mathfrak{t} \rightarrow T$. From the cocycle description of gerbes one can easily prove the following proposition. 
\begin{proposition}\label{prop:gerbes}
    \begin{enumerate}
        \item Every $T$-gerbe admits a connective structure and curving. 
        \item $T$-gerbes over $M$ are classifed by $H^3(M,Z)$ and the class of a gerbe is represented in de Rham cohomology by taking the curvature of any curving.
        \item\label{it:flatgerbes} A gerbe admits a flat connection if and only if it admits locally constant cocycle data. An isomorphism of flat gerbes admits a flat connection if and only if it admits locally constant cocycle data in the same frame in which the gerbes are described by locally constant cocycle data. A 2-isomorphism between isomorphisms with flat connections is flat if and only if it is described by locally constant functions in the same frame in which the gerbes and the isomorphisms are described by locally constant cocycle data.
    \end{enumerate}

\end{proposition}

\subsection{Holomorphic gerbes and holomorphic connective structures}\label{sec:cxgerbes}

Let $X$ be a complex manifold and $T$ a complex abelian Lie group with Lie algebra $\mathfrak{t}$.

\begin{definition}
A \emph{holomorphic $T$-gerbe} over $X$ is a $T$-gerbe $(\{U_i\},\{\lambda_{ijk}\})$ such that $\lambda_{ijk}$ are holomorphic functions. In this case, a \emph{compatible connective structure} is a connective structure $\{A_{ij}\}$ such that $A_{ij} \in \Omega^{1,0}(U_{ij},\mathfrak{t})$ and a \emph{compatible curving} is a curving $\{B_i\}$ such that $B_i \in \Omega^{2,0+1,1}(U_i,\mathfrak{t})$. A \emph{holomorphic connective structure} is a connective structure $\{A_{ij}\}$ such that $A_{ij} \in \Omega^{1,0}(U_{ij},\mathfrak{t})$ and $\bar{\partial}A_{ij} = 0$ and in this case a curving $\{B_i\}$ is \emph{compatible with the holomorphic connective structure} if $B_i \in \Omega^{2,0}(U_i,\mathfrak{t})$. It is then a \emph{holomorphic curving} if, moreover, $\overline{\partial}B_i = 0$.

A \emph{holomorphic isomorphism} $(\{V_a\},\{s_{ab}\})$ between holomorphic $T$-gerbes is an isomorphism such that $s_{ab}$ are holomorphic. If, moreover, the holomorphic gerbes have compatible connective structures, a connection $A_a$ on a holomorphic isomorphism is \emph{compatible} if $A_a \in \Omega^{1,0}(V_a,\mathfrak{t})$; if they have holomorphic connective structures then the connection is \emph{holomorphic} when $A_a \in \Omega^{1,0}(V_a,\mathfrak{t})$ and $\bar{\partial} A_a = 0$. A \emph{holomorphic 2-isomorphism} $(\{W_r\},\{t_{r}\})$ between holomorphic isomorphisms is one for which $t_r$ are holomorphic.
\end{definition}

Write $\mathcal{O}_T$ for the sheaf of holomorphic $T$-valued functions and $\Omega^1_{\overline{\partial}-cl,\mathfrak{t}}$ for the sheaf of holomorphic $\mathfrak{t}$-valued 1-forms. It follows from the definitions (see also \cite{BryHol}) that holomorphic $T$-gerbes are classified by $H^2(X,\mathcal{O}_T)$, their automorphisms are classified by $H^1(X,\mathcal{O}_T)$ and there are $H^0(X,\mathcal{O}_T)$ 2-automorphisms of a given isomorphism. Similarly, holomorphic gerbes with holomorphic connective structure are classified by $\mathbb H^2(X,\mathcal{O}_{X,T} \rightarrow \Omega^1_{\overline{\partial}-cl,\mathfrak{t}})$, their automorphisms by $\mathbb H^1(X,\mathcal{O}_{X,T} \rightarrow \Omega^1_{\overline{\partial}-cl,\mathfrak{t}})$ and their 2-automorphisms by $\mathbb H^0(X,\mathcal{O}_{X,T} \rightarrow \Omega^1_{\overline{\partial}-cl,\mathfrak{t}})$.

Given a smooth $T$-gerbe, we say that it \emph{admits} a holomorphic structure (with holomorphic connective structure) if it is isomorphic as a smooth $T$-gerbe to the underlying smooth $T$-gerbe of a holomorphic gerbe (with holomorphic connective structure). One can characterize the different holomorphic structures that a smooth gerbe admits in terms of an analog of the classical `semiconnections' (or `Dolbeault operators') for $T$-bundles. 

\begin{definition}\label{def:dolbgerbes}
    Let $\mathcal{L}$ be a smooth $T$-gerbe described by $(\{U_i\},\{\lambda_{ijk}\})$.
    \begin{itemize}
        \item A \emph{1-semiconnection} on $\mathcal{L}$ is the data of $D_{ij} \in \Omega^{0,1}(U_{ij},\mathfrak{t})$, $D_i \in \Omega^{0,2}(U_i,\mathfrak{t})$ with $D_{ij} - D_{ik} + D_{jk} = \lambda_{ijk}^{\ast}\theta^{0,1}$, $D_i - D_j = (dD_{ij})^{0,2}$, $(dD_i)^{0,3} = 0$.
        \item A \emph{2-semiconnection} on $\mathcal{L}$ is the data of $D_{ij} \in \Omega^{1}(U_{ij},\mathfrak{t})$, $D_i \in \Omega^{1,1+0,2}(U_i,\mathfrak{t})$ with $D_{ij} - D_{ik} + D_{jk} = \lambda_{ijk}^{\ast}\theta$, $D_i - D_j = (dD_{ij})^{1,1+0,2}$, $(dD_i)^{1,2+0,3} = 0$.
    \end{itemize}
    Let $\phi:\mathcal{L}^1 \rightarrow \mathcal{L}^2$ be an isomorphism of smooth $T$-gerbes described by $(\{U_i\},\{s_{ij}\})$.
    \begin{itemize}
        \item If $\mathcal{L}^1$, $\mathcal{L}^2$ have 1-semiconnections $D^1$, $D^2$, a \emph{1-semiconnection} on $\phi$ is the data of $D^{\phi}_{i} \in \Omega^{0,1}(U_{i},\mathfrak{t})$ with $D^{\phi}_i - D^{\phi}_j = D^1_{ij} - D^2_{ij} - s_{ij}^{\ast}\theta^{0,1}$, $(dD^{\phi}_i)^{0,2} = D^1_i - D^2_i$.
        \item If $\mathcal{L}^1$, $\mathcal{L}^2$ have 2-semiconnections $D^1$, $D^2$, a \emph{2-semiconnection} on $\phi$ is the data of $D^{\phi}_{i} \in \Omega^{1}(U_{i},\mathfrak{t})$ with $D^{\phi}_i - D^{\phi}_j = D^1_{ij} - D^2_{ij} - s_{ij}^{\ast}\theta$, $(dD^{\phi}_i)^{1,1+0,2} = D^1_i - D^2_i$.
    \end{itemize}
    Let $\psi: \phi \Rightarrow \phi': \mathcal{L}^1 \rightarrow \mathcal{L}^2$ be a 2-isomorphism of smooth $T$-gerbes described by $(\{U_i\},\{t_{i}\})$.
    \begin{itemize}
        \item If $\mathcal{L}^1$, $\mathcal{L}^2$, $\phi$, $\phi'$ have 1-semiconnections, then we say that $\psi$ \emph{preserves the 1-semiconnections} if $D^{\phi'}_i - D^{\phi}_i + t_i^{\ast}\theta^{0,1} = 0$. 
        \item If $\mathcal{L}^1$, $\mathcal{L}^2$, $\phi$, $\phi'$ have 2-semiconnections, then we say that $\psi$ \emph{preserves the 2-semiconnections} if $D^{\phi'}_i - D^{\phi}_i + t_i^{\ast}\theta = 0$.
    \end{itemize}
\end{definition}

Note that 1-semiconnections and 2-semiconnections play roles similar to connective structures and curvings (Definition \ref{def:gerbes}). Furthermore, if $D_1$, $D_2$ are semiconnections on $\mathcal{L}_1$, $\mathcal{L}_2$ then there is a canonical semiconnection $D_1 \otimes D_2$ on $\mathcal{L}_1 \otimes \mathcal{L}_2$ and if $f: X \rightarrow Y$ is a holomorphic map and $(\mathcal{L},D) \rightarrow Y$ is a gerbe with semiconnection then $f^{\ast}\mathcal{L}$ carries a canonical semiconnection $f^{\ast}D$ (but $f$ needs to be holomorphic, otherwise $f^{\ast}$ may not preserve types of differential forms).

\begin{remark}\label{rk:dolbgerbes}
    A 1-semiconnection on a $T$-gerbe $\mathcal{L}$ can be equivalently defined as an equivalence class of connections with curvature $H$ of type $(3,0)+(2,1)+(1,2)$, up to addition of cocycles of $(1,0)$-forms and addition of $(2,0)+(1,1)$ forms. Then a 1-semiconnection on an isomorphism of gerbes with 1-semiconnections is an equivalence class of connections on the isomorphism whose curvature with respect to any choice of representing connections is of type $(2,0)+(1,1)$, up to addition of $(1,0)$-forms.

    Similarly, a 2-semiconnection on a $T$-gerbe $\mathcal{L}$ can be equivalently defined as an equivalence class of connections with curvature $H$ of type $(3,0)+(2,1)$, up to addition of $(2,0)$ forms. Then a 2-semiconnection on an isomorphism of gerbes with 2-semiconnections is a connection on the isomorphism (on the nose, without taking equivalence classes) whose curvature with respect to any choice of representing connections is of type $(2,0)$.
\end{remark}

The content of the following proposition is well-known in the literature 
\cite{Ald,BryHol,Chat,Upm} (sometimes in implicit terms).

\begin{proposition}\label{prop:dolbeaultgerbes}
The bicategory of holomorphic $T$-gerbes is equivalent to the bicategory of smooth $T$-gerbes with 1-semiconnections. The bicategory of holomorphic $T$-gerbes with holomorphic connective structure is equivalent to the bicategory of smooth $T$-gerbes with 2-semiconnections. Under this correpondence, compatible connections are precisely the ones that yield the given semiconnections as in Remark \ref{rk:dolbgerbes}.
\end{proposition}
\begin{proof}
    We give a brief sketch of the proof in the case of 1-semiconnections. If $\mathcal{L}$ is given by holomorphic data $\lambda_{ijk}$, then we define a 1-semiconnection by $D_{ij} = 0$, $D_i = 0$. Conversely, if $(\mathcal{L},D)$ is a smooth gerbe with 1-semiconnection given by $\lambda_{ijk}$, $D_{ij}$, $D_i$ then we choose $c_i \in \Omega^{0,1}(U_i,\mathfrak{t})$, $f_{ij} \in C^{\infty}(U_{ij},\mathfrak{t})$ with $\bar{\partial}c_i = D_i$ and $\bar{\partial}f_{ij} = c_i - c_j - D_{ij}$; then $\lambda_{ijk}exp(f_{ij})exp(f_{ik})^{-1}exp(f_{jk})$ is data for a holomorphic gerbe. Different choices of $c_i$, $f_{ij}$ yield canonically isomorphic holomorphic gerbes. These maps can be extended to an equivalence of bicategories in a similar way.
\end{proof}

We conclude by defining the fibrewise complexification of a gerbe.

\begin{definition}\label{def:complexificationgerbes}
    Let $T_{\mathbb R}$ be a compact abelian Lie group and let $j_T: T_{\mathbb R} \rightarrow T$ be its complexification (inducing an inclusion $dj_T: \mathfrak{t}_{\mathbb R} \rightarrow \mathfrak{t}$). For $(\mathcal{L},\nabla,B) \rightarrow X$ a $T_{\mathbb R}$-gerbe with connection described by $(\{U_i\},\{\lambda_{ijk}\},\{A_{ij}\},\{B_i\})$, its \emph{fibrewise complexification} is the smooth $T$-gerbe with connection $(\mathcal{L}^{\mathbb C},\nabla^{\mathbb C},B^{\mathbb C})$ described by $(\{U_i\},\{j_T(\lambda_{ijk})\},\{dj_T \circ A_{ij}\}, \{dj_T \circ B_i\})$.
\end{definition}

\end{document}